\documentclass[11pt,reqno]{amsart}
\usepackage[T1]{fontenc}
\usepackage[utf8]{inputenc}
\usepackage{lmodern}
\usepackage{amsmath,amssymb,amsthm,mathtools,mathrsfs}
\usepackage[margin=1.02in]{geometry}
\usepackage[expansion=false]{microtype}
\usepackage{enumitem}
\usepackage[hidelinks,unicode]{hyperref}
\allowdisplaybreaks[1]
\numberwithin{equation}{section}
\makeatletter
\renewcommand{\andify}{%
	\nxandlist{\unskip, }{\unskip{} \@@and~}{\unskip{} \@@and~}}
\renewcommand{\author@andify}{%
	\nxandlist{\unskip,\penalty-1 \space\ignorespaces}%
	{\unskip{} \@@and~}{\unskip{}\penalty-2 \space \@@and~}}
\makeatother
\newtheorem{theorem}{Theorem}[section]
\newtheorem{proposition}[theorem]{Proposition}
\newtheorem{lemma}[theorem]{Lemma}
\newtheorem{corollary}[theorem]{Corollary}
\theoremstyle{definition}
\newtheorem{definition}[theorem]{Definition}
\theoremstyle{remark}
\newtheorem{remark}[theorem]{Remark}
\newcommand{\C}{\mathbb C}
\newcommand{\Pone}{\mathbb P^1}
\newcommand{\dbar}{\bar\partial}
\newcommand{\id}{\operatorname{Id}}
\DeclareMathOperator{\Hom}{Hom}
\newcommand{\norm}[1]{\left\lVert#1\right\rVert}
\newcommand{\doi}[1]{\href{https://doi.org/#1}{\nolinkurl{doi:#1}}}
\title[Oka-1 approximation with rationally connected fibers]{Analytic and
	Algebraic Oka-1 Approximation for Smooth Projective Morphisms with
	Rationally Connected Fibers}
\author{Yun-Heng Du}
\address{Academy of Mathematics and Systems Science,
	Chinese Academy of Sciences, Beijing 100190, China}
\email{duyunheng@amss.ac.cn}

\author{Bin Guo}
\address{Academy of Mathematics and Systems Science,
	Chinese Academy of Sciences, Beijing 100190, China}
\email{guobin181@mails.ucas.ac.cn}

\author{Song-Yan Xie}
\address{State Key Laboratory of Mathematical Sciences,
	Academy of Mathematics and Systems Science, Chinese Academy of Sciences,
	Beijing 100190, China; School of Mathematical Sciences, University of
	Chinese Academy of Sciences, Beijing 100049, China}
\email{xiesongyan@amss.ac.cn}
\date{}


\hypersetup{
	pdftitle={Analytic and Algebraic Oka-1 Approximation for Smooth Projective Morphisms with Rationally Connected Fibers},
	pdfauthor={Yun-Heng Du, Bin Guo, Song-Yan Xie}
}

\begin{document}
	\begin{abstract}
		Let $\pi:Z\to Y$ be a smooth projective morphism of complex manifolds
		with connected rationally connected fibers. We prove holomorphic
		approximation on arbitrary compact sets and finite-jet interpolation on
		arbitrary closed discrete sets for continuous liftings defined on open Riemann
		surfaces and holomorphic near those sets.
		For smooth projective morphisms of smooth complex algebraic varieties
		and algebraic base maps from smooth affine curves, the approximating
		liftings can be chosen algebraic, with interpolation on any finite set.
		In both cases the resulting lifting is homotopic to the initial one
		through continuous liftings of the fixed base map.
		For connected smooth projective complex manifolds, this gives the
		equivalence between the algebraic Oka-1 property and rational
		connectedness.  Every rationally connected smooth projective complex
		manifold is also Oka-1.
	\end{abstract}
	\maketitle
	
	\section{Introduction}\label{sec:introduction}
	We study approximation and finite-jet interpolation for liftings of maps
	from complex curves through smooth projective morphisms.  Rational
	connectedness of the fibers suffices for a holomorphic endpoint on an open
	Riemann surface and, when the source and morphisms are algebraic, for an
	algebraic endpoint on a smooth affine curve.  The resulting lifting is homotopic to the given one through
	continuous liftings of the fixed base map.  The homotopy is obtained
	by a topological argument on the open source.  The algebraic construction
	also uses a cutoff homotopy on the projective completion of that source.
	
	Alarc\'on--Forstneri\v c conjectured that every rationally connected smooth
	projective manifold is Oka-1 \cite[Conjecture~9.1]{AF25}.
	Forstneri\v c--L\'arusson introduced the algebraic Oka-1 property and proved
	that rational connectedness is necessary in the smooth projective case
	\cite[Definition~1.5 and Proposition~1.7]{FL25}.  We prove the relative
	analytic and algebraic approximation statements below.  Specializing Theorems~\ref{thm:rel-main} and~\ref{thm:relative-algebraic-oka1}
	to $Y=\{\mathrm{pt}\}$ settles the analytic conjecture and proves
	that rational connectedness also suffices for the algebraic Oka-1 property.
	
	Forstneri\v c--L\'arusson proved that compact rational manifolds and
	algebraically elliptic projective manifolds are algebraically Oka-1
	\cite[Corollary~3.5 and Theorem~4.1]{FL25}. The algebraic Oka-1 property
	implies the Oka-1 property \cite[Proposition~1.9]{FL25}.
	Our analytic argument isolates the first-jet sphere families supplied by
	rational connectedness; the algebraic argument uses the same section
	equation after a separate construction of a bundle with vanishing first
	cohomology.
	
	\subsection{Conventions and lifting properties}
	Throughout this paper, open Riemann surfaces are assumed to be connected.
	By a \emph{compact bordered Riemann surface} we mean the closure of a
	relatively compact smoothly bounded domain in an open Riemann surface. A map is holomorphic near a compact set if it is
	holomorphic on an open neighborhood of that set.  We call a compact subset $K$ of an open Riemann surface $R$ \emph{Runge} if $R\setminus K$
	has no relatively compact connected component.
	The term \emph{closed discrete} means that $A\subset R$ is closed and
	every point of $R$ has a neighborhood meeting $A$ in finitely many points.
	
	If two holomorphic maps $f$ and $F$ are defined near $a\in R$ and have the
	same value at $a$, choose a source coordinate $z$ with $z(a)=0$ and a
	holomorphic target chart $\phi$ at $f(a)=F(a)$.  We write $j_a^kF=j_a^kf$ if
	\[
	\left.\frac{d^j}{dz^j}(\phi\circ F-\phi\circ f)\right|_{z=0}=0
	\qquad(0\leq j\leq k).
	\]
	Equivalently, each component of $\phi\circ F-\phi\circ f$ is divisible
	by $z^{k+1}$.  The jet-equality condition is independent of the coordinates.  A jet of order $k$ is equivalently
	a restriction to the infinitesimal neighborhood of length $k+1$; in particular, a
	first jet corresponds to the divisor $2[a]$.  For a closed discrete set $A$, an \emph{order function} is a map
	$q:A\to\mathbb N$, where $\mathbb N$ denotes the set of nonnegative integers.  Fix a smooth Hermitian metric
	on every complex manifold used as a target and denote the associated
	distance by $d$ with a subscript when needed.  All algebraic varieties
	are over $\C$, unless another ground field is specified.  Distances and holomorphic jets of algebraic maps are
	understood on their analytifications.  Unless otherwise stated, a smooth algebraic curve is
	connected, and hence integral.
	
	Rational connectedness of a connected smooth projective manifold $X$
	means that there is a nonempty Zariski-open subset $U\subset X$ such that
	any $x,y\in U$ lie on the image of some morphism $h:\Pone\to X$.
	
	\begin{definition}
		\label{def:rel-smooth-projective}
		A holomorphic map $\pi:Z\to Y$ between complex manifolds is called
		\emph{smooth projective} if $\pi$ is a holomorphic submersion and every point
		of $Y$ has an open neighborhood $U$ for which there exist an integer $N$ and
		a closed holomorphic embedding
		\[
		Z|_U\hookrightarrow U\times\mathbb P^N
		\]
		over $U$, where $Z|_U=\pi^{-1}(U)$.  Thus projectivity in this analytic definition is local on the
		base; no global relatively ample line bundle is chosen.  In particular,
		$\pi$ is proper and every nonempty fiber $Z_y$ is a smooth projective
		manifold.  Throughout, a rationally connected fiber is understood to be
		nonempty, so the morphisms in the main theorems are surjective.
	\end{definition}
	
	For a continuous map $q:E\to B$, the \emph{Serre fibration} condition
	can be expressed as follows for every CW complex $T$.  Given continuous maps
	$a:T\to E$ and $H:T\times[0,1]\to B$ with $H(t,0)=q(a(t))$, there is a
	continuous map $\widetilde H:T\times[0,1]\to E$ satisfying
	\[
	q\circ\widetilde H=H,\qquad \widetilde H(t,0)=a(t).
	\]
	
	The following lifting property is the Oka-1 map condition
	of~\cite[Definition~7.7]{AF25}.
	
	\begin{definition}
		\label{def:rel-oka1-map}
		We call a holomorphic map $\pi:Z\to Y$ an \emph{Oka-1 map} if it is a
		Serre fibration with the following lifting property.  Let $R$ be an open
		Riemann surface, let $K\subset R$ be compact and Runge, let $A\subset K$ be
		finite, and let $g:R\to Y$ be holomorphic.  Suppose that $f:R\to Z$ is a
		continuous lifting of $g$, meaning $\pi\circ f=g$, and is holomorphic
		on a neighborhood of $K$.
		Given $\epsilon>0$ and an integer $k_a\geq0$ for every $a\in A$, there
		are a holomorphic lifting $F:R\to Z$ of $g$ and a homotopy from $f$ to $F$
		through continuous liftings $f_t$ of $g$, with $t\in[0,1]$,
		$\pi\circ f_t=g$, $f_0=f$, and $f_1=F$, such that
		\[
		\sup_{z\in K}d_Z(F(z),f(z))<\epsilon,
		\qquad j_a^{k_a}F=j_a^{k_a}f\quad(a\in A).
		\]
	\end{definition}
	
	For a point base, this is the Oka-1 property of the target manifold
	\cite[Proposition~2.7 and Definition~7.7]{AF25}.
	
	Definition~7.7 of~\cite{AF25} requires interpolation only at finitely many
	points of $K$.  That definition does not require the intermediate maps in the homotopy
	to preserve the prescribed jets.  Theorem~\ref{thm:rel-analytic-criterion} gives a stronger endpoint
	interpolation statement on arbitrary closed discrete subsets, together with
	approximation on arbitrary compact sets.
	
	\subsection{The analytic criterion and rationally connected fibers}
	
	Fix an affine coordinate $w$ centered at $0\in\Pone$. The following
	property is the geometric input to the analytic construction.
	\begin{definition}
		\label{def:rel-first-jet-spheres}
		For a holomorphic submersion $p:X\to S$ of complex manifolds, the \emph{relative first-jet sphere-family property} is the following requirement: for every $(x,b)\in T_{X/S}$, there are an open neighborhood $\Omega\subset T_{X/S}$ of $(x,b)$ and a jointly holomorphic map
		\[
		G:\Pone\times\Omega\longrightarrow X
		\]
		such that, for every $(x',b')\in\Omega$,
		\begin{equation}\label{eq:rel-sphere-identities}
			p(G(w;x',b'))=p(x'),\qquad
			G(0;x',b')=x',\qquad \partial_wG(0;x',b')=b'.
		\end{equation}
		Here $T_{X/S}=\ker dp$ is the vertical tangent bundle.  The requirement includes $b=0$; a realizing sphere need not be constant when its derivative at $0$ vanishes.
	\end{definition}
	
	\begin{theorem}
		\label{thm:rel-analytic-criterion}
		Let $\pi:Z\to Y$ be a surjective proper holomorphic submersion of complex
		manifolds with connected and simply connected fibers.  Suppose that $\pi$ has
		the relative first-jet sphere-family property.  Then $\pi$ is an Oka-1 map.  If
		$R$ is an open Riemann surface, $K\subset R$ is compact, $A\subset R$ is closed
		and discrete, and $g:R\to Y$ is holomorphic, and if a continuous lifting
		$f:R\to Z$ of $g$ is holomorphic near $K\cup A$, then for every $\epsilon>0$
		and every order function $q:A\to\mathbb N$, the lifting $f$ is homotopic
		through continuous liftings of $g$ to a holomorphic lifting $F$ satisfying
		\[
		\sup_K d_Z(F,f)<\epsilon,
		\qquad j_a^{q(a)}F=j_a^{q(a)}f\quad(a\in A).
		\]
	\end{theorem}
	
	For smooth projective morphisms, rational connectedness supplies the
	sphere families and simple connectivity of the fibers.
	
	\begin{theorem}\label{thm:rel-main}
		Let $\pi:Z\to Y$ be a smooth projective morphism whose fibers are connected
		and rationally connected. Then all the conclusions of
		Theorem~\ref{thm:rel-analytic-criterion} hold. In particular, $\pi$ is an
		Oka-1 map, with approximation on arbitrary compact sets and finite-jet
		interpolation on arbitrary closed discrete subsets of the source.
	\end{theorem}
	
	Approximation and interpolation are required at the holomorphic
	endpoint; the intermediate liftings need only be continuous. The compact set $K$ need not be Runge, and neither
	$A\subset K$ nor $K\subset A$ is required.
	
	\subsection{Algebraic approximation}
	For algebraic source data we obtain an algebraic lifting.
	We formulate a relative version of the algebraic Oka-1 property
	introduced by Forstneri\v c and L\'arusson
	in~\cite[Definition~1.5]{FL25}, using the lifting formulation
	of Oka-1 maps in~\cite[Definition~7.7]{AF25}.
	
	\begin{definition}
		\label{def:algebraic-oka1-map}
		An \emph{algebraic Oka-1 map}, or \emph{aOka-1 map}, is a morphism
		$\pi:Z\to Y$ of smooth complex algebraic varieties for which
		$\pi^{\mathrm{an}}$ is a Serre fibration and the following lifting property
		holds.  Let $R$ be a smooth connected affine algebraic curve, let
		$K\subset R^{\mathrm{an}}$ be compact and Runge, let $A\subset K$ be finite,
		let $g:R\to Y$ be algebraic, and let
		$f:R^{\mathrm{an}}\to Z^{\mathrm{an}}$ be a continuous lifting of
		$g^{\mathrm{an}}$ which is holomorphic near $K$.  Given $\epsilon>0$ and
		integers $k_a\geq0$, there are an algebraic lifting $F:R\to Z$ and a
		homotopy from $f$ to $F^{\mathrm{an}}$ through continuous liftings of
		$g^{\mathrm{an}}$ such that
		\[
		\sup_Kd_Z(F,f)<\epsilon,
		\qquad
		j_a^{k_a}F=j_a^{k_a}f\quad(a\in A).
		\]
		When $Y$ is a point, the preceding property is the algebraic Oka-1
		property of \cite[Definition~1.5]{FL25}. The common order in that
		definition is a positive integer; for nonempty finite $A$ one may take
		$k=\max\{1,\max_{a\in A}k_a\}$, and conversely, point-dependent
		interpolation includes the choice $k_a=k$ for every $a$.
		For $A=\varnothing$ there is no interpolation condition.
	\end{definition}
	
	The same rational-connectedness hypothesis gives an algebraic lifting when the source and the base map are algebraic.
	
	\begin{theorem}
		\label{thm:relative-algebraic-oka1}
		Let
		\[
		\pi:Z\longrightarrow Y
		\]
		be a smooth projective morphism of smooth complex algebraic varieties whose
		fibers are connected and rationally connected.  Then $\pi$ is an aOka-1
		map.  More precisely, let $R$ be a smooth connected affine
		complex algebraic curve, let $K\subset R^{\mathrm{an}}$ be compact, let
		$A\subset R$ be finite, and let $g:R\to Y$ be algebraic.  Suppose that
		$f:R^{\mathrm{an}}\to Z^{\mathrm{an}}$ is a continuous lifting of
		$g^{\mathrm{an}}$ which is holomorphic on a neighborhood of $K\cup A$.
		Given $\epsilon>0$ and an integer $k_a\geq0$ for each $a\in A$,
		there is an algebraic lifting
		$F:R\to Z$ such that
		\[
		\sup_K d_Z(F,f)<\epsilon,
		\qquad
		j_a^{k_a}F=j_a^{k_a}f\quad(a\in A).
		\]
		Moreover, $F^{\mathrm{an}}$ and $f$ are homotopic through continuous liftings
		of $g^{\mathrm{an}}$.
	\end{theorem}
	
	The following characterization follows by taking a point as base in
	Theorem~\ref{thm:relative-algebraic-oka1} and using the necessary condition
	in~\cite[Proposition~1.7]{FL25}.
	
	\begin{corollary}
		\label{cor:projective-aoka-characterization}
		A connected smooth projective complex manifold $X$ is algebraically Oka-1
		if and only if $X$ is rationally connected. Every rationally connected smooth projective complex manifold is also
		an Oka-1 manifold.
	\end{corollary}
	
	\begin{proof}
		If $X$ is algebraically Oka-1, \cite[Proposition~1.7]{FL25} implies that
		$X$ is rationally connected.  If $X$ is rationally connected, apply
		Theorem~\ref{thm:relative-algebraic-oka1} to $X\to\{\mathrm{pt}\}$.
		Definition~\ref{def:algebraic-oka1-map} with a point as base gives the
		algebraic Oka-1 property of $X$. Applying Theorem~\ref{thm:rel-main} to
		$X\to\{\mathrm{pt}\}$ gives the Oka-1 assertion.
	\end{proof}
	
	\subsection{Relation to previous work and proof strategy}
	Benoist--Wittenberg prove algebraic approximation for rationally simply
	connected projective varieties and tight approximation for such varieties
	over function fields of complex curves
	\cite[Theorems~1.2 and~1.4]{BW25}.  The relative formulation of Benoist--Wittenberg starts with
	a regular projective model over a projective curve and a holomorphic
	section on an arbitrary open subset; that formulation includes approximation on compact
	subsets and finite jets, allowing bad fibers of the model
	\cite[Definitions~3.1 and~3.3 and Remarks~3.4]{BW25}.
	Theorem~\ref{thm:relative-algebraic-oka1} assumes rational connectedness,
	with a smooth family over the affine source and a continuous global
	lifting holomorphic near the approximation and interpolation sets.
	In the projective model used in the proof of
	Theorem~\ref{thm:relative-algebraic-oka1}, bad fibers can occur only outside
	the affine source; the theorem does not assert the full bad-reduction
	formulation of tight approximation.
	
	Alarc\'on--Forstneri\v c~\cite[\S9]{AF25} observe that Gournay's Runge
	theorem, taken at face value, implies the approximation part of their
	conjecture, but report that they were unable to understand its proof.
	Gournay describes two relevant difficulties: dependence of the linearized
	equation on derivatives of the grafted map and control of the target chart
	used for matching~\cite[\S\S1.3 and~3.1]{Gou12}.
	Lemma~\ref{lem:rel-normal-form} fixes the principal Dolbeault operator,
	and Lemma~\ref{lem:rel-uniform-matching} supplies uniform fiber-derivative
	bounds and a positive target-chart margin.  The replacements change only
	the zero-order term.  Both the residual and the difference between the
	linearization and the fixed Dolbeault operator vanish outside measurable
	sets whose areas tend to zero.  A fixed right inverse for the principal
	part then yields uniformly bounded right inverses for the full
	linearizations; the operator calculation is given in
	Lemma~\ref{lem:fixed-bundle-solution}.
	
	Lemma~\ref{lem:fixed-bundle-solution} solves the semilinear equation
	\[
	\mathscr F_r(v)=\bar\partial_Ev+A_r(z,v)
	\]
	on a fixed holomorphic bundle with $H^1(\Sigma,E)=0$, once both
	$\mathscr F_r(0)$ and $d\mathscr F_r(0)-\bar\partial_E$ are small in fixed
	Sobolev norms.  The geometric construction supplies these two smallness
	conditions by replacing sections on disjoint rational disks and matching
	on thin annuli.
	
	In the analytic construction a positive auxiliary bundle removes the
	Dolbeault obstruction outside the bordered reconstruction region.  Bordered
	approximation is followed by an exhaustion retaining the prescribed jets;
	connectedness and simple connectivity of the fibers give extension and
	homotopy of sections over the open source.
	
	The algebraic proof works on a projective completion of the source.  An
	algebraic section with vanishing twisted first cohomology provides the
	fixed bundle.  A smooth bundle isomorphism, holomorphic near the prescribed
	data and near the omitted fibers, transfers the section equation to this
	bundle.  The constant and linear gauge defects are then controlled before
	Lemma~\ref{lem:fixed-bundle-solution} is applied; global reconstruction and
	Chow's theorem give the algebraic endpoint.
	
	Section~\ref{sec:preliminaries} contains the analytic and topological
	tools.  Sphere families are constructed in
	Section~\ref{sec:sphere-families}.  The section equation and the solution
	lemma are treated in Section~\ref{sec:vertical-preparation}; the analytic
	criterion follows in Section~\ref{sec:analytic-approximation}.  Algebraic
	preparation is in Section~\ref{sec:relative-algebraic-oka1}, and compact
	approximation and algebraic liftings are in
	Section~\ref{sec:compact-algebraic-proof}.
	
	\section{Analytic and topological tools}\label{sec:preliminaries}
	
	\subsection{Dolbeault estimates and holomorphic frames}\label{subsec:analytic-tools}
	
	Fix smooth metrics and an area form whenever Sobolev norms are used.
	For a complex vector bundle $E$ on a Riemann surface $\Sigma$, write
	$L^p_{0,1}(\Sigma,E)=L^p(\Sigma,\Lambda^{0,1}T^*\Sigma\otimes E)$.
	When $\Sigma$ is compact, we use the Sobolev embeddings
	\[
	W^{1,4}\hookrightarrow C^{0,1/2},\qquad
	W^{2,4}\hookrightarrow C^{1,1/2}
	\]
	for sections of $E$ \cite[Theorem~B.1.11]{MS12}. In particular, for $v\in W^{1,4}(\Sigma,E)$
	and a bounded coefficient $a$ vanishing outside a measurable set $S$,
	the estimates used below are
	\begin{equation}\label{eq:tools-small-area}
		\|v\|_{C^0}\leq C_S\|v\|_{W^{1,4}},\qquad
		\|a\|_{L^4}\leq\|a\|_{L^\infty}\operatorname{area}(S)^{1/4}.
	\end{equation}
	Here $C_S$ is the norm of the fixed Sobolev embedding, determined by the
	surface, bundle, metrics, and area form; the subscript stands for Sobolev.
	
	For $R>0$, the Cauchy--Green operator on $D_R=\{z\in\C:|z|<R\}$ is
	\begin{equation}\label{eq:tools-cauchy-green}
		T_Rf(z)=\frac1\pi\int_{D_R}\frac{f(\zeta)}{z-\zeta}\,dA(\zeta),
		\qquad \partial_{\bar z}T_Rf=f.
	\end{equation}
	Here $dA$ is Euclidean area measure, and the differential identity holds
	in the sense of distributions. For $1<p<\infty$, the operator maps
	$L^p(D_R)$ to $W^{1,p}(D_R)$; for bounded $f$, the same integral is continuous
	on $\overline{D_R}$. The corresponding estimates are
	\begin{equation}\label{eq:tools-cauchy-green-estimates}
		\begin{aligned}
			\|T_Rf\|_{W^{1,p}(D_R)}&\leq C_{p,R}\|f\|_{L^p(D_R)},\\
			\|T_Rf\|_{C^0(\overline{D_R})}&\leq CR\|f\|_{L^\infty(D_R)}.
		\end{aligned}
	\end{equation}
	The formula is given in \cite[Lemma~13.1]{For81}; the $L^p$ estimate
	follows from the Calder\'on--Zygmund estimate
	\cite[Theorem~B.2.7]{MS12}. We also use the corresponding interior
	higher-order estimates and smooth regularity.
	
	For a holomorphic bundle $E$ on a compact Riemann surface $\Sigma$,
	Dolbeault Fredholm theory identifies the kernel and cokernel of
	\[
	\bar\partial_E:W^{k+1,p}(\Sigma,E)\longrightarrow W^{k,p}_{0,1}(\Sigma,E)
	\]
	with $H^0(\Sigma,E)$ and $H^1(\Sigma,E)$, for $k\geq0$ and
	$1<p<\infty$. Here $W^{k,p}_{0,1}$ denotes Sobolev sections of
	$\Lambda^{0,1}T^*\Sigma\otimes E$. The range is closed, and
	\begin{equation}\label{eq:tools-dolbeault-estimate}
		\|u\|_{W^{k+1,p}}\leq C\bigl(\|\bar\partial_Eu\|_{W^{k,p}}
		+\|u\|_{L^p}\bigr).
	\end{equation}
	We also use the local version of this estimate and smoothness of weak
	solutions with smooth right-hand side
	\cite[Theorem~C.1.10, Lemma~C.2.1, and Theorem~C.2.3]{MS12}.
	Thus $H^1(\Sigma,E)=0$ gives a bounded right inverse of $\bar\partial_E$.
	Serre duality identifies $H^1(\Sigma,E)^*$ with
	$H^0(\Sigma,K_\Sigma\otimes E^*)$, where $K_\Sigma$ is the canonical
	bundle \cite[Section~C.1 and Remark~C.1.12]{MS12}.
	
	A complex-linear Cauchy--Riemann operator on a smooth bundle over a
	Riemann surface is equivalently a holomorphic structure. We retain the
	local argument to fix the frame convention used for the section equation.
	\begin{lemma}
		\label{lem:frames}
		Let $E$ be a smooth complex vector bundle on a Riemann surface $\Sigma$, and
		let
		\[
		D:C^\infty(\Sigma,E)\longrightarrow
		C^\infty(\Sigma,\Lambda^{0,1}T^*\Sigma\otimes E)
		\]
		be a complex-linear first-order operator satisfying
		\[
		D(fu)=\dbar f\otimes u+fDu.
		\]
		Then $D$ defines a holomorphic structure on $E$: there are local smooth
		frames in which $D=\dbar$, and the transition functions of these frames
		are holomorphic.
	\end{lemma}
	\begin{proof}
		Put $n=\operatorname{rank}E$. In a smooth frame on a disk, write
		$D=\dbar+A(z)\,d\bar z$, where $A$ is a
		smooth matrix. Use the operator norm on matrices. A new frame represented
		by an invertible matrix $T$ is
		$D$-holomorphic exactly when
		$\partial_{\bar z}T=-AT$.  Let $T_R$ act entrywise on matrices.  On a sufficiently small disk, consider
		\[
		T=\id-T_R(AT).
		\]
		The $C^0$ estimate in \eqref{eq:tools-cauchy-green-estimates} makes the right-hand
		side a contraction on the Banach space
		$C^0(\overline{D_R},\operatorname{Mat}_{n\times n}(\C))$ when $CR\norm A_\infty<1/4$.
		Put $\kappa=CR\|A\|_\infty<1/4$.  The fixed point $T$ satisfies
		\[
		\|T\|_\infty\leq1+\kappa\|T\|_\infty,\qquad
		\|T-\id\|_\infty\leq\frac{\kappa}{1-\kappa}<\frac12.
		\]
		Thus every $T(z)$ is invertible.  Applying
		the $L^p$ estimate in \eqref{eq:tools-cauchy-green-estimates} to $AT$ first gives
		$T\in W^{1,p}$ and then, from the equation and interior regularity,
		$T\in C^\infty$.  If $T_1,T_2$ are two such frames, applying $D$ to
		$T_1=T_2(T_2^{-1}T_1)$ shows that
		$\dbar(T_2^{-1}T_1)=0$.  The transition matrix $T_2^{-1}T_1$ is therefore holomorphic.
	\end{proof}
	
	\subsection{Surfaces, bundles, and sections}
	\label{subsec:surface-tools}\label{subsec:bundle-tools}
	
	An open Riemann surface has an exhaustion by connected compact bordered
	surfaces $C_j\Subset\operatorname{int}C_{j+1}$, whose boundaries can be
	chosen to avoid a prescribed closed discrete set and with $C_1$
	containing a prescribed compact set in its interior. One obtains such an
	exhaustion by taking regular levels of compactly supported smooth
	functions, joining the required compact pieces by finitely many arcs,
	and avoiding the finitely many forbidden levels at each stage.
	Smooth triangulations are available by \cite[Theorem~10.6]{Mun66}.
	
	We use the isothermal-coordinate theorem
	\cite[Theorem~3.11.1]{Jost06}: an oriented smooth surface with a smooth
	positive-definite metric has a unique complex structure compatible with
	its conformal class. We also use the fact that an open Riemann surface
	deformation retracts onto a countable locally finite graph
	\cite[Lemma~2.1]{Wh61}. A proof in terms of a triangulation is given in
	\cite[Theorem~0.1 and its proof]{PutmanSpines}.
	The isothermal-coordinate theorem gives the compact-neighborhood construction below;
	the graph retraction is used for bundle trivializations and section homotopies.
	\begin{lemma}\label{lem:compact-neighborhood}
		Let $R$ be an open Riemann surface and let $C\subset R$ be compact.
		There exist an open neighborhood $W\subset R$ of $C$, a compact Riemann
		surface $\Sigma$, and a holomorphic embedding $W\hookrightarrow\Sigma$.
		The embedding identifies $W$ with an open subset of $\Sigma$ and
		preserves the complex structure that $W$ inherits from $R$.
	\end{lemma}
	\begin{proof}
		Choose connected bordered smooth surfaces $D_0,D_1\subset R$ with
		$C\subset\operatorname{int}D_0$ and
		$D_0\subset\operatorname{int}D_1$.  Form the smooth double $M$ of $D_1$:
		take a second copy with the opposite orientation, identify corresponding
		boundary points, and use a boundary collar to give the identification a
		smooth structure.  In a collar, the signed distance normal to the
		boundary is the transverse coordinate on the double.  Thus $M$ is a
		compact oriented smooth surface without boundary, containing
		$\operatorname{int}D_1$ as an open subset.
		
		Choose a smooth Hermitian metric for the original complex structure on
		$R$.  Let $g_0$ be its restriction near $D_0$, and choose a smooth
		positive-definite metric $g_1$ on $M$.  Take $\chi\in C^\infty(M,[0,1])$ equal to one near $D_0$ and
		supported where $g_0$ is defined inside $\operatorname{int}D_1$.  Set
		\[
		g=\chi g_0+(1-\chi)g_1
		\]
		on that open set and $g=g_1$ outside $\operatorname{supp}\chi$.
		The formulas agree where $\chi=0$; positivity of $g_0$ and $g_1$
		makes $g$ positive definite.  The isothermal-coordinate theorem \cite[Theorem~3.11.1]{Jost06} gives a complex
		structure on $M$.  Near $D_0$, the conformal class of $g$ and the orientation of $M$ are
		those of the original Riemann surface, so the two complex structures
		agree.  Write $\Sigma$ for $M$ with the complex structure determined by $g$.
		Any sufficiently small neighborhood $W$ of $C$ contained in
		$\operatorname{int}D_0$ has the required embedding.
	\end{proof}
	
	Ehresmann's theorem \cite{Ehr51} makes a proper surjective smooth
	submersion a locally trivial smooth fiber bundle. Over a paracompact
	base it is a Hurewicz fibration, hence a Serre fibration
	\cite[p.~379]{Hatcher02}. Pullback by a smooth map
	is a smooth fiber bundle; pullback by a continuous map is a topological
	fiber bundle.
	
	We use the following relative smoothing statement for sections. Let
	$p:E\to M$ be a smooth fiber bundle, let $C\subset M$ be closed, and let
	$s:M\to E$ be a continuous section that is smooth near $C$. Fix a metric $d_E$ inducing
	the topology of $E$. Given a neighborhood $U$
	of $C$ and a positive continuous function $\delta$ on $M$, there is a
	smooth section $\widetilde s$ equal to $s$ near $C$, with
	$d_E(\widetilde s(x),s(x))<\delta(x)$, homotopic to $s$ through sections
	by a homotopy fixed near $C$. The statement also holds for manifolds with
	boundary or corners. Apply the finite-dimensional section case of
	Wockel's relative Steenrod approximation theorem~\cite[Theorem~11]{Wockel09}
	on each connected component of $M$. To obtain the stated form,
	choose $C\subset V\subset\overline V\subset U$ with $\overline V$ in the
	region where $s$ is smooth. In Wockel's theorem, take $L=M$, allow
	modifications on $M\setminus\overline V$, and use the neighborhood
	\[
	\{e\in E:d_E(e,s(p(e)))<\delta(p(e))\}.
	\]
	
	An open Riemann surface is Stein by the Behnke--Stein theorem
	\cite[Corollary~26.8]{For81}. Grauert's Oka principle for vector bundles
	\cite[Satz~I]{Grauert58} implies that a topologically trivial complex vector
	bundle with a holomorphic structure on a Stein manifold is holomorphically
	trivial. In particular this applies to every holomorphic vector bundle
	on an open Riemann surface: a graph retraction reduces its topological
	trivialization to extending frames along edges, using path connectedness
	of $\operatorname{GL}_n(\C)$, where $n$ is the bundle rank.
	
	The next two lemmas are the topological steps in exhaustion and in the
	passage from a holomorphic endpoint to a homotopy through liftings.
	\begin{lemma}
		\label{lem:rel-top-extension}
		Let $p:E\to M$ be a smooth locally trivial fiber bundle over a smooth
		real surface without boundary, with connected and simply connected
		fibers.  Suppose that $M_0\subset M_1\subset M$ are compact smooth
		subsurfaces with boundary and $M_0\subset\operatorname{int}M_1$.
		Fix pairwise disjoint closed coordinate disks
		$D_1,\ldots,D_s\subset\operatorname{int}(M_1\setminus M_0)$.
		Given a smooth section $s_*$ on a neighborhood of
		\[
		M_0\cup D_1\cup\cdots\cup D_s,
		\]
		there is a smooth section on a neighborhood of $M_1$ that agrees with
		$s_*$ on a possibly smaller neighborhood of this union.
	\end{lemma}
	
	\begin{proof}
		Let $s_*:O\to E$ be the prescribed section, where $O$ is a neighborhood
		of $M_0\cup D_1\cup\cdots\cup D_s$.
		Choose a compact smoothly bounded, possibly disconnected surface $N$ such
		that
		\[
		M_0\cup D_1\cup\cdots\cup D_s
		\subset\operatorname{int}N\subset N\subset O\cap\operatorname{int}M_1.
		\]
		Also choose a compact smoothly bounded neighborhood $M_2$ of $M_1$ in $M$,
		with $M_1\subset\operatorname{int}M_2$.  It is enough to construct the
		section on $M_2$.
		
		Triangulate $M_2$ so that $N$ is a subcomplex.  First triangulate
		the compact one-manifolds $\partial N$ and $\partial M_2$.
		The relative triangulation theorem
		in~\cite[Theorem~10.6]{Mun66} extends the boundary triangulations first
		over $N$ and then over the closure of $M_2\setminus N$; the latter is a
		compact smooth surface with boundary $\partial N\sqcup\partial M_2$.
		The two triangulations
		agree on $\partial N$ and hence glue to a finite smooth triangulation of
		$M_2$ in which $N$ is a subcomplex.  Choose a finite cover of $M_2$ by bundle trivialization domains.
		After sufficiently many barycentric subdivisions, the mesh is smaller
		than a Lebesgue number for this cover.  Each closed simplex then lies
		in one trivialization domain, which is an open neighborhood of that
		simplex.
		Put $\widehat s=s_*|_N$ and extend $\widehat s$ over the
		remaining simplices in increasing dimension.
		
		For each vertex $a$ outside $N$, choose $\widehat s(a)\in p^{-1}(a)$.
		For an edge $e=[a,b]$ not contained in $N$, the endpoint values are now
		prescribed.  Choose a trivialization
		$\tau_e:p^{-1}(e)\to e\times F_e$, and let $\alpha_e,\beta_e\in F_e$
		be the fiber coordinates of $\widehat s(a),\widehat s(b)$.  Connectedness of the fiber
		manifold $F_e$ gives a path from $\alpha_e$ to $\beta_e$.  Parametrize
		$e$ by $[0,1]$ and apply $\tau_e^{-1}$ to the graph of that path to
		extend $\widehat s$ over $e$ without changing the prescribed
		values at $a$ and $b$.
		
		For a two-simplex $\Delta$ not contained in $N$, the section is now
		defined on $\partial\Delta$.  A trivialization
		$\tau_\Delta:p^{-1}(\Delta)\to\Delta\times F_\Delta$ identifies the
		boundary section with a map $b_\Delta:\partial\Delta\to F_\Delta$.
		The loop $b_\Delta$ is null-homotopic because
		$\pi_1(F_\Delta)=0$.  A null-homotopy descends to the cone on
		$\partial\Delta$, which is a closed disk, and hence gives an extension
		$\widetilde b_\Delta:\Delta\to F_\Delta$ with
		$\widetilde b_\Delta|_{\partial\Delta}=b_\Delta$.  Set
		$\widehat s(x)=\tau_\Delta^{-1}(x,\widetilde b_\Delta(x))$
		for $x\in\Delta$.
		The extensions agree on common faces, so finite gluing gives a continuous
		section on $M_2$.  There are no simplices of higher dimension.
		
		The resulting $\widehat s$ is continuous on $M_2$ and agrees with
		$s_*$ on $N$.  It is therefore smooth on $\operatorname{int}N$, a neighborhood of
		$M_0\cup D_1\cup\cdots\cup D_s$.  Apply the relative smoothing statement above on $M_2$, relative to a
		smaller compact smoothly bounded neighborhood
		\[
		M_0\cup D_1\cup\cdots\cup D_s\subset N'\Subset\operatorname{int}N.
		\]
		Use $C=N'$ and $U=\operatorname{int}N$ in that statement (with any positive
		continuous tolerance $\delta$).  It gives a smooth section on $M_2$ which is unchanged
		near $N'$.  Since
		$M_1\subset\operatorname{int}M_2$, its restriction is the required smooth
		section on a neighborhood of $M_1$.
	\end{proof}
	
	\begin{lemma}
		\label{lem:rel-section-homotopy}
		Let $p:E\to R$ be a locally trivial fiber bundle over an open Riemann
		surface.  If the fibers of $p$ are connected and simply connected, any two
		continuous sections of $p$ are homotopic through sections.
	\end{lemma}
	
	\begin{proof}
		Let $s_0,s_1:R\to E$ be the sections and put $I=[0,1]$.  By
		the graph-retraction statement above, there is a deformation
		retraction $d_t:R\to R$, $t\in[0,1]$, onto a countable locally finite graph
		$\Gamma\subset R$.  Write $i:\Gamma\hookrightarrow R$ and $r:R\to\Gamma$
		for the inclusion and retraction, so $d_0=\operatorname{id}_R$ and
		$d_1=i\circ r$.
		
		First construct a homotopy through sections over $\Gamma$.  Subdivide the edges of $\Gamma$ so that each closed edge lies in a bundle chart.  At each vertex $a$,
		choose a path in $E_a=p^{-1}(a)$ from $s_0(a)$ to $s_1(a)$.  On the boundary of an
		edge square $e\times I$, prescribe $s_0$ and $s_1$ on the two horizontal
		sides and the chosen vertex paths on the two vertical sides.  In a
		trivialization over $e$, the four prescribed sides form a loop in the model fiber.
		Simple connectivity fills the loop over the square.  The fillings agree
		on common sides, and local finiteness gives a continuous map
		\[
		H_\Gamma:\Gamma\times I\to E,
		\qquad p(H_\Gamma(a,u))=i(a),
		\qquad H_\Gamma(a,j)=s_j(i(a))\quad(j=0,1).
		\]
		
		We now return to $R$ while retaining the original endpoints.  A fiber
		bundle has the relative homotopy-lifting property for every CW pair
		\cite[Proposition~4.48]{Hatcher02}.  Triangulate $R$, give $R\times I$ the product CW
		structure, and apply the relative homotopy-lifting property to the CW pair
		$(R\times I,R\times\partial I)$, using $v\in I$ as the
		homotopy-lifting variable; $u$ remains the section-homotopy parameter.
		The base homotopy is
		\[
		B(x,u,v)=d_{1-v}(x).
		\]
		At $v=0$ prescribe the lift $H_\Gamma(r(x),u)$; on the two sides $u=j$
		prescribe $s_j(d_{1-v}(x))$ for $j=0,1$.  On the intersections $v=0$, $u=j$, the two prescriptions agree because
		\[
		H_\Gamma(r(x),j)=s_j(i(r(x)))=s_j(d_1(x)).
		\]
		Both prescribed lifts project to $B$.  Relative homotopy lifting therefore
		extends the prescribed boundary data to a continuous lift $\widetilde B:R\times I\times I\to E$.
		Set
		\[
		H(x,u)=\widetilde B(x,u,1).
		\]
		Then $p(H(x,u))=d_0(x)=x$, $H(x,0)=s_0(x)$, and $H(x,1)=s_1(x)$.
		Thus $H$ is the required homotopy through sections.  The relative
		lifting step retains both prescribed endpoints while returning the base
		map to $\operatorname{id}_R$.
	\end{proof}
	
	\section{First-jet sphere families}\label{sec:sphere-families}
	
	We verify the geometric hypothesis of
	Theorem~\ref{thm:rel-analytic-criterion} for rationally connected fibers.
	Koll\'ar's finite-jet theorem \cite[Definition--Theorem~2.1(5)]{Kol00}
	and simple-connectivity result \cite[Proposition~2.3]{Kol00} give the
	following assertions, respectively.
	\begin{theorem}
		\label{thm:rc-facts}
		Let $X$ be a connected smooth projective rationally connected complex
		manifold.
		\begin{enumerate}[label=\textup{(\roman*)},leftmargin=2.1em]
			\item Let $p_1,\dots,p_\ell$ be distinct points of $\Pone$, let
			$m_1,\dots,m_\ell$ be positive integers, and prescribe at each $p_i$ a
			holomorphic map jet of length $m_i$ (that is, derivatives through order
			$m_i-1$), with values in $X$.  There exists a morphism $h:\Pone\to X$
			realizing these jets and satisfying
			\[
			H^1\left(\Pone,h^*TX\left(-\sum_i m_i p_i\right)\right)=0.
			\]
			\item The topological fundamental group of $X$ is trivial.  Equivalently,
			every continuous loop in $X$ is homotopic to a constant loop.
		\end{enumerate}
	\end{theorem}
	
	To vary the base point and the prescribed first jet holomorphically,
	we need relative spaces of holomorphic maps. Fix a smooth projective
	morphism $\pi:X\to S$ of complex manifolds, a smooth projective curve
	$C$, and a finite effective Cartier divisor $D\subset C$, possibly empty.
	For analytic spaces $T\to S$, consider the functors
	\[
	T\longmapsto\operatorname{Mor}_S(C\times T,X),
	\qquad T\longmapsto\operatorname{Mor}_S(D\times T,X),
	\]
	where $\operatorname{Mor}_S$ denotes holomorphic maps over $S$.
	The following theorem establishes their representability. We write
	\[
	\mathcal H=\Hom_S(C\times S,X),
	\qquad \mathcal J_D=\Hom_S(D\times S,X)
	\]
	for the representing spaces, with $\mathcal J_D=S$ when $D$ is empty,
	and $\rho_D:\mathcal H\to\mathcal J_D$ for restriction to $D$.
	For $s_0\in S$ and $h:C\to X_{s_0}$, put $E=h^*T_{X/S}$ and write
	$[h]$ for the point of $\mathcal H$ represented by $h$.
	
	\begin{theorem}
		\label{thm:rel-hom-deformation}
		The two functors above, including their values on nonreduced analytic
		spaces, are represented by complex analytic spaces $\mathcal H$ and
		$\mathcal J_D$ with universal evaluations
		$C\times\mathcal H\to X$ and $D\times\mathcal J_D\to X$.
		These representing properties are compatible with analytic base change.
		The restriction map $\rho_D$ is holomorphic, and $\mathcal J_D$ is smooth
		over $S$.
		
		At $[h]$, the relative tangent space of $\mathcal H\to S$ is
		$H^0(C,E)$, and $H^1(C,E)$ is a complete obstruction
		space for lifting maps over a prescribed infinitesimal extension of the
		base point.  For $\rho_D$, the corresponding relative tangent and complete
		obstruction spaces are
		\[
		H^0(C,E(-D)),\qquad H^1(C,E(-D)).
		\]
		In particular, if $H^1(C,E(-D))=0$, then $\mathcal H$ is nonsingular at
		$[h]$ and $\rho_D$ is a holomorphic submersion at $[h]$.  Thus $\rho_D$ has a
		local holomorphic section through $[h]$.  More explicitly, if
		$\gamma:S'\to\mathcal J_D$ is a holomorphic family of jets and
		$\gamma(s'_0)=\rho_D([h])$, then the fiber product
		\[
		\mathcal H_\gamma=\mathcal H\mathbin{\times}_{\mathcal J_D}S'
		\]
		is smooth over $S'$ near $([h],s'_0)$.
	\end{theorem}
	
	\begin{proof}
		The proof has three steps: analytic representability, the relative
		\v Cech obstruction calculation, and the analytic submersion criterion.
		Pourcin's relative Douady theorem~\cite[Th\'eor\`eme~2]{Pourcin69},
		applied to $C\times X\to S$ and the structure sheaf, gives a relative
		Douady space $\mathcal Q\to S$ and a universal proper flat family
		\[
		\mathcal U\subset \mathcal Q\mathbin{\times_S}(C\times X).
		\]
		Let
		\[
		q:\mathcal U\longrightarrow \mathcal Q\times C
		\]
		be the projection.  Both $\mathcal U$ and $\mathcal Q\times C$ are proper and flat over
		$\mathcal Q$.  By Douady~\cite[\S10.1, Proposition~1]{Douady66},
		the set of $u\in\mathcal Q$ for which $q_u$ is an isomorphism is open, and
		$q$ itself is an isomorphism after restriction to that open set.  Denote
		the open set by $\mathcal H$.  Writing $\mathcal U_{\mathcal H}$ for the restriction of $\mathcal U$,
		the universal evaluation is
		\[
		\operatorname{ev}(c,h')
		=\operatorname{pr}_X\bigl((q|_{\mathcal U_{\mathcal H}})^{-1}(h',c)\bigr),
		\qquad (c,h')\in C\times\mathcal H.
		\]
		
		We verify the representing property, including nonreduced test spaces.
		For every analytic space $T\to S$, a $T$-point of $\mathcal H$ pulls the
		universal family back to a graph and hence gives an $S$-map
		$C\times T\to X$.  Conversely, the graph of any such map is a closed
		analytic subspace of $T\mathbin{\times_S}(C\times X)$, isomorphic to
		$C\times T$ and therefore proper and flat over $T$.  Pourcin's universal
		property gives a unique classifying map $c_T:T\to\mathcal Q$.  Since the
		pullback of $q$ to $T$ is an isomorphism, $c_T$ factors through the open set
		$\mathcal H$.  The two operations are inverse and functorial in $T$.
		The correspondence between maps and graphs is precisely the argument of
		Douady~\cite[\S10.2, Th\'eor\`eme~1]{Douady66}, now applied to the
		relative Douady space.  Replacing $C$ by the compact, possibly nonreduced space $D$ in the
		relative Douady construction and taking the open locus of graphs gives
		$\mathcal J_D$.  Pourcin works with analytic
		spaces that need not be reduced, and
		\cite[\S5, Remark~3]{Pourcin69} gives compatibility with arbitrary
		analytic base change.  Thus the asserted functors, universal evaluations,
		and restriction morphism have the stated strength.  No algebraicity of
		$S$ is used here.
		
		Write $D=\sum_{i=1}^{\ell}m_i a_i$ and let $n$ be the relative dimension of $X/S$.
		Since $\pi$ is a submersion, relative coordinates near the finitely many
		image points identify a neighborhood in $\mathcal J_D$ with a neighborhood
		in
		\[
		S_0\times\prod_i\mathbb C^{n m_i},
		\]
		where $S_0\subset S$ is a sufficiently small neighborhood of the base
		point;
		the factors record the coefficients modulo $w_i^{m_i}$ in local coordinates
		$w_i$ centered at $a_i$.  Thus $\mathcal J_D\to S$ is smooth, including
		when $D$ is nonreduced.
		
		We next compute the obstruction, allowing both a moving base and a
		nonreduced jet condition. We regard the spectrum of a local Artinian
		$\mathbb C$-algebra as its associated one-point analytic space.
		Let $A'\twoheadrightarrow A$ be a small
		extension of local Artinian $\mathbb C$-algebras with residue field
		$\mathbb C$ and kernel $I$, with the maximal ideal of $A'$ annihilating
		$I$. Fix a base point $s_A:\operatorname{Spec}A\to S$ and a deformation
		$h_A:C\times\operatorname{Spec}A\to X$ of $h$ over $s_A$.
		Prescribe a base-point lift $s_{A'}:\operatorname{Spec}A'\to S$ and a
		jet lift $j_{A'}:D\times\operatorname{Spec}A'\to X$ over $s_{A'}$.
		These lifts must reduce to $s_A$ and $h_A|_D$, respectively.
		Choose a finite Leray cover $U_1,\ldots,U_N$ of $C$ by coordinate disks
		such that each $h(U_i)$ lies in a relative coordinate chart for $X\to S$.
		For $1\leq i\leq\ell$, take $U_i$ centered at $a_i$ and containing no
		other point of $\operatorname{supp}D$; the remaining disks avoid
		$\operatorname{supp}D$.
		
		In the relative charts chosen for the Leray cover, the component functions of the given local
		maps lift from $A$ to $A'$.  On a disk centered at $a_i$, the difference
		between such a lift and the prescribed map on $m_i a_i$ is $I$-valued.
		The restriction
		\[
		\mathcal O(U_i)\longrightarrow
		\mathbb C[w_i]/(w_i^{m_i})
		\]
		is onto: a class is lifted by its Taylor polynomial.  Let $\delta_i$ be the $I$-valued discrepancy on $m_i a_i$, and let
		$\widetilde\delta_i$ be a Taylor-polynomial lift of $\delta_i$ on $U_i$.
		Subtract $\widetilde\delta_i$ from the chosen local lift of $h_A$.
		The adjusted map agrees with $j_{A'}$ on $m_i a_i$ and still reduces to
		$h_A$, because $\widetilde\delta_i$ has coefficients in $I$.
		
		On an overlap, two adjusted local lifts agree modulo $I$.  Because
		$I^2=0$, the difference of the two lifts transforms linearly as a vertical derivation in
		$E\otimes_{\mathbb C}I$.  Writing $c_{ij}$ for the lift on $U_j$ minus the lift
		on $U_i$, one has $c_{ij}+c_{jk}+c_{ki}=0$ on triple overlaps.  Since both local maps agree with the prescribed map on
		$D$, the cocycle vanishes to order $m_i$ at every $a_i$ and hence has
		coefficients in $E(-D)\otimes_{\mathbb C}I$.  Changing the local lifts changes $(c_{ij})$ by a coboundary.  The class
		\[
		o=[(c_{ij})]\in H^1(C,E(-D))\otimes_{\mathbb C}I
		\]
		is therefore well defined; the Leray property identifies the displayed
		\v Cech group with sheaf cohomology.  The equality $o=0$ means that $c_{ij}=b_j-b_i$ for local sections
		$b_i$ of $E(-D)\otimes I$.  Subtracting $b_i$ from the lift on $U_i$
		makes all adjusted lifts agree on overlaps.  Conversely, global gluing
		makes the difference cocycle a coboundary.
		When a lift exists, all lifts with the same base point and $D$-restriction
		form a torsor under $H^0(C,E(-D))\otimes I$.  Thus $o$ is a complete
		obstruction for every small extension.  Taking dual numbers proves the
		tangent assertion, and taking $D$ empty gives the unconstrained assertions.
		
		The infinitesimal lifting property implies an analytic submersion as
		follows.  If
		$H^1(C,E(-D))=0$, the preceding calculation gives the relative
		small-extension lifting property for $\rho_D$.  Given an absolute lifting
		problem for $(\mathcal H,[h])$, first lift the composite of the given Artinian map with $\rho_D$ to $\mathcal J_D$,
		which is possible since $\mathcal J_D$ is nonsingular, and then use the
		relative lifting property.  Thus $(\mathcal H,[h])$ has the absolute
		small-extension lifting property.  The absolute lifting property forces $(\mathcal H,[h])$ to be nonsingular:
		write the local algebra of this germ in a minimal embedding of dimension $e$ as
		\[
		B=\mathbb C\{t_1,\ldots,t_e\}/J,
		\qquad J\subset\mathfrak m^2,
		\qquad \mathfrak m=(t_1,\ldots,t_e).
		\]
		If $J\ne0$, choose a relation of least order $d\ge2$ among all nonzero
		relations.  The natural map from $B$ to
		$\mathbb C\{t\}/\mathfrak m^d$ would lift to
		$\mathbb C\{t\}/\mathfrak m^{d+1}$.  Any such lift sends $t_i$ to
		$t_i$ plus a term of degree $d$, so substitution preserves the nonzero homogeneous degree-$d$ part
		of the chosen relation.  This contradiction gives $J=0$.
		Finally, lifting arbitrary dual-number points of $\mathcal J_D$ through
		$[h]$ says that $d\rho_D$ is surjective.  The holomorphic submersion
		theorem now gives a local section through $[h]$.  Locally, $\rho_D$ has the form $V\times\mathbb B^k\to V$, where $\mathbb B^k$ is a ball
		in $\mathbb C^k$ and $V$ is open in $\mathcal J_D$.  Pullback by $\gamma$ gives
		$\gamma^{-1}(V)\times\mathbb B^k\to\gamma^{-1}(V)$, proving the asserted
		smoothness of $\mathcal H_\gamma\to S'$.
	\end{proof}
	
	\begin{remark}
		\label{rem:rel-first-jet-base-change}
		Let $p:X\to S$ have the relative first-jet sphere-family property.
		This property is preserved by restriction to an open subset of the base
		and by holomorphic base change $g:S'\to S$ between complex manifolds.
		For $X'=S'\times_S X$, one has $T_{X'/S'}\simeq S'\times_S T_{X/S}$.
		Given one of its families $G:\Pone\times\Omega\to X$, the pulled-back family is
		\[
		G'(w;s',x,b)=(s',G(w;x,b)),\qquad g(s')=p(x).
		\]
		The domain of $G'$ is $\Pone\times\Omega'_0$, where
		\[
		\Omega'_0
		=\{(s',x,b):(x,b)\in\Omega,\ g(s')=p(x)\}
		\subset T_{X'/S'}.
		\]
		The identities for $G'$ follow componentwise from \eqref{eq:rel-sphere-identities}.
		
		For each fixed family $G$ and each relatively compact parameter subdomain $\Omega'\Subset\Omega$, compactness of $\Pone\times\overline{\Omega'}$ gives bounds for source and parameter derivatives of every fixed finite order.  In local coordinates one covers this compact set by finitely many smaller source--parameter coordinate neighborhoods whose closures lie in larger neighborhoods on which $G$ takes values in fixed target charts; the coordinate expressions are holomorphic on the larger neighborhoods, so their derivatives are bounded on the smaller ones.  The bounds may depend on the order and on the finite charts; no single holomorphic choice over an arbitrary compact jet set is asserted.
	\end{remark}
	
	\begin{proposition}
		\label{prop:rel-unobstructed-jet-criterion}
		Let $p:X\to S$ be a smooth projective morphism of complex manifolds.
		Suppose that, for every $s\in S$, $x\in X_s$, and
		$b\in T_{X_s,x}$, there is a morphism $h:\Pone\to X_s$ with
		\[
		h(0)=x,\qquad \partial_wh(0)=b,\qquad
		H^1\bigl(\Pone,h^*T_{X_s}(-2[0])\bigr)=0.
		\]
		Then $p$ has the relative first-jet sphere-family property.
	\end{proposition}
	
	\begin{proof}
		Apply Theorem~\ref{thm:rel-hom-deformation} with $C=\Pone$ and $D=2[0]$.
		In the fixed source coordinate $w$, the relative jet space
		$\mathcal J_D$ is naturally $T_{X/S}$: a map from $2[0]$ records its
		value and vertical derivative.  At each chosen $[h]$, the displayed
		vanishing makes $\rho_{2[0]}$ a holomorphic submersion, with respect to the
		base point, value, and derivative.  It therefore has a
		local holomorphic section $\sigma:\Omega\to\mathcal H$ through $[h]$.
		The universal evaluation map gives the jointly holomorphic family
		\[
		G(w;x',b')=\operatorname{ev}(w,\sigma(x',b')).
		\]
		Since $\rho_{2[0]}\circ\sigma=\id_\Omega$, the restriction of
		$G(\,\cdot\,;x',b')$ to $2[0]$ has value $x'$ and derivative $b'$.
		The evaluation map is over $S$, so $p(G(w;x',b'))=p(x')$.
		This verifies the three identities in \eqref{eq:rel-sphere-identities}.
		The vanishing hypothesis and submersion argument also apply at $b=0$.
	\end{proof}
	\begin{proposition}
		\label{prop:rel-jet-family}
		Let $p:X\to S$ be a smooth projective morphism of complex manifolds with
		rationally connected fibers.  Then $p$ has the relative first-jet sphere-family
		property.  Explicitly, fix $s_0\in S$, $x_0\in X_{s_0}$, and
		$b_0\in T_{X/S,x_0}$.  There are a neighborhood $\Omega$ of
		$(s_0,x_0,b_0)$ in the relative first-jet bundle
		and a jointly holomorphic map
		\[
		G:\Pone\times\Omega\longrightarrow X
		\]
		such that
		\[
		p(G(w;s,x,b))=s,
		\qquad G(0;s,x,b)=x,
		\qquad \partial_w G(0;s,x,b)=b.
		\]
		On relatively compact parameter subsets, the derivatives of $G$ of every
		fixed finite order are uniformly bounded on $\Pone$, measured in fixed
		finite collections of parameter, source, and target charts.
	\end{proposition}
	
	\begin{proof}
		Theorem~\ref{thm:rc-facts}(i), applied in the fiber $X_{s_0}$ to one jet of
		length two, gives a morphism $h:\Pone\to X_{s_0}$ with the prescribed
		value and derivative, including the zero derivative, and
		\[
		H^1\bigl(\Pone,h^*T_{X_{s_0}}(-2[0])\bigr)=0.
		\]
		Proposition~\ref{prop:rel-unobstructed-jet-criterion} gives the family
		$G$.  The derivative bounds follow from
		Remark~\ref{rem:rel-first-jet-base-change}.
	\end{proof}
	
	\begin{proof}[Proof of Theorem~\ref{thm:rel-main}]
		By Definition~\ref{def:rel-smooth-projective}, $\pi$ is proper and a
		holomorphic submersion; it is surjective by the nonempty-fiber convention.
		Proposition~\ref{prop:rel-jet-family} gives the relative first-jet sphere-family
		property, and Theorem~\ref{thm:rc-facts}(ii) gives simple connectivity of
		the fibers.  Together with their assumed connectedness, these are all
		the hypotheses of Theorem~\ref{thm:rel-analytic-criterion}.  That theorem
		gives the Oka-1 map property and the stated closed discrete interpolation
		conclusion.
	\end{proof}
	
	\section{The fixed section equation}\label{sec:vertical-preparation}
	
	\subsection{Vertical equations and jet twisting}
	
	Fix a proper holomorphic submersion $\pi:Z\to Y$, a Riemann surface $R$,
	and a holomorphic map $g:R\to Y$. Put
	\[
	X=R\times_Y Z,\qquad p:X\longrightarrow R.
	\]
	Then $p$ is a proper holomorphic submersion. Write $X_z=p^{-1}(z)$ for its
	fiber over $z$. Its vertical tangent bundle is
	\[
	T_{X/R}=\ker(dp:TX\to p^*TR).
	\]
	The correspondence between liftings and sections is
	\[
	f:R\to Z\quad\longleftrightarrow\quad s_f:R\to X,\qquad
	s_f(z)=(z,f(z)),\quad f=\operatorname{pr}_Z\circ s_f.
	\]
	Equip $X$ with the metric induced from a product Hermitian metric on
	$R\times Z$. For points $x_1,x_2\in X$ over the same point of $R$,
	\[
	d_Z(\operatorname{pr}_Zx_1,\operatorname{pr}_Zx_2)\leq d_X(x_1,x_2),
	\]
	since projection to $Z$ does not increase lengths. The local constructions
	below also apply near a compact image of a section lying in the
	submersion locus of a holomorphic map.
	
	To express nearby sections in one bundle, we use a local addition that
	stays in each fiber and is holomorphic in its vector parameter.
	\begin{lemma}
		\label{lem:rel-addition}
		Let $D\subset R$ be compact.  There are an open neighborhood $\mathcal U\subset T_{X/R}$ of the
		zero vectors over $p^{-1}(D)$ and a smooth map
		\[
		\Psi:\mathcal U\longrightarrow X
		\]
		with the following properties:
		\begin{enumerate}[label=\textup{(\roman*)},leftmargin=2.2em]
			\item $p(\Psi(x,v))=p(x)$ and $\Psi(x,0)=x$;
			\item for fixed $x$, the map $v\mapsto\Psi(x,v)$ is holomorphic and its
			derivative at zero is the identity of $T_{X/R,x}$;
			\item after the fiber neighborhood is uniformly shrunk, the fiber map is a
			biholomorphism onto a neighborhood of $x$ inside $X_{p(x)}$;
			\item all derivatives through order two in the variables $x$ and order
			three in the variables $v$ are uniformly bounded on smaller fiber
			neighborhoods.
		\end{enumerate}
	\end{lemma}
	
	\begin{proof}
		Put $K_0=p^{-1}(D)$, which is compact because $p$ is proper.  Cover $K_0$ by
		finitely many relative holomorphic coordinate charts
		\[
		(z,y):U_i\longrightarrow V_i\times W_i\subset\C\times\C^n,
		\qquad p(z,y)=z,
		\]
		where $n$ is the fiber dimension. We average coordinate differences in the vertical tangent space at $x$,
		using weights that depend only on $x$.  Choose a
		smooth partition of unity $\rho_i$ on a neighborhood of
		$K_0$, subordinate to the relative coordinate cover, and put
		$K_i=\operatorname{supp}\rho_i\Subset U_i$.  We first specify the domain
		of every summand.  In the fiber product $X\times_R X$, put
		\[
		\theta_i(x,y')=\rho_i(x)
		\bigl((d_y\phi_i)_x\bigr)^{-1}
		\bigl(\phi_i(y')-\phi_i(x)\bigr)
		\]
		when $x,y'\in U_i$, where $\phi_i$ is the vertical coordinate, and set
		$\theta_i=0$ when $x\notin K_i$.  The two formulas agree on
		$(U_i\times_RU_i)\cap((X\setminus K_i)\times_RX)$ because $\rho_i(x)=0$
		on that intersection.  Thus $\theta_i$ is smooth on the union
		\[
		(U_i\times_RU_i)\ \cup\ ((X\setminus K_i)\times_RX).
		\]
		This union contains the relative diagonal: if
		$x\in K_i$, then $x\in U_i$, and otherwise the zero definition applies.
		Intersecting these finitely many domains gives an open neighborhood
		$\mathcal V$ of the relative diagonal over $K_0$ on which
		\[
		\Theta_x(y')=\sum_i\theta_i(x,y')\in T_{X/R,x}
		\]
		is well defined and smooth.  For fixed $x$, $\Theta_x$ is holomorphic in $y'$, and
		\[
		\Theta_x(x)=0,
		\qquad (d_{y'}\Theta_x)_{y'=x}=
		\sum_i\rho_i(x)\id=
		\id.
		\]
		
		Consider
		\[
		\Gamma:\mathcal V\longrightarrow T_{X/R},
		\qquad \Gamma(x,y')=(x,\Theta_x(y')).
		\]
		Along the relative diagonal, the derivative of $\Gamma$ is block triangular: the first
		diagonal block is the identity in the $x$-variable and the second is the
		displayed identity in the vertical $y'$-variable.  Hence $\Gamma$ is a
		local diffeomorphism near every point of that diagonal.  Compactness of
		$K_0$, together with the fact that $\Gamma$ retains the first component
		$x$, allows one to shrink to a single neighborhood of the diagonal on
		which $\Gamma$ is injective.  Otherwise there would be pairs
		$(x_j,y_j)\ne(x_j,y'_j)$ approaching the diagonal, with
		$\Gamma(x_j,y_j)=\Gamma(x_j,y'_j)$.  After taking a subsequence,
		$x_j\to x\in K_0$ and $y_j,y'_j\to x$.  Both pairs then lie in a
		neighborhood where $\Gamma$ is injective, a contradiction.
		
		The inverse has the form $(x,v)\mapsto(x,\Psi(x,v))$.  Shrinking once more,
		there is a number $\rho>0$ such that the closed metric ball bundle
		$\{(x,v):x\in K_0,\ |v|\leq2\rho\}$ lies in the domain of $\Gamma^{-1}$.  We work on
		$|v|<\rho$, leaving a fixed margin for the derivative estimates.
		
		\noindent\textup{(i)} The map $\Gamma$ preserves the first component, so
		its inverse satisfies $p(\Psi(x,v))=p(x)$.  Since $\Theta_x(x)=0$,
		one also has $\Psi(x,0)=x$.
		
		\noindent\textup{(ii)} For fixed $x$, the holomorphic inverse-function
		theorem makes $\Psi(x,\cdot)$ holomorphic.  The identity
		$(d_{y'}\Theta_x)_x=\id$ gives $d_v\Psi(x,0)=\id$.
		
		\noindent\textup{(iii)} On the neighborhood already chosen, $\Gamma$
		is injective and a local diffeomorphism.  Its restriction to each fiber
		and the corresponding restriction of its inverse are holomorphic;
		thus $\Psi(x,\cdot)$ is a biholomorphism onto a neighborhood of $x$.
		
		\noindent\textup{(iv)} The inverse is smooth jointly in $(x,v)$.
		On smaller closed balls, a finite coordinate cover and compactness bound
		all mixed derivatives, including the orders in (iv).
	\end{proof}
	
	Given a local addition as in Lemma~\ref{lem:rel-addition}, let $s$ be a
	smooth section of $p$ on a neighborhood of a compact bordered Riemann
	surface.  Put
	\[
	E=s^*T_{X/R}.
	\]
	For a small section $u$ of $E$, put
	\begin{equation}\label{eq:rel-normalized-operator}
		\mathscr S_s(u)=
		(d_u\Psi(s(z),u))^{-1}
		\dbar\bigl[\Psi(s(z),u(z))\bigr].
	\end{equation}
	Here $d_u\Psi(s(z),u(z))$ is an isomorphism
	$E_z\to T_{X/R,\Psi(s(z),u(z))}$.  The derivative
	$\dbar[\Psi(s(z),u(z))]$ is vertical because
	$p(\Psi(s(z),u(z)))=z$.  Thus $\mathscr S_s(u)$ is an $E$-valued
	$(0,1)$-form on the domain of $s$.
	
	\begin{lemma}
		\label{lem:rel-normal-form}
		In a smooth complex frame of $E$,
		\[
		\mathscr S_s(u)=\dbar u+A(z,u),
		\]
		where $A$ is smooth in $z$ and holomorphic in $u$, and $A$ does not depend on
		derivatives of $u$.  The linearization
		\[
		D_s=d\mathscr S_s(0)
		\]
		is a complex-linear Cauchy--Riemann operator on $E$.  We write $E_{D_s}$ for
		$E$ endowed with the holomorphic structure defined by $D_s$.  In a
		$D_s$-holomorphic frame,
		\[
		A_u(z,0)=0.
		\]
		If $s$ is holomorphic on an open set $U$, then $A(z,0)=0$ for $z\in U$.
	\end{lemma}
	
	\begin{proof}
		In relative target coordinates, write the vertical coordinate expression of
		$\Psi(s(z),q)$ as $\Phi(z,q)$.  Let $\partial_{\bar z}\Phi(z,q)$ denote differentiation with $q$ fixed.
		Since $\Phi$ is holomorphic in $q$,
		\[
		\partial_{\bar z}\Phi(z,u(z))
		=\partial_{\bar z}\Phi(z,u)
		+\Phi_q(z,u)\partial_{\bar z}u.
		\]
		Multiplication by $\Phi_q^{-1}$ gives the displayed semilinear form with
		\[
		A=\Phi_q^{-1}\partial_{\bar z}\Phi\,d\bar z.
		\]
		The linearization is $D_s=\dbar+A_u(z,0)$ and satisfies the complex
		Leibniz rule; Lemma~\ref{lem:frames} supplies the holomorphic structure
		on $E$.  Under a smooth frame change $q=T(z)q'$, the coefficient becomes
		\[
		A'(z,q')=T^{-1}A(z,Tq')+T^{-1}(\dbar T)q'.
		\]
		A $D_s$-holomorphic frame cancels the linear term.  If $s$ is holomorphic,
		then $\mathscr S_s(0)=0$, so $A(z,0)=0$ on every open set where $s$ is holomorphic.
	\end{proof}
	
	In a general smooth frame, the coefficient $A$ may include a linear term
	and obeys the inhomogeneous transformation law in the proof of
	Lemma~\ref{lem:rel-normal-form}.  The intrinsic zero-order remainder is
	$\mathcal R_s(u)=\mathscr S_s(u)-D_su$.  In $D_s$-holomorphic frames,
	$\mathcal R_s$ is represented by $A$, has zero fiber-linear term,
	and transforms tensorially.  The
	auxiliary extension in the bordered construction applies a cutoff to the remainder $\mathcal R_s$ in holomorphic frames.  For a nonholomorphic section, $D_s$ may depend on the
	chosen local addition; Lemma~\ref{lem:alg-natural-Ds} identifies $D_s$ canonically wherever
	$s$ is holomorphic.
	
	\begin{lemma}
		\label{lem:alg-natural-Ds}
		Let $p:X\to\Sigma$ be holomorphic and let $s$ be a smooth section
		whose image lies in the submersion locus.  Choose a vertical local
		addition near $s(\Sigma)$ and put $D_s=d\mathscr S_s(0)$ by
		\eqref{eq:rel-normalized-operator}.  If $s$ is holomorphic on an open
		set $U\subset\Sigma$, then the
		Cauchy--Riemann operator $D_s|_U$ is the natural Dolbeault operator on the
		holomorphic bundle $s^*T_{X/\Sigma}|_U$.
	\end{lemma}
	
	\begin{proof}
		Choose a source coordinate $z$, relative holomorphic target coordinates,
		and a holomorphic frame of $s^*T_{X/\Sigma}$ on a smaller open subset of $U$.  Let
		$\Phi(z,q)$ denote the vertical-coordinate expression of
		$\Psi(s(z),q)$ and put
		\[
		B(z)=\partial_q\Phi(z,0).
		\]
		Since $d_u\Psi(s(z),0)$ is the identity of the vertical tangent space,
		$B(z)$ is the coordinate matrix of the identity between two holomorphic
		frames.  The matrix $B(z)$ is therefore holomorphic and invertible.  Since $s$ is
		holomorphic, $\Phi(z,0)$ is holomorphic.  The normalized equation is
		\[
		\dbar q+A(z,q),
		\qquad
		A(z,q)=\bigl(\partial_q\Phi(z,q)\bigr)^{-1}
		\partial_{\bar z}\Phi(z,q)\,d\bar z.
		\]
		Thus $A(z,0)=0$.  On differentiating in $q$ at zero, the derivative of the
		inverse matrix is multiplied by $\partial_{\bar z}\Phi(z,0)=0$, while the
		remaining term is $B^{-1}\partial_{\bar z}B=0$.  Hence $A_q(z,0)=0$, and
		the linearization in the chosen holomorphic frame is $\dbar$.
	\end{proof}
	
	Prescribed jets are enforced by requiring the vertical displacement
	to vanish along a divisor.  The absence of constant and linear terms
	near the interpolation points makes the twisted equation smooth across
	that divisor.
	
	\begin{lemma}
		\label{lem:rel-twist}
		Let $A\subset R$ be finite, let $k_a\geq0$ for $a\in A$, and set
		\[
		\Delta=\sum_{a\in A}(k_a+1)a.
		\]
		Assume that $\operatorname{supp}\Delta$ is contained in an open set on which
		$s$ is holomorphic.  Set $E_{\Delta}=E_{D_s}(-\Delta)$, and let
		$\iota:E_{\Delta}\to E_{D_s}$ be the canonical holomorphic morphism.  Write $\dbar_{E_{\Delta}}$
		for the Dolbeault operator of $E_{\Delta}$.
		The operator
		\[
		\mathscr S_{\Delta}(v)=\iota^{-1}\mathscr S_s(\iota v),
		\]
		initially defined off $\operatorname{supp}\Delta$, extends across
		$\operatorname{supp}\Delta$ to a smooth semilinear
		operator on a sufficiently small fixed fiber neighborhood in $E_{\Delta}$.  In
		$\dbar_{E_{\Delta}}$-holomorphic frames $\mathscr S_{\Delta}$ has the form
		\begin{equation}\label{eq:rel-twisted-form}
			\mathscr S_{\Delta}(v)=\dbar v+A_{\Delta}(z,v),
			\qquad (A_{\Delta})_v(z,0)=0.
		\end{equation}
		The coefficient $A_{\Delta}(z,0)$ vanishes on every open set on which $s$ is holomorphic.
	\end{lemma}
	
	\begin{proof}
		Put $n=\operatorname{rank}E$. Fix $a\in\operatorname{supp}\Delta$, choose a coordinate $z$ centered at $a$,
		and put $m=k_a+1$.  In compatible holomorphic frames, $\iota$ is
		multiplication by $z^m$.  Write $\mathscr S_s(u)=\dbar u+\mathcal A(z,u)$
		in these frames.  Lemma~\ref{lem:rel-normal-form} and the
		holomorphy of $s$ near $a$ give $\mathcal A(z,0)=\mathcal A_u(z,0)=0$.  Hence
		\[
		\mathcal A(z,u)=\sum_{|\alpha|\geq2}a_\alpha(z)u^\alpha
		\]
		on a smaller fiber ball, where $\alpha\in\mathbb N^n$
		and $a_\alpha$ is a smooth $(0,1)$-form coefficient.  Since
		$\dbar z^m=0$, one obtains
		\begin{equation}\label{eq:rel-divisor-division}
			z^{-m}\mathscr S_s(z^m v)
			=\dbar v+\sum_{|\alpha|\geq2}
			a_\alpha(z)z^{m(|\alpha|-1)}v^\alpha.
		\end{equation}
		The right-hand side is smooth at $z=0$ and holomorphic in $v$; Cauchy
		estimates give locally uniform convergence with all mixed derivatives on a
		smaller fixed fiber ball.  Every exponent $m(|\alpha|-1)$ is nonnegative, and every monomial
		has degree at least two in $v$.  Thus the extended fiber-linear term
		vanishes at $v=0$.  Off $\operatorname{supp}\Delta$, the
		identity $D_s\iota=\iota\dbar_{E_{\Delta}}$ and
		$\mathcal A_u(z,0)=0$ give $(A_{\Delta})_v(z,0)=0$ directly.  Finally, $\mathscr S_s(0)=0$ wherever $s$ is
		holomorphic, which proves the assertion about the constant term.
	\end{proof}
	
	In the reconstruction, the identity $u=\iota v=z^{k_a+1}v$, with $v$ smooth and bounded,
	forces $u=O(|z|^{k_a+1})$ near $a$.  Once the reconstructed section is holomorphic,
	the bound for $u$ gives equality of the prescribed jets, as shown at the end of the
	proof of Theorem~\ref{thm:rel-bordered}; the parameter itself need not be
	holomorphic in an arbitrary smooth frame.
	
	\subsection{Small solutions on a fixed bundle}
	
	In both approximation constructions the bundle, norms, and right inverse
	are fixed before the disks are selected.
	\begin{lemma}
		\label{lem:fixed-bundle-solution}
		Let $E$ be a holomorphic vector bundle on a compact Riemann surface
		$\Sigma$, with $H^1(\Sigma,E)=0$.  Fix smooth metrics, the resulting area
		form, a radius $\rho>0$, and a positive number $r_0$. Write
		$E(\rho)=\{v\in E:|v|<\rho\}$.  For $0<r<r_0$, let
		\[
		A_r:E(\rho)\longrightarrow\Lambda^{0,1}T^*\Sigma\otimes E
		\]
		be smooth bundle maps, holomorphic in the fiber variable, with
		\[
		\sup_{E(\rho)}\bigl(|A_r|+|d_vA_r|+|d_v^2A_r|\bigr)\leq M
		\]
		for a constant $M$ independent of $r$.  Put
		\[
		\mathbb X=W^{1,4}(\Sigma,E),\quad
		\mathbb Y=L^4(\Sigma,\Lambda^{0,1}T^*\Sigma\otimes E),\quad
		\mathscr F_r(v)=\dbar_Ev+A_r(z,v(z)).
		\]
		The domain of $\mathscr F_r$ is the open set
		$\{v\in\mathbb X:\|v\|_{C^0}<\rho\}$.
		Suppose that $\varepsilon_r>0$ tends to zero as $r\downarrow0$ and
		\[
		\|\mathscr F_r(0)\|_{\mathbb Y}\leq\varepsilon_r,
		\qquad
		\|d\mathscr F_r(0)-\dbar_E\|_{\mathbb X\to\mathbb Y}
		\leq\varepsilon_r.
		\]
		For all sufficiently small $r$, the following hold:
		\begin{enumerate}[label=\textup{(\roman*)},leftmargin=2.2em]
			\item There is a solution $v_r\in\mathbb X$ of $\mathscr F_r(v_r)=0$ satisfying
			\[
			\|v_r\|_{W^{1,4}}\leq C\varepsilon_r,
			\qquad \|v_r\|_{C^0}<\rho.
			\]
			The constant $C$ and the required upper bound for $\varepsilon_r$ depend
			only on the fixed bundle, metrics, $\rho$, and $M$.
			\item The solution $v_r$ in \textup{(i)} is smooth. No bounds for source
			derivatives of $A_r$ that are uniform in $r$ are required.
		\end{enumerate}
	\end{lemma}
	
	\begin{proof}
		\noindent\textup{(i)} Dolbeault Fredholm theory and $H^1(\Sigma,E)=0$ give
		surjectivity of $\dbar_E:\mathbb X\to\mathbb Y$; the corresponding
		regularity estimate is~\eqref{eq:tools-dolbeault-estimate}. Put
		\[
		\mathbb X_0=\{u\in\mathbb X:
		\langle u,h\rangle_{L^2}=0\ \text{for every }h\in H^0(\Sigma,E)\}.
		\]
		The restriction $\dbar_E:\mathbb X_0\to\mathbb Y$ is a bounded bijection
		between Banach spaces.  The bounded inverse theorem gives
		$T:\mathbb Y\to\mathbb X_0\subset\mathbb X$ with $\dbar_ET=I_{\mathbb Y}$.
		
		On $\{v\in\mathbb X:\|v\|_{C^0}<\rho\}$, the map $\mathscr F_r$ is $C^2$.
		Indeed, Taylor expansion in each of finitely many fixed bundle frames gives
		\[
		\begin{aligned}
			d\mathscr F_r(v)\xi
			&=\dbar_E\xi+(d_vA_r)(z,v(z))\xi(z),\\
			d^2\mathscr F_r(v)(\xi,\xi')
			&=(d_v^2A_r)(z,v(z))[\xi(z),\xi'(z)].
		\end{aligned}
		\]
		Sobolev embedding $W^{1,4}\hookrightarrow C^0$, density, and the fiber
		bounds give the displayed first- and second-derivative formulas and continuity on the indicated open
		set. More explicitly, for fixed $r$ and $v$, choose
		$\|v\|_{C^0}<\rho'<\rho$. On the closed $\rho'$-ball bundle,
		$d_v^2A_r$ has a modulus of continuity $\omega_{r,\rho'}(t)\to0$.
		For sufficiently small $\|h\|_{\mathbb X}$, the segments $v+th$,
		$0\leq t\leq1$, lie in that ball bundle. Taylor's integral formula and $W^{1,4}\hookrightarrow C^0$ give
		\[
		\|A_r(\cdot,v+h)-A_r(\cdot,v)-d_vA_r(\cdot,v)h\|_{L^4}
		\leq C\|h\|_{C^0}\|h\|_{L^4}=O(\|h\|_{\mathbb X}^2).
		\]
		Applying the integral formula to $d_vA_r$ bounds the remainder after
		subtracting $d_v^2A_r(\cdot,v)[h,\xi]$ by
		\[
		\omega_{r,\rho'}(\|h\|_{C^0})\|h\|_{C^0}\|\xi\|_{L^4}.
		\]
		Taking the supremum over $\|\xi\|_{\mathbb X}\leq1$ proves
		Fr\'echet differentiability of the first differential in operator norm;
		the same modulus proves continuity of the second differential.
		No modulus uniform in $r$ is needed. For
		\[
		\mathcal N_r(v)=\mathscr F_r(v)-\mathscr F_r(0)-d\mathscr F_r(0)v,
		\]
		the fiber bound and Sobolev embedding give a constant $K$ independent of $r$
		such that
		\[
		\|d^2\mathscr F_r(u)(\xi,\xi')\|_{\mathbb Y}
		\leq K\|\xi\|_{\mathbb X}\|\xi'\|_{\mathbb X}.
		\]
		Let $C_S$ be a Sobolev embedding constant for $\mathbb X\hookrightarrow C^0$.
		Take $\delta>0$ with $C_S\delta<\rho/2$.  For
		$\|v\|_{\mathbb X},\|w\|_{\mathbb X}\leq\delta$, the segment
		$w+t(v-w)$ stays in that ball, and
		\[
		\mathcal N_r(v)-\mathcal N_r(w)
		=\int_0^1\bigl(d\mathscr F_r(w+t(v-w))-d\mathscr F_r(0)\bigr)(v-w)\,dt.
		\]
		The second-derivative bound therefore yields
		\begin{equation}\label{eq:fixed-nonlinear}
			\|\mathcal N_r(v)-\mathcal N_r(w)\|_{\mathbb Y}
			\leq C(\|v\|_{\mathbb X}+\|w\|_{\mathbb X})\|v-w\|_{\mathbb X}.
		\end{equation}
		The constant in~\eqref{eq:fixed-nonlinear} depends on $M$, the finite frame changes, the Sobolev embedding
		constant, and the fixed area of $\Sigma$, and is independent of $r$.
		
		Put $P_r=d\mathscr F_r(0)-\dbar_E$.  For $\varepsilon_r\|T\|<1/2$, one has $\|P_rT\|<1/2$, so
		\[
		R_r=T(I+P_rT)^{-1}
		\]
		is a right inverse of $d\mathscr F_r(0)$ with $\|R_r\|\leq2\|T\|$.
		Indeed, $d\mathscr F_r(0)T=I+P_rT$, hence
		$d\mathscr F_r(0)R_r=I$.  Consider the fixed-point equation
		\begin{equation}\label{eq:fixed-contraction}
			v=-R_r\bigl(\mathscr F_r(0)+\mathcal N_r(v)\bigr).
		\end{equation}
		Increase $K$, if necessary, to bound the constant in~\eqref{eq:fixed-nonlinear},
		put $C_0=\max\{1,4\|T\|\}$, and take $r$ small enough that
		\[
		C_0\varepsilon_r<\delta,\qquad
		4\|T\|KC_0\varepsilon_r<\tfrac12.
		\]
		Write $\mathcal B_r$ for the closed $\mathbb X$-ball of radius
		$C_0\varepsilon_r$, and write $Q_r(v)$ for the right side of
		\eqref{eq:fixed-contraction}.  Since $\mathcal N_r(0)=0$,
		\[
		\begin{aligned}
			\|Q_r(v)\|_{\mathbb X}
			&\leq2\|T\|\bigl(\varepsilon_r+KC_0^2\varepsilon_r^2\bigr)
			< C_0\varepsilon_r,\\
			\|Q_r(v)-Q_r(w)\|_{\mathbb X}
			&\leq4\|T\|KC_0\varepsilon_r\|v-w\|_{\mathbb X}
			<\tfrac12\|v-w\|_{\mathbb X}
			\qquad(v,w\in\mathcal B_r).
		\end{aligned}
		\]
		The choice of $\delta$ also gives
		$C_SC_0\varepsilon_r<\rho/2$.  The Banach contraction theorem gives a fixed
		point $v_r$.  Applying $d\mathscr F_r(0)$ to
		\eqref{eq:fixed-contraction} gives
		\[
		d\mathscr F_r(0)v_r=-\mathscr F_r(0)-\mathcal N_r(v_r),
		\qquad \mathscr F_r(v_r)=0.
		\]
		Membership in $\mathcal B_r$ gives $\|v_r\|_{\mathbb X}\leq C_0\varepsilon_r$.
		
		\noindent\textup{(ii)} Fix a value of $r$ for which the fixed point in \textup{(i)}
		exists. In a local holomorphic frame the equation is
		$\dbar v_r=-A_r(z,v_r)$.  The coefficient is smooth across all source
		charts. In a source coordinate, let $\nabla$ denote the real gradient.
		Since $v_r\in W^{1,4}\subset C^0$, the Sobolev chain rule gives
		\[
		\nabla[A_r(z,v_r)]
		=(\nabla_zA_r)(z,v_r)+(d_vA_r)(z,v_r)\nabla v_r\in L^4.
		\]
		Local Dolbeault regularity yields $v_r\in W^{2,4}$.  On successively smaller source charts, the Sobolev product and chain
		rules give $A_r(z,v_r)\in W^{k,4}$ whenever $v_r\in W^{k,4}$,
		for $k\geq1$. The local form of~\eqref{eq:tools-dolbeault-estimate} then gives
		$v_r\in W^{k+1,4}$.  Induction and the Sobolev embeddings prove smoothness.  Source derivatives enter only the regularity argument for a fixed $r$;
		the bounds for those source derivatives may depend on $r$.
	\end{proof}
	
	\section{Analytic approximation}\label{sec:analytic-approximation}
	
	\subsection{Prepared rational disks}
	
	The remainder of the analytic construction assumes the relative first-jet
	sphere-family property, rather than projectivity or rational connectedness.
	Write $n$ for the relative dimension, $D_r=\Delta_r=\{z\in\C:|z|<r\}$,
	and $B_r=\{q\in\C^n:|q|<r\}$. The substitution $w=\sigma^2/z$ replaces
	the antiholomorphic first-order term of a smooth section by a holomorphic
	rational disk.
	\begin{lemma}
		\label{lem:rel-prepared}
		Let $p:X\to D_{r_0}$ be a proper holomorphic submersion with the relative
		first-jet sphere-family property, and let
		\[
		\Phi:D_{r_0}\times B_{2d_0}\longrightarrow X,
		\qquad p(\Phi(z,q))=z,
		\]
		be smooth and holomorphic in $q$.  Suppose that the image of $\Phi$ near $(0,0)$ lies
		in a relative holomorphic chart $(p,\phi)$.  Writing
		$\widehat\Phi=\phi\circ\Phi$, we have, after shrinking the two balls,
		\begin{equation}\label{eq:rel-family-taylor}
			\widehat\Phi(z,q)=m(q)+a(q)z+b(q)\bar z+R(z,q),
			\qquad R=O_{C_q^3}(|z|^2),
		\end{equation}
		where
		\[
		m(q)=\widehat\Phi(0,q),\quad
		a(q)=\partial_z\widehat\Phi(0,q),\quad
		b(q)=\partial_{\bar z}\widehat\Phi(0,q)
		\]
		are holomorphic in $q$.  Here $O_{C_q^3}(|z|^2)$ means that
		the $q$-derivatives through order three are bounded by $C|z|^2$ on the
		fixed parameter ball. The notation $\|\cdot\|_{C_q^j}$ denotes the supremum
		of all parameter derivatives through order $j$ on the indicated ball.
		On $|z|=\sigma$, the identity $\bar z=\sigma^2/z$ is the matching relation used to replace the
		antiholomorphic first-order term by a holomorphic rational disk.
		
		There are $d,s_0,C>0$ such that, for every $0<\sigma<s_0/4$, there is a
		jointly holomorphic map
		\[
		H_\sigma:D_{s_0}\times B_{2d}\longrightarrow X,
		\qquad p(H_\sigma(z,q))=z.
		\]
		On $\sigma/2\leq|z|\leq2\sigma$, the vertical coordinate of $H_\sigma$ satisfies
		\begin{equation}\label{eq:rel-prepared-expansion}
			\phi(H_\sigma(z,q))
			=m(q)+a(q)z+b(q)\frac{\sigma^2}{z}+R_\sigma(z,q),
		\end{equation}
		with, for $|\alpha|\leq3$,
		\begin{equation}\label{eq:rel-prepared-estimates}
			\sup_{|q|\leq d}|\partial_q^\alpha R_\sigma(z,q)|\leq C\sigma^2,
			\qquad
			\sup_{|q|\leq d}|\partial_z\partial_q^\alpha R_\sigma(z,q)|
			\leq C\sigma.
		\end{equation}
		Moreover, $d_q H_\sigma$ is uniformly bounded on
		$D_{s_0}\times B_d$, independently of $\sigma$.  The constants are locally
		uniform when the source and target charts are fixed, the relevant first
		jets remain in a relatively compact subset of the domain of one fixed
		first-jet family, and the indicated mixed derivatives of $\Phi$ are
		uniformly bounded.  The additional compact-center uniformity needed below,
		including target-chart margins and inverse bounds, is stated separately in
		Lemma~\ref{lem:rel-uniform-matching}.
	\end{lemma}
	
	\begin{proof}
		Equation~\eqref{eq:rel-family-taylor} is Taylor's formula in
		$(\operatorname{Re}z,\operatorname{Im}z)$.  The functions $m,a,b$ are
		holomorphic in $q$: smoothness permits commuting the source derivatives
		with $\partial_{\bar q}$, and $\partial_{\bar q}\widehat\Phi=0$.
		
		Choose a family $G$ from Definition~\ref{def:rel-first-jet-spheres} at the
		vertical first jet represented by $(0,m(0),b(0))$.  In the relative chart,
		write the parameters of $G$ as $(t,x,\beta)$.  After shrinking $s_0$ and $d$, the
		triple
		\[
		\bigl(z,m(q)+a(q)z,b(q)\bigr)
		\]
		lies in a fixed compact subset $\mathcal L$ of the parameter domain of $G$
		whenever $|z|\leq s_0$ and $|q|\leq2d$. By continuity of $G$ and compactness
		of $\mathcal L$, there is $\epsilon_G>0$ such that $G(w;\lambda)$ lies in
		the chosen target chart for $|w|\leq\epsilon_G$ and $\lambda\in\mathcal L$.
		Decrease $s_0$ so that $s_0/2<\epsilon_G$. Then on the annulus used below,
		$|\sigma^2/z|\leq2\sigma<s_0/2$, so the coordinate expansion is defined.
		Set
		\begin{equation}\label{eq:rel-prepared-disc}
			H_\sigma(z,q)=
			G\left(\frac{\sigma^2}{z};
			z,m(q)+a(q)z,b(q)\right),
		\end{equation}
		where $z\mapsto\sigma^2/z$ is viewed as a holomorphic map to
		$\Pone$ taking $0$ to $\infty$.  For each fixed $\sigma>0$, the coordinate
		$w'=1/w$ at infinity gives $w'=z/\sigma^2$, an expression that is
		holomorphic at zero and independent of $q$.  Thus the formula
		in~\eqref{eq:rel-prepared-disc} extends holomorphically across $z=0$, and the
		identity $p(G(w;t,x,\beta))=t$ gives $p\circ H_\sigma=z$.
		In particular, $H_\sigma(0,q)=G(\infty;0,m(q),b(q))$.
		
		The prescribed first jet of $G$ gives a holomorphic map $\mathcal C$ such that, for
		$w$ near zero,
		\[
		\phi(G(w;t,x,\beta))=x+\beta w+w^2\mathcal C(w;t,x,\beta).
		\]
		Substitution into~\eqref{eq:rel-prepared-disc}
		gives~\eqref{eq:rel-prepared-expansion}, with
		\[
		R_\sigma(z,q)=\frac{\sigma^4}{z^2}
		\mathcal C\left(\frac{\sigma^2}{z};
		z,m(q)+a(q)z,b(q)\right).
		\]
		On $\sigma/2\leq|z|\leq2\sigma$ one has
		\[
		\left|\frac{\sigma^4}{z^2}\right|\leq4\sigma^2,\quad
		\left|\partial_z\frac{\sigma^4}{z^2}\right|\leq16\sigma,\quad
		\left|\partial_z\frac{\sigma^2}{z}\right|\leq4.
		\]  Cauchy estimates for $\mathcal C$ and the fixed
		functions $m,a,b$ give~\eqref{eq:rel-prepared-estimates}.  The uniform
		bound for $d_q H_\sigma$ follows from the parameter-derivative bound in
		Remark~\ref{rem:rel-first-jet-base-change}, since $\sigma^2/z$ is independent of
		$q$.  Differentiation in $q$ therefore introduces no factor from the rational source map $z\mapsto\sigma^2/z$; its
		$z$-derivatives need not be uniformly bounded near
		zero as $\sigma\to0$.  The local uniformity assertion follows by taking compact subsets of
		the chosen source, target, and parameter charts and applying the three displayed bounds and Cauchy estimates
		on those compact subsets.
	\end{proof}
	
	The value $z=0$ in this substitution corresponds to $w=\infty$.
	Thus the construction uses a jointly holomorphic family on all of
	$\Pone$, not merely germs near $w=0$. Invertibility of the parameter
	derivative is needed on the transition annulus, not on the disk core.
	
	\subsection{Uniform matching on annuli}
	\begin{lemma}
		\label{lem:rel-uniform-matching}
		Let $p:X\to R$ be a proper holomorphic submersion with the relative
		first-jet sphere-family property, let
		$P\subset R$ be compact, and let $V$ be a neighborhood of $P$.  Suppose
		that
		\[
		\Phi:V\times B_{2d_0}\longrightarrow X,
		\qquad p(\Phi(z,q))=z,
		\]
		is smooth, holomorphic in $q$, and
		\[
		d_q\Phi(z,0):\C^n\longrightarrow T_{X/R,\Phi(z,0)}
		\]
		is an isomorphism for every $z\in P$.  Then $P$ has a finite cover by
		open sets $V_\nu'\Subset V_\nu\Subset V$, equipped with source coordinates,
		relative target charts, and first-jet families as in
		Definition~\ref{def:rel-first-jet-spheres}, for which there are common constants
		\[
		d>0,\qquad \sigma_*>0,\qquad \mu>0,\qquad C>0
		\]
		with the following properties.
		
		Let $c\in V_\nu'$, put
		$\zeta=\zeta_\nu(z)-\zeta_\nu(c)$, and let
		$0<\sigma<\sigma_*$, with
		\[
		h_\sigma=\sigma^{3/2}<\frac{\sigma}{4}.
		\]
		Put
		\begin{equation}\label{eq:rel-disk-geometry}
			\begin{aligned}
				\Delta^+_{c,\sigma}&=\{|\zeta|<\sigma+h_\sigma\},\\
				\Delta^-_{c,\sigma}&=\{|\zeta|<\sigma-h_\sigma\},\\
				\mathcal T_{c,\sigma}
				&=\Delta^+_{c,\sigma}\setminus\overline{\Delta^-_{c,\sigma}}.
			\end{aligned}
		\end{equation}
		There is a jointly holomorphic prepared family $H_{c,\sigma}$ on the outer
		disk and, on $\mathcal T_{c,\sigma}\times B_d$, a smooth $X$-valued
		matched family $M_{c,\sigma}$ with the following properties:
		\begin{enumerate}[label=\textup{(\roman*)},leftmargin=2.2em]
			\item In the selected vertical target coordinate, write $\widehat\Phi=\phi_\nu\circ\Phi$ and
			$\widehat H_{c,\sigma}=\phi_\nu\circ H_{c,\sigma}$.  Then
			\[
			\Xi_{c,\sigma}
			=(1-\eta_{c,\sigma})\widehat\Phi
			+\eta_{c,\sigma}\widehat H_{c,\sigma},
			\qquad
			M_{c,\sigma}(z,q)
			=(p,\phi_\nu)^{-1}(z,\Xi_{c,\sigma}(z,q)),
			\]
			where
			\[
			\eta_{c,\sigma}(z)
			=\chi_0\!\left(\frac{|\zeta|-\sigma}{h_\sigma}\right),
			\]
			and $\chi_0\in C^\infty(\mathbb R,[0,1])$ is fixed, equal to one on
			$(-\infty,-1/2]$ and zero on $[1/2,\infty)$.  In particular, the displayed
			inverse target chart is defined for every indicated $(z,q)$.
			\item On the transition annulus,
			\begin{equation}\label{eq:rel-uniform-match}
				\|\widehat H_{c,\sigma}-\widehat\Phi\|_{C_q^3}
				\leq Ch_\sigma,
				\qquad
				\|(\partial_q\Xi_{c,\sigma})^{-1}\|\leq\mu^{-1}.
			\end{equation}
			The coefficient
			\begin{equation}\label{eq:rel-physical-coefficient}
				C^M_{c,\sigma}(z,q)
				=(\partial_q\Xi_{c,\sigma}(z,q))^{-1}
				\partial_{\bar\zeta}\Xi_{c,\sigma}(z,q)\,d\bar\zeta
			\end{equation}
			satisfies
			\begin{equation}\label{eq:rel-physical-bounds}
				\sup_{|q|\leq d}|\partial_q^jC^M_{c,\sigma}(z,q)|\leq C,
				\qquad 0\leq j\leq2,
			\end{equation}
			uniformly in $c$, $\nu$, and $\sigma$.  The notation $\partial_q^j$
			denotes the collection of all order-$j$ derivatives in $q$.
			\item On the open outer collar
			$\mathcal T_{c,\sigma}\cap\{|\zeta|>\sigma+h_\sigma/2\}$ the matched family
			equals $\Phi$ and~\eqref{eq:rel-physical-coefficient} equals
			\[
			C^\Phi=(\partial_q\widehat\Phi)^{-1}
			\partial_{\bar\zeta}\widehat\Phi\,d\bar\zeta.
			\]
			On the open inner collar
			$\mathcal T_{c,\sigma}\cap\{|\zeta|<\sigma-h_\sigma/2\}$ the family $M_{c,\sigma}$ equals $H_{c,\sigma}$ and
			$C^M_{c,\sigma}=0$.
		\end{enumerate}
		The inverse-matrix estimate is required only on the transition annulus.
		On the inner disk, reconstruction uses the joint holomorphy of
		$H_{c,\sigma}$ and does not require $d_qH_{c,\sigma}$ to be invertible.
		Only the $q$-derivatives in~\eqref{eq:rel-physical-bounds} are uniform as
		$\sigma\to0$; no uniform assertion is made about $z$-derivatives of the
		matched family or coefficient.
	\end{lemma}
	
	\begin{proof}
		Fix $c_0\in P$.  Choose a source coordinate near $c_0$ and a relative
		target chart whose image contains a product of the source neighborhood with
		a Euclidean ball about the vertical-coordinate image of $\Phi(c_0,0)$.
		Shrink twice to obtain nested convex balls
		\[
		P^0_{c_0}\Subset P^1_{c_0}\Subset P^2_{c_0}
		\]
		inside the vertical target-coordinate image.  Definition~\ref{def:rel-first-jet-spheres}
		at the antiholomorphic first jet of $\Phi(\,\cdot\,,0)$ at $c_0$ gives a
		first-jet family on a parameter neighborhood.  Shrink the source
		neighborhood and the $q$-ball so that, for every center in the smaller source neighborhood, the
		triples consisting of the moving base point, the value parameter, and the
		antiholomorphic first derivative lie in a relatively compact subdomain of
		that parameter neighborhood.  Also arrange that the coordinate image of
		$\Phi$ lies in $P^0_{c_0}$ on the smaller source and parameter sets, and
		that $d_q\Phi(z,0)$ remains invertible throughout the smaller source set.
		
		Compactness of $P$ gives finitely many such smaller source sets $V_\nu'$
		covering $P$, with closures in the corresponding $V_\nu$. The mixed
		derivatives of $\Phi$ through order two in the source and three in $q$
		have a common bound on the compact source and parameter sets just chosen.  The
		evaluation maps of the finitely many first-jet families have common bounds
		for the required parameter derivatives on the chosen relatively compact
		domains.  The construction and proof of Lemma~\ref{lem:rel-prepared}, after
		translation of the source coordinate to each center $c$, therefore give
		common $d,\sigma_*$ and $C$ and prepared families $H_{c,\sigma}$.  Relabel
		the three nested target-coordinate balls belonging to the $\nu$th retained
		choice of local data as $P^0_\nu\Subset P^1_\nu\Subset P^2_\nu$.
		Decrease $\sigma_*$ once more so that, for every
		$c\in V_\nu'$ and $0<\sigma<\sigma_*$, the closed outer disk
		$\overline{\Delta^+_{c,\sigma}}$ is contained in $V_\nu$.  A common outer-disk radius exists
		because only finitely many inclusions
		$\overline{V_\nu'}\Subset V_\nu$ occur.
		
		\noindent\textup{(i)} The finitely many inclusions
		$P^0_\nu\Subset P^1_\nu\Subset P^2_\nu$ have a common positive coordinate
		margin.  On the transition annulus the two Taylor expansions give
		\[
		\widehat\Phi-\widehat H_{c,\sigma}
		=b_c(q)\left(\bar\zeta-\frac{\sigma^2}{\zeta}\right)
		+R_c(\zeta,q)-R_{c,\sigma}(\zeta,q).
		\]
		Here
		\[
		\left|\bar\zeta-\frac{\sigma^2}{\zeta}\right|
		=\frac{\bigl||\zeta|^2-\sigma^2\bigr|}{|\zeta|}
		\leq h_\sigma\frac{2\sigma+h_\sigma}{\sigma-h_\sigma}
		\leq 3h_\sigma,
		\qquad \sigma^2\leq h_\sigma.
		\]
		The factor $\bar\zeta-\sigma^2/\zeta$ is independent of $q$.
		Combining the displayed bound with the uniform $C_q^3$ bounds for $b_c$
		and the $O_{C_q^3}(\sigma^2)$ remainders proves the first estimate in~\eqref{eq:rel-uniform-match}.  After decreasing
		$\sigma_*$, the prepared family lies in $P^1_\nu$ on the annulus.  Since
		$P^1_\nu$ is convex, the entire segment defining $\Xi_{c,\sigma}$ remains
		in $P^1_\nu\Subset P^2_\nu$.  Thus $M_{c,\sigma}$ takes values in $X$.
		
		\noindent\textup{(ii)} To prove the inverse bound in~\eqref{eq:rel-uniform-match}, use continuity and the assumed invertibility of $d_q\Phi(z,0)$ to obtain,
		after shrinking the common $q$-ball and source sets,
		\[
		\|(\partial_q\widehat\Phi(z,q))^{-1}\|\leq(2\mu)^{-1}
		\]
		for one $\mu>0$ on all chosen source and parameter sets.  The $C_q^1$ part of the first estimate in
		\eqref{eq:rel-uniform-match} makes
		$\|\partial_q\Xi_{c,\sigma}-\partial_q\widehat\Phi\|<\mu$
		after another decrease of $\sigma_*$.  For $L_0=\partial_q\widehat\Phi$ and $L_1=\partial_q\Xi_{c,\sigma}$,
		\[
		L_1=L_0\bigl(I+L_0^{-1}(L_1-L_0)\bigr),\qquad
		\|L_0^{-1}(L_1-L_0)\|<\tfrac12.
		\]
		Hence $\|L_1^{-1}\|\leq2\|L_0^{-1}\|\leq\mu^{-1}$, proving the
		second bound in~\eqref{eq:rel-uniform-match}.
		
		For the following calculation, suppress the fixed indices $c,\sigma$
		in $\eta,\Xi,\widehat H$.  One has $|\dbar\eta_{c,\sigma}|\leq Ch_\sigma^{-1}$.  Since
		$\widehat H$ is holomorphic, the numerator in the coefficient satisfies
		\[
		\partial_{\bar\zeta}\Xi
		=(1-\eta)\partial_{\bar\zeta}\widehat\Phi
		+(\partial_{\bar\zeta}\eta)(\widehat H-\widehat\Phi).
		\]
		The matching difference is $O_{C_q^3}(h_\sigma)$, so its product with the
		cutoff derivative is uniformly bounded through two $q$-derivatives.  Set
		$L_1=\partial_q\Xi$.  The identity
		\[
		\partial_{q_j}(L_1^{-1})
		=-L_1^{-1}(\partial_{q_j}L_1)L_1^{-1}
		\]
		and
		\[
		\begin{aligned}
			\partial_{q_k}\partial_{q_j}L_1^{-1}
			={}&L_1^{-1}(\partial_{q_k}L_1)L_1^{-1}(\partial_{q_j}L_1)L_1^{-1}\\
			&+L_1^{-1}(\partial_{q_j}L_1)L_1^{-1}(\partial_{q_k}L_1)L_1^{-1}\\
			&-L_1^{-1}(\partial_{q_j}\partial_{q_k}L_1)L_1^{-1}
		\end{aligned}
		\]
		bound the first two derivatives of $L_1^{-1}$ by the inverse bound and
		the derivatives of $\Xi$ through order three.  Applying the product
		rule to $L_1^{-1}\partial_{\bar\zeta}\Xi\,d\bar\zeta$ proves
		\eqref{eq:rel-physical-bounds}.
		
		\noindent\textup{(iii)} On the outer collar, $\eta=0$ gives $M_{c,\sigma}=\Phi$ and
		$C^M_{c,\sigma}=C^\Phi$.  On the inner collar, $\eta=1$ gives
		$M_{c,\sigma}=H_{c,\sigma}$ and $C^M_{c,\sigma}=0$, because
		$H_{c,\sigma}$ is jointly holomorphic.
	\end{proof}
	Lemma~\ref{lem:rel-uniform-matching} also applies to finitely many local families $\Phi_\alpha$
	expressed in different holomorphic bundle frames, provided their smaller center sets cover the compact set under consideration.  Apply the proof
	to each family and take minima of the finitely many radii and margins and
	maxima of the bounds.  In a global construction, assign each selected disk
	to one family.  Distinct disks are disjoint, and on every outer collar the
	coefficient agrees with the same global equation.  Thus no compatibility
	between different inner families, and no global holomorphic frame, is
	required.
	
	The compact center sets, finite charts, and bundle frames are fixed first;
	then choose the common parameter radius, target margin, and $\sigma_*$.
	Only afterwards choose the global scale $r$ and the finite disk family.
	For $h_i=\sigma_i^{3/2}$ and $\sigma_i\leq r$, the two roles of
	$h_\sigma=\sigma^{3/2}$ are
	\[
	O(h_\sigma^{-1})O_{C_q^3}(h_\sigma)=O_{C_q^2}(1),
	\qquad
	\sum_i\sigma_i h_i
	=\sum_i\sigma_i^2\sigma_i^{1/2}
	\leq r^{1/2}\sum_i\sigma_i^2.
	\]
	The first bound controls the matched coefficients; disjointness of the disks bounds $\sum_i\sigma_i^2$ and gives the area estimate~\eqref{eq:rel-area}.
	
	\subsection{Approximation on bordered surfaces}
	\begin{theorem}
		\label{thm:rel-bordered}
		Let $\pi:Z\to Y$ be a proper holomorphic submersion with connected
		fibers and the relative first-jet sphere-family property, let $R$ be an open Riemann surface, let $g:R\to Y$ be
		holomorphic, and let
		$p:X=R\times_Y Z\to R$ be the pullback.  Let $D\subset R$ be a compact
		bordered Riemann surface, let $K\subset\operatorname{int}D$ be compact, and
		let $A\subset\operatorname{int}D$ be finite.  Let $s$ be a smooth section of $p$ on
		a neighborhood of $D$, holomorphic on a neighborhood of $K\cup A$.
		Prescribe an integer $k_a\geq0$ for every $a\in A$.  For every $\epsilon>0$ there is a section
		$F$ of $p$ holomorphic on a neighborhood of $D$ such that
		\[
		\sup_K d_X(F,s)<\epsilon,
		\qquad j_a^{k_a}F=j_a^{k_a}s\quad(a\in A).
		\]
	\end{theorem}
	
	\begin{proof}
		The pullback $p$ has the relative first-jet sphere-family property by
		Remark~\ref{rem:rel-first-jet-base-change}.
		Properness of $p$ makes $p(X)$ closed in $R$, and the local projection
		form of a holomorphic submersion makes $p(X)$ open.
		The section $s$ is defined near the nonempty bordered surface $D$, so
		$p(X)\ne\varnothing$.
		Connectedness of $R$ therefore gives $p(X)=R$.
		If the fibers are zero-dimensional, every nonempty fiber is a finite
		discrete set and hence, by connectedness, a singleton.  The proper
		holomorphic submersion $p$ is then a bijective local biholomorphism,
		and hence a biholomorphism.
		The inverse $p^{-1}:R\to X$ is holomorphic, and every local section
		of $p$ equals the restriction of $p^{-1}$.  Taking $F=p^{-1}$ proves
		the zero-dimensional case.  We may therefore assume that the fiber dimension is
		positive.
		
		We proceed in seven steps.
		
		\medskip\noindent
		\emph{Step 1: embed a larger source neighborhood in a compact surface.}
		Choose connected compact bordered Riemann surfaces
		\[
		D\Subset\operatorname{int}D_0,
		\qquad D_0\Subset\operatorname{int}D_1,
		\]
		inside the neighborhood on which $s$ is defined.
		Lemma~\ref{lem:compact-neighborhood} embeds a neighborhood $W$ of $D_1$
		holomorphically into a compact Riemann surface $\Sigma$.  We identify $W$
		with its image. Write $g_\Sigma$ for the genus of $\Sigma$. Fix a smooth area form $\omega_\Sigma$ on $\Sigma$ and Hermitian metrics
		on the bundles used below.  All areas and Sobolev norms refer to these
		fixed choices; the finitely many source-coordinate area forms are uniformly
		equivalent to $\omega_\Sigma$.  By first replacing $W$ by a smaller connected precompact
		neighborhood of $D_1$, we may assume that $\overline W$ lies in the original domain
		of $s$.  Fix the vertical local addition of Lemma~\ref{lem:rel-addition}
		over $p^{-1}(\overline W)$.  All reconstructions of sections of $X$ below
		will be performed in $W$; auxiliary modifications will avoid the
		reconstruction neighborhood chosen in Step~4.
		
		Set $\Delta=\sum_{a\in A}(k_a+1)a$.  Lemma~\ref{lem:rel-twist} gives $E_{\Delta}$ and
		the operator~\eqref{eq:rel-twisted-form} on $W$.  By
		the deformation retraction onto a graph in Subsection~\ref{subsec:surface-tools}, $W$ deformation retracts onto a
		countable locally finite graph $\Gamma$.  Every complex vector bundle on
		$\Gamma$ is trivial: choose frames at the vertices and extend them along
		the edges, which is possible because $\operatorname{GL}_n(\C)$, with $n$ the bundle rank, is path
		connected; local finiteness makes the resulting frame continuous.
		Pullback along the deformation retraction therefore shows that every
		complex vector bundle on $W$ is topologically trivial.  By
		the Behnke--Stein theorem \cite[Corollary~26.8]{For81}, $W$ is a Stein Riemann surface, and
		Grauert's Oka principle \cite{Grauert58} makes $E_{\Delta}$ holomorphically trivial.  Fix
		a holomorphic trivialization of $E_{\Delta}$ on $W$.
		
		\medskip\noindent
		\emph{Step 2: extend the bundle to $\Sigma$ with vanishing first cohomology.}
		After shrinking $W$ slightly while retaining $D_1$, the complement of $\overline W$ in
		$\Sigma$ is nonempty.  Choose an effective divisor $Q$ on $\Sigma$ supported in
		$\Sigma\setminus\overline W$, of degree greater than $2g_\Sigma-2$, and put
		\[
		\widetilde E=\mathcal O_\Sigma(Q)^{\oplus n},
		\qquad n=\dim X_z.
		\]
		The canonical section of $\mathcal O_\Sigma(Q)$ is nonzero on $W$, and therefore trivializes $\mathcal O_\Sigma(Q)|_W$.  Using the chosen trivialization of
		$E_{\Delta}$, identify
		\[
		\widetilde E|_W\cong E_{\Delta}.
		\]
		Serre duality gives
		\[
		H^1(\Sigma,\mathcal O_\Sigma(Q))
		\cong H^0(\Sigma,K_\Sigma(-Q))^*=0,
		\]
		because $\deg K_\Sigma(-Q)=2g_\Sigma-2-\deg Q<0$.
		A nonzero holomorphic section of a line bundle has an effective zero
		divisor of degree equal to the bundle degree, so a negative-degree line
		bundle has no nonzero holomorphic section.  Hence
		\begin{equation}\label{eq:rel-H1-zero}
			H^1(\Sigma,\widetilde E)=0.
		\end{equation}
		
		Take a smooth cutoff $\chi$ supported in $W$ and equal to one on a
		neighborhood of $D_1$.  Under $\widetilde E|_W\cong E_{\Delta}$ from Step~2, extend the nonlinear
		coefficient by
		\[
		\widetilde A(z,v)=\chi(z)A_{\Delta}(z,v)
		\]
		on a fixed small fiber ball and set $\widetilde A=0$ outside $W$.  Compact support of $\chi$ makes
		$\widetilde A$ a
		smooth bundle-valued $(0,1)$-form, holomorphic in $v$, and
		\[
		\widetilde A_v(z,0)=0.
		\]
		On the compact surface, put
		\begin{equation}\label{eq:rel-global-operator}
			\widetilde{\mathscr S}(v)
			=\dbar_{\widetilde E}v+\widetilde A(z,v).
		\end{equation}
		The operator $\widetilde{\mathscr S}$ agrees with $\mathscr S_{\Delta}$ on a neighborhood of
		$D_1$, while outside that neighborhood $\widetilde{\mathscr S}=0$ is only
		an auxiliary equation.  Since $\widetilde A_v(z,0)=0$, the linearization
		$d\widetilde{\mathscr S}(0)$ is exactly $\dbar_{\widetilde E}$.
		
		By~\eqref{eq:rel-H1-zero}, fix a bounded right inverse $T$ of
		$\dbar_{\widetilde E}:W^{1,4}(\Sigma,\widetilde E)\to
		L^4_{0,1}(\Sigma,\widetilde E)$ as in the proof of
		Lemma~\ref{lem:fixed-bundle-solution}.  This choice precedes all disk
		selections.
		
		\medskip\noindent
		\emph{Step 3: prepare the rational disks.}
		Let
		\[
		e_0=\widetilde{\mathscr S}(0).
		\]
		The support of $e_0$ is disjoint from a fixed neighborhood of $K\cup A$.  Choose
		an open set $W_1$ such that
		\[
		D_1\Subset W_1,
		\qquad
		\overline W_1\Subset\{z\in W:\chi(z)=1\}.
		\]
		After decreasing the fixed neighborhood of $K\cup A$, the compact set
		\[
		P=\operatorname{supp}e_0\cap D_1
		\]
		is disjoint from $\operatorname{supp}\Delta$.  Use the holomorphic
		trivialization of $E_{\Delta}$ fixed in Step~1, and its identification with
		$\widetilde E$ over $W$, to regard $q$ as a fixed holomorphic bundle
		coordinate.  Compactness of $\overline W_1$ and the domain of the
		fixed local addition give $d_0>0$ such that, on $W_1$, the family
		\[
		\Phi(z,q)=\Psi(s(z),\iota_z q)
		\]
		is defined for every $q\in B_{2d_0}$.
		At $q=0$,
		\[
		d_q\Phi(z,0)=d_u\Psi(s(z),0)\circ\iota_z=\iota_z.
		\]
		The map $\iota_z$ is an isomorphism for $z\in P$, because
		$P\cap\operatorname{supp}\Delta=\varnothing$.  Lemma~\ref{lem:rel-uniform-matching}
		therefore supplies finitely many source and target charts, first-jet
		families, parameter domains, and common constants. The resulting
		holomorphic disk families are constructed from rational curves; we call
		the disks used in this part of the construction rational disks.
		Every such disk will have its center in $\{e_0\ne0\}\cap W_1$ and its
		closed outer disk contained in the corresponding source set $V_\nu$.  We use exactly the inner disk, transition annulus,
		cutoff, prepared family $H_{c,\sigma}$, target-valued matched family
		$M_{c,\sigma}$, and coefficient $C^M_{c,\sigma}$ from Lemma~\ref{lem:rel-uniform-matching}.
		
		The chain rule in the selected target coordinate gives, for every smooth
		$q(z)\in B_d$ on the transition annulus of a rational disk,
		\begin{equation}\label{eq:rel-factorization}
			\dbar\!\left(\phi_\nu\circ M_{c,\sigma}(z,q(z))\right)
			=\partial_q\Xi_{c,\sigma}(z,q(z))
			\bigl(\dbar q+C^M_{c,\sigma}(z,q(z))\bigr).
		\end{equation}
		The formula is an identity of classical forms for smooth $q$ and of
		distributional forms for $q\in W^{1,4}$ with values in $B_d$, by the
		Sobolev chain rule.  The target-coordinate inverse is
		legitimate by the positive chart margin in
		Lemma~\ref{lem:rel-uniform-matching}.
		
		\medskip\noindent
		\emph{Step 4: prepare auxiliary disks.}
		Choose once and for all an open neighborhood $N_0$ of $D$ with
		\[
		D\subset N_0,\qquad \overline N_0\Subset\operatorname{int}D_0.
		\]
		Reconstruction will be restricted to $N_0$. For the auxiliary disks,
		we modify only the equation on $\Sigma$ and define no $X$-valued family. The divisor $\Delta$ records the interpolation
		orders and is distinct from the bordered source region $D$.
		
		Since the sets $V_\nu'$ cover
		$P=\operatorname{supp}e_0\cap D_1$, the compact set
		\[
		P_0=\operatorname{supp}e_0\setminus\bigcup_\nu V_\nu'
		\]
		is disjoint from $D_1$, and hence from $\overline N_0$.  Cover
		$P_0$ by finitely many source-coordinate neighborhoods whose
		closures avoid $\overline N_0$, and choose a
		$\dbar_{\widetilde E}$-holomorphic frame on each.  In each such frame the
		auxiliary equation~\eqref{eq:rel-global-operator} is
		\[
		\dbar v+C(z,v)=0,
		\qquad C_v(z,0)=0.
		\]
		For a center $c$ in such a frame and a sufficiently small $\sigma$, use
		the three regions in~\eqref{eq:rel-disk-geometry}, the fixed cutoff
		$\eta_{c,\sigma}$, and set on the transition annulus
		\begin{equation}\label{eq:rel-artificial}
			C^\circ_{c,\sigma}(z,v)
			=(1-\eta_{c,\sigma}(z))C(z,v).
		\end{equation}
		Thus the coefficient is exactly $C$ on an open outer collar and exactly
		zero on an open inner collar.  More precisely, use the zero coefficient on
		$\{|\zeta|<\sigma-h_\sigma/2\}$, the coefficient
		$C^\circ_{c,\sigma}$ on $\mathcal T_{c,\sigma}$, and the original
		coefficient $C$ on $\{|\zeta|>\sigma+h_\sigma/2\}$, with all sets
		intersected with the assigned frame domain.  These three open sets cover
		that domain near the disk, and the formulas agree on both open overlaps.
		The three coefficient formulas therefore glue smoothly.  The finite choice of
		frames and the fact that the cutoff is independent of $v$ give uniform
		bounds through two fiber derivatives.  Since $C_v(z,0)=0$, one has $(C^\circ_{c,\sigma})_v(z,0)=0$,
		so the linearization of the modified auxiliary equation remains
		$\dbar_{\widetilde E}$.  Auxiliary disks carry no $X$-valued family.
		
		\medskip\noindent
		\emph{Step 5: select the disks.}
		If $e_0$ vanishes identically, then on a neighborhood of $D$ the identity
		$\widetilde{\mathscr S}(0)=\mathscr S_{\Delta}(0)=0$ says that $s$ is already
		holomorphic.  Taking $F=s$ proves the theorem.  We henceforth assume that
		$U=\{e_0\ne0\}$ is nonempty.
		
		Choose a finite smooth triangulation of $\Sigma$, whose existence follows
		from~\cite[Theorem~10.6]{Mun66}, and subdivide the triangulation so finely that every
		closed triangle meeting $\operatorname{supp}e_0$ is contained either in
		one smaller center set $V_\nu'$ inside $W_1$ or in one of
		the auxiliary frame neighborhoods chosen in Step~4.  The center sets from Step~3 and frame neighborhoods from Step~4 form a finite open cover of
		$\operatorname{supp}e_0$.  Compactness gives a positive Lebesgue number
		for this cover on a neighborhood of $\operatorname{supp}e_0$;
		subdivision with smaller mesh gives the stated containment. Denote these
		closed triangles by $T_1,\ldots,T_M$. Assign one such choice of local data
		and one of the two types to each $T_j$. The union of the triangle edges has area
		zero.  In particular, every triangle assigned the auxiliary type, and
		hence every auxiliary disk selected in that triangle, is disjoint from
		$\overline N_0$.
		
		Before choosing the disk scale, fix the parameter domain of the global
		equation.  Let $d_*>0$ be smaller than all coordinate-parameter radii used
		above.  Since the finitely many assigned frames are fixed on compact
		source sets, there is a constant $C_F\geq1$ such that the coordinate vector
		$q$ of any fiber vector $v$ satisfies $|q|\leq C_F|v|$ in each applicable
		frame.  Choose a fixed metric radius $\rho>0$ with
		$C_F\rho<d_*$, and also within the domain of $\widetilde A$.
		These choices depend only on the fixed local data, before $r$ and the disks
		are selected.
		
		Choose $0<r<1/16$, to be decreased in Step~6.  Then
		$h_\sigma=\sigma^{3/2}<\sigma/4$ whenever $0<\sigma\leq r$.
		For $1\leq j\leq M$ and $c\in U\cap\operatorname{int}T_j$, take as candidate sets all disks
		$\Delta^+_{c,\sigma}$ from~\eqref{eq:rel-disk-geometry} such that
		\[
		0<\sigma\leq r,\qquad
		\overline{\Delta^+_{c,\sigma}}
		\subset U\cap\operatorname{int}T_j,
		\]
		and such that the closed disk remains in the corresponding larger
		set $V_\nu$ or in the assigned auxiliary frame. For a disk assigned to
		Step~3, also require
		$\sigma<\sigma_*$ from Lemma~\ref{lem:rel-uniform-matching}.  Every
		condition is open at the center, so these candidates contain arbitrarily
		small disks and form a fine Vitali cover of $U$ away from the null set of
		triangle edges.
		
		In the assigned coordinate for $T_j$, restrict the fixed area measure to
		$T_j$ and extend it by zero to $\mathbb R^2$. This is a finite Radon measure.
		Apply the covering theorem~\cite[Section~1.5.2, Theorem~1.28]{EG15} to the
		closures of the candidate disks with this measure. It gives a countable
		pairwise disjoint family covering $U\cap\operatorname{int}T_j$ up to a null set.
		In each $T_j$, take a finite initial subfamily for which the remaining
		uncovered area is at most
		$r^{1/2}/M$.  The combined selection is finite, each closed outer disk
		lies in its assigned triangle interior, and the closed outer disks are
		pairwise disjoint. Index the selected disks by $i$, and let $\zeta_i$ be
		the assigned source coordinate centered at $c_i$. Moreover,
		\begin{equation}\label{eq:rel-uncovered}
			\operatorname{area}\left(
			U\setminus\bigcup_i\Delta^+_{c_i,\sigma_i}\right)
			\leq r^{1/2}.
		\end{equation}
		Write $h_i=\sigma_i^{3/2}$ and let $\mathcal T_r$ be the union of the
		transition annuli.  The actual outer radius is $\sigma_i+h_i$, which lies
		between $\sigma_i$ and $5\sigma_i/4$.  In each assigned coordinate, the outer disk has Euclidean area
		$\pi(\sigma_i+h_i)^2\geq\pi\sigma_i^2$, and the annulus has area
		\[
		\pi\bigl((\sigma_i+h_i)^2-(\sigma_i-h_i)^2\bigr)=4\pi\sigma_i h_i.
		\]
		Comparability with the fixed area form and disjointness imply
		$\sum_i\sigma_i^2\leq C\operatorname{area}(\Sigma)$.
		Absorbing $\operatorname{area}(\Sigma)$ into the fixed constant gives
		\begin{equation}\label{eq:rel-area}
			\sum_i\sigma_i^2\le C,
			\qquad
			\operatorname{area}(\mathcal T_r)
			\le C\sum_i\sigma_i h_i
			=C\sum_i\sigma_i^2\sigma_i^{1/2}
			\le Cr^{1/2},
		\end{equation}
		after enlarging $C$.  Since every closed outer disk is contained in $U$,
		no selected disk meets the fixed neighborhood of $K\cup A$ on which
		$e_0=0$.  For later gluing, put
		\[
		\mathcal I_i=
		\{z\in\Delta^+_{c_i,\sigma_i}:
		|\zeta_i|<\sigma_i-h_i/2\},
		\qquad
		\mathcal O_r=\Sigma\setminus\bigcup_i
		\{z\in\Delta^+_{c_i,\sigma_i}:
		|\zeta_i|\leq\sigma_i+h_i/2\}.
		\]
		The sets $\mathcal O_r$, $\mathcal I_i$, and the individual transition
		annuli form an open cover of $\Sigma$.  Every selected outer disk is
		contained in $\mathcal I_i\cup\mathcal T_{c_i,\sigma_i}$.
		
		Define the global modified coefficient as follows.  On the ball bundle
		$\widetilde E(\rho)=\{\xi\in\widetilde E:|\xi|<\rho\}$, use
		$\widetilde A$ from~\eqref{eq:rel-global-operator} over $\mathcal O_r$,
		use the zero coefficient over every $\mathcal I_i$, use
		$C^M_{c_i,\sigma_i}$ over the transition annulus of a rational disk in the
		fixed holomorphic frame from Step~1, and use
		$C^\circ_{c_i,\sigma_i}$ over an auxiliary transition annulus in its
		assigned $\dbar_{\widetilde E}$-holomorphic frame.  Call the bundle map
		obtained by gluing these formulas
		\[
		\widetilde A_r:\widetilde E(\rho)\longrightarrow
		\Lambda^{0,1}T^*\Sigma\otimes\widetilde E.
		\]
		On the outer collar of a rational disk,
		$C^M_{c_i,\sigma_i}=C^\Phi$ by
		Lemma~\ref{lem:rel-uniform-matching}.  Lemmas~\ref{lem:rel-normal-form} and~\ref{lem:rel-twist} show
		that $C^\Phi$ is exactly the coefficient $A_{\Delta}$ of the twisted equation;
		because its assigned source set lies where $\chi=1$, one has
		$C^\Phi=A_{\Delta}=\widetilde A$ on that outer collar. On its inner collar the coefficient is exactly
		zero.  The corresponding equalities for auxiliary disks follow from
		\eqref{eq:rel-artificial}.  Thus all local definitions agree on open
		collars.  Since the disks are disjoint and every modification equals the
		original coefficient near its outer boundary, $\widetilde A_r$ is a
		globally defined smooth bundle map.  Moreover,
		\begin{equation}\label{eq:rel-global-fiber-bounds}
			\sup_{\widetilde E(\rho)}
			\bigl(|\widetilde A_r|+|d_v\widetilde A_r|
			+|d_v^2\widetilde A_r|\bigr)\leq C,
		\end{equation}
		with $C$ independent of the number and radii of the disks.  For each fixed
		$r$, the source derivatives of $\widetilde A_r$ are bounded, by the open-collar
		gluing, but no source-derivative bound uniform in $r$ is asserted.
		
		At $v=0$, $\widetilde A_r(z,0)$ vanishes on every $\mathcal I_i$, while outside
		the selected outer disks $\widetilde A_r(z,0)$ equals the original coefficient.  Since every
		selected outer disk lies in $U$, the residual of the modified equation is
		therefore zero outside the measurable set
		\[
		\mathcal E_r=
		\left(U\setminus\bigcup_i\Delta^+_{c_i,\sigma_i}\right)
		\cup\mathcal T_r.
		\]
		Only $\operatorname{area}(\mathcal E_r)$ is estimated; no estimate
		for $\operatorname{area}(\overline{\mathcal E_r})$ is used.
		
		Set
		\[
		\mathcal X=W^{1,4}(\Sigma,\widetilde E),
		\quad
		\mathcal Y=L^4_{0,1}(\Sigma,\widetilde E).
		\]
		Sobolev embedding makes
		\[
		\mathcal B=\{v\in\mathcal X:\|v\|_{C^0}<\rho\}
		\]
		an open neighborhood of zero.  Put
		\begin{equation}\label{eq:rel-global-modified-operator}
			\mathscr S_r(v)=\dbar_{\widetilde E}v
			+\widetilde A_r(z,v(z)),
			\qquad \mathscr S_r:\mathcal B\longrightarrow\mathcal Y.
		\end{equation}
		The fiber bounds give the $C^2$ property and uniform nonlinear estimate
		of Lemma~\ref{lem:fixed-bundle-solution}.  No derivative of a grafted
		$X$-valued map occurs in~\eqref{eq:rel-global-modified-operator}.
		
		\medskip\noindent
		\emph{Step 6: uniform estimates and solution.}
		By Step~5, $\mathscr S_r(0)$ vanishes outside $\mathcal E_r$.  The residual vanishes on every $\mathcal I_i$ and on the unmodified
		part of $\{e_0=0\}$.  The pointwise norm of $\mathscr S_r(0)$ is uniformly bounded by
		\eqref{eq:rel-global-fiber-bounds}.  Hence
		\eqref{eq:rel-uncovered}--\eqref{eq:rel-area} give
		\[
		\|\mathscr S_r(0)\|_{\mathcal Y}^4
		\le C\bigl(r^{1/2}+\operatorname{area}(\mathcal T_r)\bigr)
		\le Cr^{1/2},
		\]
		so
		\begin{equation}\label{eq:rel-residual}
			\|\mathscr S_r(0)\|_{\mathcal Y}\le Cr^{1/8}.
		\end{equation}
		On $\mathcal O_r$ and every $\mathcal I_i$, the linearization at zero is
		$\dbar_{\widetilde E}$.  On auxiliary annuli, the linearization also equals $\dbar_{\widetilde E}$
		by~\eqref{eq:rel-artificial} and $C_v(z,0)=0$.  Consequently
		$d\mathscr S_r(0)-\dbar_{\widetilde E}$ is multiplication by a uniformly
		bounded zero-order coefficient vanishing outside the transition
		annuli of the rational disks.  For $\xi\in\mathcal X$,
		\[
		\|(d\mathscr S_r(0)-\dbar_{\widetilde E})\xi\|_{L^4}
		\leq C\operatorname{area}(\mathcal T_r)^{1/4}\|\xi\|_{C^0}.
		\]
		The Sobolev embedding $W^{1,4}\hookrightarrow C^0$ and
		\eqref{eq:rel-area} therefore give
		\begin{equation}\label{eq:rel-linear}
			\|d\mathscr S_r(0)-\dbar_{\widetilde E}\|_{\mathcal X\to\mathcal Y}
			\le Cr^{1/8}.
		\end{equation}
		Apply Lemma~\ref{lem:fixed-bundle-solution} to $\widetilde E$ and
		$\widetilde A_r$, using~\eqref{eq:rel-H1-zero},
		\eqref{eq:rel-global-fiber-bounds}, \eqref{eq:rel-residual}, and
		\eqref{eq:rel-linear}, with $\varepsilon_r=C_1r^{1/8}$ for a fixed
		sufficiently large $C_1$.  Lemma~\ref{lem:fixed-bundle-solution} gives a smooth solution $v_r$ satisfying
		\begin{equation}\label{eq:rel-solution-bound}
			\mathscr S_r(v_r)=0,
			\qquad \|v_r\|_{W^{1,4}}\leq Cr^{1/8}.
		\end{equation}
		For each selected disk, let $q_i(z)$ be the coordinate vector of $v_r(z)$
		in its assigned frame. Let $C_S$ be the norm of a fixed Sobolev embedding
		$\mathcal X\hookrightarrow C^0$. Then
		\[
		|q_i(z)|\leq C_F|v_r(z)|
		\leq C_FC_SCr^{1/8}<d_*
		\]
		after $r$ is decreased.  Thus every coordinate parameter lies in the
		ball used to construct its assigned disk.
		Smoothness uses the open-collar gluing for each fixed $r$, without uniform
		source-derivative estimates as $r\to0$.
		
		\medskip\noindent
		\emph{Step 7: reconstruct the section.}
		Put $N=N_0$.  By the choice made in Steps~4--5, $N$ meets no auxiliary
		disk.
		For each selected rational disk, put
		\[
		U_i^-=N\cap\mathcal I_i,\qquad
		U_i^0=N\cap\mathcal T_{c_i,\sigma_i}.
		\]
		Also set
		\[
		U^+=N\cap\mathcal O_r.
		\]
		The sets $U^+$, $U_i^-$, and $U_i^0$ cover $N$.  Let $F_r$ be the section given by
		\begin{equation}\label{eq:rel-reconstruction}
			F_r(z)=
			\begin{cases}
				H_{c_i,\sigma_i}(z,q_i(z)),
				&z\in U_i^-,\\
				M_{c_i,\sigma_i}(z,q_i(z)),
				&z\in U_i^0,\\
				\Psi(s(z),\iota_z v_r(z)),
				&z\in U^+
			\end{cases}.
		\end{equation}
		On
		$U_i^-\cap U_i^0$ the cutoff is identically one, so
		$M_{c_i,\sigma_i}=H_{c_i,\sigma_i}$ for every $q$.  On
		$U_i^0\cap U^+$ the cutoff is identically zero, so
		$M_{c_i,\sigma_i}=\Phi$.  Thus the formulas agree on open collars, not
		merely on boundary circles.  The three formulas give a smooth map $F_r:N\to X$.
		Each branch satisfies $p(F_r(z))=z$, so $F_r$ is a section of $p$.
		
		We check holomorphy on each type of open region.  On an outer
		region the definitions of $\mathscr S_s$ and $\mathscr S_{\Delta}$ give
		\[
		\dbar F_r
		=d_u\Psi(s,\iota v_r)\circ\iota\,
		\mathscr S_{\Delta}(v_r)=0,
		\]
		because $\mathscr S_r=\mathscr S_{\Delta}$ on $U^+$.  On a transition annulus,
		\eqref{eq:rel-factorization} and
		$\mathscr S_r(v_r)=0$ give
		\[
		\dbar(\phi_\nu\circ F_r)
		=\partial_q\Xi_{c_i,\sigma_i}(z,q_i(z))
		\bigl(\dbar q_i+
		C^M_{c_i,\sigma_i}(z,q_i(z))\bigr)=0.
		\]
		On an inner region the local equation is $\dbar q_i=0$.  Since
		$H_{c_i,\sigma_i}$ is jointly holomorphic,
		\[
		\dbar F_r=d_qH_{c_i,\sigma_i}(z,q_i(z))\,\dbar q_i=0.
		\]
		The core reconstruction uses only joint holomorphy of $H_{c_i,\sigma_i}$.  All three displayed $\dbar F_r=0$ identities are classical because
		$v_r$ is smooth.  These open regions cover $N$, and their formulas
		agree on the open collars, so the identities prove that $F_r$ is holomorphic on $N$.
		
		Let $V_{K,A}\subset N$ be a fixed neighborhood of $K\cup A$ on which
		$e_0=0$, as chosen before the disk selection.  Since every selected disk
		is contained in $U=\{e_0\ne0\}$, it is disjoint from $V_{K,A}$.
		On $V_{K,A}$,
		\[
		F_r(z)=\Psi(s(z),\iota v_r(z)).
		\]
		On a smaller compact neighborhood of $K$ inside $V_{K,A}$, the uniform derivative bound for $\Psi$,
		Sobolev embedding,
		and~\eqref{eq:rel-solution-bound} yield
		\[
		\sup_K d_X(F_r,s)\leq C_\Psi\sup_K\|\iota_z\|\,\|v_r\|_{C^0}
		\leq Cr^{1/8}.
		\]
		Choose $r$ so small that $Cr^{1/8}<\epsilon$.  Near $a\in A$, choose a
		local coordinate $z$ centered at $a$.  In compatible holomorphic frames, the
		local-addition parameter is
		\[
		\iota v_r=z^{k_a+1}v_r
		\]
		and $\Psi(s,u)=s+O(u)$ in target coordinates.  Since $F_r$ and $s$ are holomorphic near $a$, the difference
		$h(z)=\phi(F_r(z))-\phi(s(z))$ in a common target chart satisfies
		$|h(z)|\leq C|z|^{k_a+1}$.  Hence $h(z)/z^{k_a+1}$ is bounded and
		holomorphic on the punctured disk and extends across the origin.  The
		Taylor coefficients of $h$ through degree $k_a$ vanish, so $j_a^{k_a}F_r=j_a^{k_a}s$.
		This proves the theorem.
	\end{proof}
	
	The sets $K$ and $A$ need not contain one another. Either may be empty;
	when $K$ is empty the approximation condition is omitted.
	
	\subsection{Exhaustion and the lifting homotopy}
	\begin{proposition}
		\label{prop:rel-open-lifting}
		Let $\pi:Z\to Y$ be a surjective proper holomorphic submersion with
		connected and simply connected fibers and the relative first-jet
		sphere-family property.  Let $R$ be an open Riemann surface, let $K\subset R$ be
		compact, let $A\subset R$ be closed discrete, and let $g:R\to Y$
		be holomorphic.  Suppose that $f:R\to Z$ is a continuous lifting of $g$ and
		is holomorphic near $K\cup A$.  Given $\epsilon>0$ and
		$q:A\to\mathbb N$, there is a holomorphic lifting $F:R\to Z$ such that
		\[
		\sup_K d_Z(F,f)<\epsilon,
		\qquad j_a^{q(a)}F=j_a^{q(a)}f\quad(a\in A).
		\]
	\end{proposition}
	
	\begin{proof}
		Put $k_a=q(a)$.  Work in the pullback $p:X=R\times_Y Z\to R$, with the metric fixed in
		Section~\ref{sec:vertical-preparation}.  Let $s_f(z)=(z,f(z))$ be the
		section corresponding to $f$. The pullback has the sphere-family
		property by Remark~\ref{rem:rel-first-jet-base-change}. Choose a connected exhaustion by
		compact bordered Riemann surfaces
		\[
		C_1\Subset\operatorname{int}C_2\Subset\cdots,
		\qquad \bigcup_j C_j=R,
		\]
		whose boundaries avoid $A$, with $K\subset\operatorname{int}C_1$.
		The exhaustion statement in Subsection~\ref{subsec:surface-tools} supplies this exhaustion for an arbitrary
		compact $K$; neither the $C_j$ nor $K$ need be Runge.  At each step below,
		Theorem~\ref{thm:rel-bordered} likewise permits an arbitrary compact set for approximation.  Put
		$A_j=A\cap C_j$, which is finite.  Choose positive numbers $\epsilon_j$ with
		\[
		\sum_{j\ge1}\epsilon_j<\epsilon.
		\]
		Use $\epsilon_j$ as the error bound at the $j$th approximation step.
		
		We construct sections $F_j$, holomorphic on neighborhoods of $C_j$, such
		that
		\begin{equation}\label{eq:rel-exhaust-a}
			\sup_K d_X(F_1,s_f)<\epsilon_1,
			\qquad
			j_a^{k_a}F_j=j_a^{k_a}s_f\quad(a\in A_j),
		\end{equation}
		and
		\begin{equation}\label{eq:rel-exhaust-b}
			\sup_{C_j} d_X(F_{j+1},F_j)<\epsilon_{j+1}.
		\end{equation}
		For the first step, use the relative smoothing statement in Subsection~\ref{subsec:bundle-tools} to replace the given
		continuous section $s_f$ by a smooth section on a neighborhood of $C_1$, without
		changing $s_f$ near $K\cup A_1$, and apply Theorem~\ref{thm:rel-bordered}.
		
		Assume $F_j$ has been constructed.  Choose pairwise disjoint coordinate
		disks about the finitely many points of $A_{j+1}\setminus A_j$, with their
		closures in $\operatorname{int}C_{j+1}$ and disjoint from a neighborhood of
		$C_j$, on which the original section $s_f$ is holomorphic.  Such disks exist
		because the boundaries of the exhaustion avoid $A$.  Lemma~\ref{lem:rel-top-extension},
		applied to $C_j\subset C_{j+1}$ and these disks, extends the prescribed
		sections $F_j$ near $C_j$ and $s_f$ near the new interpolation points to a
		smooth section $s_{j+1}$ near $C_{j+1}$.  $s_{j+1}$ is holomorphic near
		$C_j\cup A_{j+1}$.  Apply Theorem~\ref{thm:rel-bordered} with approximation
		set $C_j$, interpolation set $A_{j+1}$, and error
		$\epsilon_{j+1}$.  The resulting section is $F_{j+1}$ and
		satisfies~\eqref{eq:rel-exhaust-a}--\eqref{eq:rel-exhaust-b}.
		
		For every fixed $m$ and all $n>\ell\ge m$, the nesting of the exhaustion
		and~\eqref{eq:rel-exhaust-b} give
		\[
		\sup_{C_m} d_X(F_n,F_\ell)
		\leq \sum_{j=\ell}^{n-1}\epsilon_{j+1}.
		\]
		Thus $(F_j)$ is uniformly Cauchy on every compact subset.  For fixed $m$,
		all later sections over $C_m$ take values in the compact set $p^{-1}(C_m)$,
		since $p$ is proper.  The compact metric space $p^{-1}(C_m)$ is complete.  Pointwise limits
		therefore define a map $F_\infty$, and the uniform convergence on each $C_m$
		makes $F_\infty$ continuous.  Passing to the limit in $p(F_j(z))=z$, for $z\in C_m$ and $j\geq m$,
		gives $p(F_\infty(z))=z$.  Since $\bigcup_mC_m=R$, the map
		$F_\infty:R\to X$ is a section.
		
		Fix $x\in R$ and choose a relative target chart $U$ about $F_\infty(x)$.  By
		continuity of $F_\infty$, there are a source coordinate disk $\Delta$ about $x$
		and an open set $U'\Subset U$ such that
		$F_\infty(\overline\Delta)\subset U'$.  The compact set
		$F_\infty(\overline\Delta)$ has positive distance from $X\setminus U$.
		Uniform convergence on $\overline\Delta$ therefore puts
		$F_j(\overline\Delta)$ in $U$ for all sufficiently large $j$.  In the
		vertical coordinates of $U$, the maps $F_j|_\Delta$ are holomorphic and
		converge uniformly to the coordinate expression of $F_\infty$.  The Weierstrass
		theorem proves that $F_\infty$ is holomorphic near $x$, and hence on $R$.
		
		Summing the estimates on
		$K$ in~\eqref{eq:rel-exhaust-a}--\eqref{eq:rel-exhaust-b} gives
		\[
		\sup_K d_X(F_\infty,s_f)<\sum_{j\ge1}\epsilon_j<\epsilon.
		\]
		If $a\in A$, choose $m$ with $a\in A_m$.  For every $j\geq m$, the jet
		$j_a^{k_a}F_j$ equals the prescribed jet $j_a^{k_a}s_f$.  In particular, $F_j(a)=s_f(a)$ for $j\geq m$, so $F_\infty(a)=s_f(a)$.
		Choose a small source disk about $a$ and a relative target chart containing the
		images of $s_f$, $F_\infty$, and every sufficiently late $F_j$ on that disk, as in
		the preceding paragraph.  In a source coordinate centered at $a$, choose a circle $|\zeta|=\tau$
		inside that disk.  For a vertical coordinate map $\phi$ and $0\leq k\leq k_a$,
		\[
		\partial_z^k(\phi\circ F_j)(0)
		=\frac{k!}{2\pi i}\int_{|\zeta|=\tau}
		\frac{\phi(F_j(\zeta))}{\zeta^{k+1}}\,d\zeta
		\longrightarrow \partial_z^k(\phi\circ F_\infty)(0).
		\]
		Thus the derivatives of $F_\infty$ agree with those of $s_f$ through the prescribed order.  This equality is independent of the
		chosen coordinates by the holomorphic chain rule, so
		$j_a^{k_a}F_\infty=j_a^{k_a}s_f$.  Put
		$F_Z=\operatorname{pr}_Z\circ F_\infty:R\to Z$.  Then
		\[
		\pi\circ F_Z=g,\qquad
		\sup_Kd_Z(F_Z,f)\leq\sup_Kd_X(F_\infty,s_f)<\epsilon,\qquad
		j_a^{k_a}F_Z=j_a^{k_a}f.
		\]
		The map $F_Z$ is the required holomorphic lifting.
	\end{proof}
	
	\begin{proof}[Proof of Theorem~\ref{thm:rel-analytic-criterion}]
		The map $\pi$ is a proper surjective holomorphic submersion. By
		Ehresmann's theorem \cite{Ehr51}, $\pi$ is a Serre fibration.  Given a lifting
		problem from an open Riemann surface,
		Proposition~\ref{prop:rel-open-lifting} supplies the holomorphic lifting with
		the required approximation and closed discrete jet interpolation.
		
		Every fiber is connected and simply connected by hypothesis.  The pullback $p:X=R\times_Y Z\to R$ is a locally
		trivial bundle with the same fibers.  Lemma~\ref{lem:rel-section-homotopy}
		gives a section homotopy $H:R\times[0,1]\to X$ between the
		sections associated with $f$ and $F$.  Set
		$f_t(z)=\operatorname{pr}_Z H(z,t)$.  Then
		\[
		f_0=f,\qquad f_1=F,\qquad \pi\circ f_t=g\quad(0\leq t\leq1).
		\]
		These identities prove the homotopy conclusion of the strengthened assertion.  Specializing to Runge $K$
		and finite $A\subset K$ gives all conditions in
		Definition~\ref{def:rel-oka1-map}, exactly the Oka-1 map condition
		of~\cite[Definition~7.7]{AF25}.
	\end{proof}
	\section{Algebraic preparation}
	\label{sec:relative-algebraic-oka1}
	
	The algebraic proof requires a fixed bundle with vanishing first
	cohomology and a smooth identification with the bundle carrying the section
	equation.  We first show how to make such an identification holomorphic near
	the prescribed data.  We then construct the required algebraic section by
	elementary transforms and smoothing of combs.
	
	\subsection{Bundle isomorphisms near prescribed sets}
	\label{subsec:protected-gauge}
	
	Lemma~\ref{lem:alg-natural-Ds} identifies the natural holomorphic structure
	along a holomorphic section.  We first construct a smooth bundle
	identification and then make it holomorphic near the prescribed sets.
	
	\begin{lemma}
		\label{lem:alg-homotopy-transport}
		Let $p:X\to\Sigma$ be a holomorphic map and let
		\[
		H:\Sigma\times[0,1]\longrightarrow X
		\]
		be a smooth homotopy through sections whose image lies in the submersion
		locus of $p$.  Here $H_t(z)=H(z,t)$ and $p\circ H_t=\operatorname{id}_\Sigma$.
		Then the complex vector bundles
		\[
		H_0^*T_{X/\Sigma}\quad\text{and}\quad H_1^*T_{X/\Sigma}
		\]
		are smoothly isomorphic.  If $H$ is independent of the homotopy parameter
		on a neighborhood of a set $Q\subset\Sigma$, the isomorphism may be chosen
		to be the natural identity on that neighborhood of $Q$.
	\end{lemma}
	
	\begin{proof}
		The vertical tangent bundle is a smooth complex vector bundle on a
		neighborhood of the image of $H$.  Choose a Hermitian connection on $T_{X/\Sigma}$ over that neighborhood.
		Parallel transport along $t\mapsto H(z,t)$ gives
		\[
		J_z:T_{X/\Sigma,H(z,0)}\longrightarrow T_{X/\Sigma,H(z,1)}.
		\]
		In local frames, transport solves a linear ordinary differential equation
		on $[0,1]$; the coefficients depend smoothly on $z$.  Thus $J_z$ is
		complex linear, invertible, and smooth in $z$.  If $H(z,t)$ is constant
		in $t$, the transport equation gives $J_z=\operatorname{id}$.
	\end{proof}
	
	Parallel transport gives a smooth bundle isomorphism, but the
	approximation argument requires holomorphy near the prescribed data.
	The next lemma achieves that local holomorphy while retaining the
	original isomorphism near a disjoint compact set.
	
	\begin{lemma}
		\label{lem:alg-protected-gauge}
		Let $V$ and $E$ be holomorphic vector bundles of the same rank over a compact
		Riemann surface $\Sigma$, and let $J_0:V\to E$ be a smooth complex
		vector-bundle isomorphism.  Let $P,Q\subset\Sigma$ be compact sets with
		disjoint neighborhoods.  Assume that $P$ is contained in a relatively
		compact open Riemann surface $W\subset\Sigma\setminus Q$ and that $J_0$
		is holomorphic on a neighborhood of $Q$.  Then there is a smooth bundle isomorphism
		$J:V\to E$ which is holomorphic on a neighborhood of $P\cup Q$ and agrees
		with $J_0$ outside a relatively compact open neighborhood of $P$ disjoint
		from $Q$.
	\end{lemma}
	
	\begin{proof}
		A smooth isomorphism on the compact base need not be holomorphic.  We
		construct a holomorphic isomorphism on the open subsurface $W$ and use
		a cutoff to retain $J_0$ near $Q$.  After shrinking $W$,
		choose
		\[
		P\subset W_0\Subset W\Subset\Sigma\setminus Q.
		\]
		Every component of $W$ is an open Riemann surface and retracts onto a graph.
		Hence every complex vector bundle on $W$ is topologically trivial, and Grauert's Oka principle \cite[Satz~I]{Grauert58} makes $V|_W$ and $E|_W$ holomorphically trivial.
		Put $n=\operatorname{rank}V=\operatorname{rank}E$.
		In fixed holomorphic trivializations, $J_0|_W$ is a smooth map
		\[
		g_0:W\longrightarrow\operatorname{GL}_n(\C).
		\]
		The complex Lie group $\operatorname{GL}_n(\C)$ is a homogeneous complex manifold and hence an Oka manifold.  The
		Oka principle for maps from the Stein manifold $W$
		\cite[Theorem~1.1]{Forstneric09} therefore gives a homotopy $g_t:W\to\operatorname{GL}_n(\C)$ from $g_0$ to a holomorphic map
		$g_1$.  Reparametrize the homotopy to be constant near both endpoints.
		Relative smooth approximation (Subsection~\ref{subsec:bundle-tools}), relative to
		both endpoints, applied to the trivial bundle with fiber $\operatorname{GL}_n(\C)$
		over a compact subsurface times $[0,1]$, makes $(z,t)\mapsto g_t(z)$
		smooth on the region used for the cutoff construction.  The relative
		condition preserves both $g_0$ and $g_1$.
		
		Choose $\chi\in C^\infty(\Sigma,[0,1])$ supported in $W$ and equal to one
		near $\overline W_0$.  On $W$ represent $J$ by
		\[
		z\longmapsto g_{\chi(z)}(z),
		\]
		and put $J=J_0$ outside $W$.  Near $\partial W$ the two definitions agree
		because $\chi=0$ and therefore $g_{\chi(z)}(z)=g_0(z)$.  The resulting $J$ is smooth and invertible, is represented by the holomorphic matrix $g_1$ near $P$, and equals $J_0$ near $Q$.
	\end{proof}
	
	\subsection{Positive elementary transforms}
	\label{subsec:algebraic-sections}
	
	Positive elementary transforms will remove the first-cohomology obstruction.
	Given a vector bundle $E$ of rank $r$ on a smooth projective curve $B$,
	distinct points $b_i\in B$, and lines $\xi_i\subset E_{b_i}$,
	its \emph{positive elementary transform} is the vector bundle $E'$ fitting into
	\[
	0\longrightarrow E\longrightarrow E'
	\longrightarrow\bigoplus_i\C_{b_i}\longrightarrow0
	\]
	with the indicated local pole directions.  Explicitly, if $x$ is a local
	parameter at $b_i$ and a frame $e_1,\ldots,e_r$ satisfies
	$\xi_i=\C e_1(b_i)$, the transformed sheaf is locally generated by
	$x^{-1}e_1,e_2,\ldots,e_r$.  Here $\C_{b_i}$ denotes a
	skyscraper sheaf of length one; the notation does not specify a canonical
	trivialization of its one-dimensional fiber.
	
	We now make the parameter space of elementary transforms explicit.  We use
	the convention that $\mathbb P_B(E)$ parametrizes one-dimensional
	subspaces of the fibers of $E$.  For $m\geq1$, put
	\begin{equation}\label{eq:alg-Hecke-space}
		\mathscr H_m(E)=
		\left\{
		((b_1,\xi_1),\ldots,(b_m,\xi_m))\in\mathbb P_B(E)^m:
		b_i\ne b_j\text{ for }i\ne j
		\right\}.
	\end{equation}
	$\mathscr H_m(E)$ is the ordered parameter space of simple positive elementary transforms.  If
	$O\subset B$ is Euclidean open, let $\mathscr H_m(E;O)$ denote the subset
	on which every $b_i$ belongs to $O$.
	
	\begin{lemma}
		\label{lem:alg-universal-transform}
		Let $B$ be a smooth projective curve, let $E$ be a vector bundle of
		positive rank on $B$, and fix an integer $m\geq1$.  The variety $\mathscr H_m(E)$ is smooth and
		irreducible, and there is a vector bundle
		\[
		\mathscr E_m^+\longrightarrow B\times\mathscr H_m(E)
		\]
		whose restriction over
		$\eta=((b_1,\xi_1),\ldots,(b_m,\xi_m))$ is the positive elementary
		transform $E_\eta^+$ of $E$ in the indicated directions.  Consequently,
		for every line bundle $M$ on $B$, the locus
		\begin{equation}\label{eq:alg-transform-vanishing-locus}
			\left\{\eta\in\mathscr H_m(E):
			H^1(B,E_\eta^+\otimes M)=0\right\}
		\end{equation}
		is Zariski open.
	\end{lemma}
	
	\begin{proof}
		Put $r=\operatorname{rank}E$ and $S=\mathscr H_m(E)$.  Let
		$\rho:B\times S\to B$ be projection and $\beta_i:S\to B$ the $i$th
		base-point map.  Write
		$\Gamma_i\subset B\times S$ for the graph of $\beta_i$.  The divisors
		$\Gamma_i$ are pairwise disjoint, and we put
		\[
		\mathscr D=\Gamma_1+\cdots+\Gamma_m.
		\]
		On $\Gamma_i\cong S$ there is a tautological line subbundle
		\[
		\mathscr L_i\subset \rho^*E|_{\Gamma_i}.
		\]
		If $\jmath_i:\Gamma_i\hookrightarrow B\times S$ is the inclusion, put
		\begin{equation}\label{eq:alg-universal-positive-transform}
			\mathscr E_m^+=\ker\!\left(
			\rho^*E(\mathscr D)\longrightarrow
			\bigoplus_{i=1}^m(\jmath_i)_*
			\frac{\rho^*E(\mathscr D)|_{\Gamma_i}}
			{\mathscr L_i\otimes\mathcal O(\mathscr D)|_{\Gamma_i}}
			\right).
		\end{equation}
		The arrow in~\eqref{eq:alg-universal-positive-transform} first restricts
		to each $\Gamma_i$ and then takes the indicated quotient.  The kernel requires the polar part at $\Gamma_i$ to lie in the
		chosen tautological line.
		
		We check local freeness, which also identifies the fibers.  Near a point of
		$\Gamma_i$, choose a local equation $x=0$ for $\Gamma_i$ and a frame
		$e_1,\ldots,e_r$ of $\rho^*E$ in which $\mathscr L_i$ is generated by
		$e_1$ along $\Gamma_i$.  Since the other graphs are disjoint from this
		neighborhood, the kernel in \eqref{eq:alg-universal-positive-transform} is
		freely generated by
		\[
		x^{-1}e_1,e_2,\ldots,e_r.
		\]
		Away from the graphs, $\mathscr E_m^+$ equals $\rho^*E$.  Thus $\mathscr E_m^+$ is a
		vector bundle.  The displayed local basis remains a basis after arbitrary
		base change on $S$; hence formation of the kernel commutes with restriction
		to every parameter fiber.  The restriction of $\mathscr E_m^+$ to $B\times\{\eta\}$ is therefore precisely the
		sheaf of rational sections of $E$ with at most a simple pole at $b_i$, whose
		polar direction belongs to $\xi_i$.
		
		The projective bundle $\mathbb P_B(E)$ is smooth and irreducible.  The nonempty open subset $\mathscr H_m(E)\subset\mathbb P_B(E)^m$
		is therefore smooth and irreducible.  Finally,
		$\mathscr E_m^+\otimes\rho^*M$ is locally free.  Since $B\times S\to S$ is smooth, the sheaf $\mathscr E_m^+\otimes\rho^*M$ is
		flat over $S$; the projection is also proper.  Cohomology semicontinuity therefore makes the
		function
		\[
		\eta\longmapsto h^1(B,E_\eta^+\otimes M)
		\]
		upper semicontinuous; equivalently, the locus
		\eqref{eq:alg-transform-vanishing-locus} is Zariski open
		\cite[Tag~0BDN]{Stacks}.
	\end{proof}
	
	The cohomology-vanishing construction is due to Graber--Harris--Starr
	\cite[Lemma~2.5]{GHS03}.  We give a universal-family construction and a
	collision argument that separates coincident elementary modifications.
	These make the construction's compatibility with localization in a Euclidean
	open set and with additional open conditions on the directions explicit.
	
	\begin{lemma}
		\label{lem:alg-localized-transform}
		Let $B$ be a smooth projective curve, let $E$ be a vector bundle of positive
		rank on $B$, let $O\subset B$ be a nonempty Euclidean open set, and
		let $N\geq0$ be an integer.  There are an integer $m\geq1$ and a nonempty Zariski-open subset
		$\mathscr U_m\subset\mathscr H_m(E)$ such that every
		$\eta\in\mathscr U_m$ satisfies
		\begin{equation}\label{eq:alg-uniform-transform-vanishing}
			H^1\bigl(B,E_\eta^+(-W)\bigr)=0
		\end{equation}
		for every effective divisor $W$ of degree $N$.  Moreover, if
		$\mathscr G\subset\mathscr H_m(E)$ is any Zariski-open dense subset, then
		\[
		\mathscr U_m\cap\mathscr G\cap\mathscr H_m(E;O)\ne\varnothing.
		\]
		In particular, the transform may be chosen with distinct support points in
		$O$ while satisfying any additional Zariski-open dense conditions on the
		ordered points and pole directions.
	\end{lemma}
	
	\begin{proof}
		We first obtain vanishing of the first cohomology after one fixed twist $-R_0$.
		Riemann--Roch will then give the assertion simultaneously for all $W$
		of degree $N$, since $\deg(R_0-W)$ will equal the genus of $B$.
		Write $g$ for the genus of $B$ and put $r=\operatorname{rank}E$.  Fix an effective divisor
		$R_0$ of degree $g+N$ and put
		\[
		F=E(-R_0).
		\]
		Tensoring by a line bundle canonically identifies the projective bundles of
		lines $\mathbb P_B(F)$ and $\mathbb P_B(E)$, and positive elementary
		transforms commute with tensoring: the local generators
		$x^{-1}e_1,e_2,\ldots,e_r$ are all tensored by the same local line-bundle
		generator.  Thus, for corresponding
		configurations $\eta$,
		\begin{equation}\label{eq:alg-transform-twist-compatibility}
			F_\eta^+\cong E_\eta^+(-R_0).
		\end{equation}
		
		We first exhibit a configuration supported in $O$ for which the right side
		of \eqref{eq:alg-transform-twist-compatibility} has vanishing first
		cohomology.  Choose $t\in O$, shrink a Zariski-open neighborhood
		$U\subset B$ of $t$ so that $F|_U$ has an algebraic frame
		$e_1,\ldots,e_r$, and choose a small Euclidean disk
		$\Delta\Subset U\cap O$ centered at $t$.  Since
		$\mathcal O_B(t)$ is ample, Serre vanishing gives an integer $\mu\geq1$
		such that
		\begin{equation}\label{eq:alg-collision-special-fiber}
			H^1\bigl(B,F(\mu t)\bigr)=0.
		\end{equation}
		Set $m=r\mu$.
		
		To allow the support points to coincide, use the parameter space
		\[
		Z=U^{r\mu},
		\qquad z=(z_{k,j})_{1\leq k\leq r,\ 1\leq j\leq\mu}.
		\]
		For each pair $(k,j)$ let $\Gamma_{k,j}\subset B\times Z$ be the graph of
		the coordinate map $z_{k,j}:Z\to U\hookrightarrow B$, and put
		\[
		\mathscr D_k=\sum_{j=1}^{\mu}\Gamma_{k,j}.
		\]
		Each graph is an effective Cartier divisor relative to $Z$.
		Locally on the smooth relative curve $B\times Z\to Z$, the equation
		for $\mathscr D_k$ is the product of the equations for $\Gamma_{k,j}$.
		After arbitrary base change on $Z$, each factor remains a non-zero-divisor,
		and hence so does their product, even when coordinates coincide;
		repeated factors record their multiplicities.
		All the $\mathscr D_k$ are supported in $U\times Z$.
		
		We construct a vector bundle $\mathscr F^+$ on $B\times Z$.  On
		$U\times Z$ set
		\begin{equation}\label{eq:alg-collision-family}
			\mathscr F^+|_{U\times Z}
			=\bigoplus_{k=1}^r\mathcal O_{U\times Z}(\mathscr D_k)e_k.
		\end{equation}
		On
		\[
		(B\times Z)\setminus\bigcup_{k,j}\Gamma_{k,j}
		\]
		use the pullback of $F$.  These two open sets cover $B\times Z$, since all
		of the graphs lie over $U$.  On $(U\times Z)\setminus\bigcup_{k,j}\Gamma_{k,j}$ the canonical section of
		$\mathcal O(\mathscr D_k)$ is invertible, and hence canonically identifies
		the $k$th summand of~\eqref{eq:alg-collision-family} with the line generated by $e_k$ in the pullback of $F$.
		The identifications glue \eqref{eq:alg-collision-family} to a vector bundle
		on $B\times Z$.  Since $B\times Z\to Z$ is smooth, the locally free sheaf
		$\mathscr F^+$ is flat over $Z$.
		
		Let $z_0\in Z$ be the point at which every coordinate
		$z_{k,j}$ equals $t$.  Then $\mathscr D_k|_{B\times\{z_0\}}=\mu t$ for
		every $k$.  On $U$, the restriction of the special fiber is
		\[
		\bigl(\mathscr F^+|_{B\times\{z_0\}}\bigr)|_U
		=\bigoplus_{k=1}^r\mathcal O_U(\mu t)e_k.
		\]
		The gluing with $F$ off $t$ is the canonical one used in~\eqref{eq:alg-collision-family}, so globally
		\[
		\mathscr F^+|_{B\times\{z_0\}}\cong F(\mu t).
		\]
		Write $\mathscr F_z^+=\mathscr F^+|_{B\times\{z\}}$.
		By \eqref{eq:alg-collision-special-fiber} and cohomology semicontinuity
		for the proper projection $B\times Z\to Z$,
		there is a Zariski-open neighborhood $Z^0$ of $z_0$ such that
		\begin{equation}\label{eq:alg-collision-nearby-vanishing}
			H^1(B,\mathscr F_z^+)=0
			\qquad(z\in Z^0).
		\end{equation}
		
		Let $Z^*\subset Z$ be the complement of all diagonals
		$z_{k,j}=z_{k',j'}$ for $(k,j)\ne(k',j')$.  $Z^*$ is a nonempty Zariski-open dense subset of the
		irreducible variety $Z=U^{r\mu}$.  The polydisk
		\[
		Z(\Delta)=\Delta^{r\mu}\subset Z^{\mathrm{an}}
		\]
		is a Euclidean-open neighborhood of $z_0$.  Hence
		\[
		Z^0\cap Z^*\cap Z(\Delta)\ne\varnothing.
		\]
		Indeed, $Z^0\cap Z(\Delta)$ is a nonempty Euclidean-open set containing $z_0$, whereas the finite union of diagonals
		$Z\setminus Z^*$ is a proper algebraic subset and has empty
		Euclidean interior.
		
		Choose $z\in Z^0\cap Z^*\cap Z(\Delta)$
		and order the $m=r\mu$ coordinates of $z$
		lexicographically by $(k,j)$.  The points $z_{k,j}$ are distinct and belong to $O$.  At the
		point $z_{k,j}$, formula \eqref{eq:alg-collision-family} permits a simple
		pole precisely in the line spanned by $e_k(z_{k,j})$.  Thus the ordered
		configuration
		\[
		\eta(z)=\bigl((z_{k,j},\C e_k(z_{k,j}))\bigr)_{k,j}
		\in\mathscr H_m(F;O)\cong\mathscr H_m(E;O)
		\]
		has
		\[
		\mathscr F_z^+\cong F_{\eta(z)}^+.
		\]
		Equations \eqref{eq:alg-transform-twist-compatibility} and
		\eqref{eq:alg-collision-nearby-vanishing} give
		\[
		H^1\bigl(B,E_{\eta(z)}^+(-R_0)\bigr)=0.
		\]
		By Lemma~\ref{lem:alg-universal-transform}, the locus
		\begin{equation}\label{eq:alg-U-m-definition}
			\mathscr U_m=
			\left\{\eta\in\mathscr H_m(E):
			H^1\bigl(B,E_\eta^+(-R_0)\bigr)=0\right\}
		\end{equation}
		is therefore a nonempty Zariski-open subset.
		
		We next pass from the single divisor $R_0$ to all effective divisors of
		degree $N$, keeping the same configuration $\eta\in\mathscr U_m$.
		Let $W$ be any such divisor.  The line bundle
		$\mathcal O_B(R_0-W)$ has degree $g$.  Writing $K_B$ for the canonical
		line bundle of $B$, Riemann--Roch gives
		\[
		h^0\bigl(B,\mathcal O_B(R_0-W)\bigr)
		-h^0\bigl(B,K_B\otimes\mathcal O_B(-R_0+W)\bigr)=1,
		\]
		so $\mathcal O_B(R_0-W)$ has a nonzero section.  Hence there is an effective divisor $A_W$
		with
		\[
		R_0-W\sim A_W.
		\]
		For every $\eta\in\mathscr U_m$ we have
		\[
		E_\eta^+(-W)\cong E_\eta^+(-R_0)(A_W).
		\]
		The exact sequence
		\[
		0\longrightarrow E_\eta^+(-R_0)
		\longrightarrow E_\eta^+(-R_0)(A_W)
		\longrightarrow\mathcal Q_{\eta,W}\longrightarrow0
		\]
		has torsion quotient $\mathcal Q_{\eta,W}$, so
		$H^1(B,\mathcal Q_{\eta,W})=0$.  The cohomology sequence gives a surjection
		\[
		H^1\bigl(B,E_\eta^+(-R_0)\bigr)
		\longrightarrow H^1\bigl(B,E_\eta^+(-W)\bigr)
		\longrightarrow0.
		\]
		The first group vanishes by \eqref{eq:alg-U-m-definition}.  This proves
		\eqref{eq:alg-uniform-transform-vanishing} for every $W$, with the same
		fixed configuration $\eta\in\mathscr U_m$.
		
		It remains to impose localization and any further generality conditions.
		The variety $\mathscr H_m(E)$ is irreducible, and
		$\mathscr H_m(E;O)$ is a nonempty Euclidean-open subset: choose $m$
		distinct points of $O$ and arbitrary lines over them.  A proper algebraic
		subset of an irreducible complex variety has empty Euclidean interior, so
		$\mathscr H_m(E;O)$ is Zariski dense.  If
		$\mathscr G\subset\mathscr H_m(E)$ is Zariski open and dense, then
		$\mathscr U_m\cap\mathscr G$ is a nonempty Zariski-open subset of
		$\mathscr H_m(E)$ and therefore meets $\mathscr H_m(E;O)$.  Thus $\mathscr U_m\cap\mathscr G\cap\mathscr H_m(E;O)\ne\varnothing$,
		which proves the final assertion.
	\end{proof}
	
	\subsection{Combs and sections with vanishing cohomology}
	
	Attaching a rational tooth realizes a positive elementary transform on
	the handle.  We also need to control the normal bundle on the tooth:
	a nonzero projection of the modification direction to the trivial summand
	gives the splitting below.
	
	\begin{lemma}
		\label{lem:alg-tooth-transform}
		Let $A$ be a globally generated vector bundle on $\Pone$, let $x\in\Pone$,
		and put $F=A\oplus\mathcal O_{\Pone}$.  If
		$\lambda\subset F_x$ is a line whose projection to
		$\mathcal O_{\Pone,x}$ is nonzero, then the positive elementary transform
		of $F$ at $x$ in the direction $\lambda$ is isomorphic to
		\[
		A\oplus\mathcal O_{\Pone}(x).
		\]
		In particular, if $A$ is ample, this transform is globally generated and
		has vanishing first cohomology.
	\end{lemma}
	
	\begin{proof}
		After scaling a generator, write $\lambda=\C(a_x,1)$ with $a_x\in A_x$.
		Global generation gives a section $a\in H^0(\Pone,A)$ with $a(x)=a_x$.
		The bundle automorphism
		\[
		(u,c)\longmapsto(u-ca,c)
		\]
		of $A\oplus\mathcal O_{\Pone}$ has inverse $(u,c)\mapsto(u+ca,c)$
		and sends $(a_x,1)$ to $(0,1)$.  Positive elementary transforms are functorial under
		bundle automorphisms, and the transform in the direction $\C(0,1)$ is
		$A\oplus\mathcal O_{\Pone}(x)$.  If $A$ is ample, the splitting theorem gives
		$A\cong\bigoplus_\nu\mathcal O_{\Pone}(a_\nu)$ with $a_\nu\geq1$.
		Each summand and $\mathcal O_{\Pone}(x)\cong\mathcal O_{\Pone}(1)$
		is globally generated and has zero $H^1$, proving the last assertion.
	\end{proof}
	
	The handle and teeth meet along a tree.  Surjective evaluation at
	each node allows componentwise global generation and cohomology
	vanishing to pass to the whole curve.
	
	\begin{lemma}
		\label{lem:alg-tree-bundle}
		Let $C$ be a connected nodal curve whose irreducible components are smooth
		and whose dual graph is a tree.  Let $E$ be a vector bundle on $C$.  If, for
		every irreducible component $C_\alpha\subset C$, the bundle
		$E|_{C_\alpha}$ is globally generated and
		\[
		H^1(C_\alpha,E|_{C_\alpha})=0,
		\]
		then $E$ is globally generated and $H^1(C,E)=0$.
	\end{lemma}
	
	\begin{proof}
		We argue by induction on the number of irreducible components.  The
		one-component case is the hypothesis.  Otherwise choose a leaf component
		$L$ of the dual tree, let $C'$ be the closure of $C\setminus L$, and write
		$x=L\cap C'$.  Let $j_{C'}:C'\hookrightarrow C$ and $j_L:L\hookrightarrow C$ be the
		inclusions, and let $E_x$ also denote the skyscraper sheaf with that fiber.
		The normalization sequence at $x$ is
		\[
		0\longrightarrow E\longrightarrow (j_{C'})_*(E|_{C'})\oplus(j_L)_*(E|_L)
		\longrightarrow E_x\longrightarrow0,
		\]
		where the last map is the difference of the two evaluations at $x$.  By
		induction, $E|_{C'}$ is globally generated and has vanishing first
		cohomology; global generation and $H^1(L,E|_L)=0$ are hypotheses for $L$.  The map on global sections to
		$E_x$ is surjective, since either component has surjective evaluation at
		$x$.  The exact segment
		\[
		H^0(C',E|_{C'})\oplus H^0(L,E|_L)\longrightarrow E_x
		\longrightarrow H^1(C,E)\longrightarrow0
		\]
		therefore gives $H^1(C,E)=0$.
		
		To prove global generation at $y\in C'\setminus\{x\}$, prescribe
		$v\in E_y$.  Choose $s'\in H^0(C',E|_{C'})$ with $s'(y)=v$ and
		$s_L\in H^0(L,E|_L)$ with $s_L(x)=s'(x)$.  The normalization sequence
		glues $(s',s_L)$ to a section of $E$ with value $v$ at $y$.
		For $y\in L\setminus\{x\}$, first prescribe the section on $L$ and then
		match on $C'$.  For $y=x$, prescribe the same value $v\in E_x$ on both
		components and glue.  These cases cover $C$ and complete the induction.
	\end{proof}
	
	For the comb, these normal-bundle conditions must
	persist under smoothing with a prescribed line-bundle twist.
	The degree-one projection then identifies the smoothed curve with the
	base and converts its normal bundle into a vertical tangent bundle.
	
	\begin{lemma}
		\label{lem:alg-normal-smoothing}
		Let $Y$ be a smooth projective variety, let $B$ be a smooth projective
		curve, and let $p:Y\to B$ be a morphism.  Let $C\subset Y$ be a connected
		nodal curve which is a local complete intersection in $Y$, and write $q=p|_C$.  Assume that
		\[
		q_*[C]=[B],
		\]
		that the normal bundle $N_{C/Y}$ is globally generated, and that
		\[
		H^1(C,N_{C/Y})=0.
		\]
		Let $M$ be a line bundle on $B$ and suppose in addition that
		\[
		H^1\bigl(C,N_{C/Y}\otimes q^*M\bigr)=0.
		\]
		Then $C$ has an embedded smoothing $C_t\subset Y$ over a pointed
		smooth curve $(T,0)$, with $C_0=C$, such that, for all
		sufficiently small nonzero $t$,
		\[
		H^1\bigl(C_t,N_{C_t/Y}\otimes(p|_{C_t})^*M\bigr)=0.
		\]
		The curve $C_t$ is connected and smooth, the morphism
		$p|_{C_t}:C_t\to B$ has degree one, and hence $p|_{C_t}$ is an isomorphism.  If
		$h_t:B\to Y$ is the resulting section, then
		\[
		N_{C_t/Y}\cong h_t^*T_{Y/B}.
		\]
	\end{lemma}
	
	\begin{proof}
		Global generation supplies a direction smoothing every node, and the
		untwisted $H^1$ vanishing makes the chosen smoothing direction unobstructed.  The
		twisted vanishing then persists by semicontinuity, while $q_*[C]=[B]$
		recovers a section from the smooth curve.
		
		At every node $x$ of $C$, the local embedded deformation space has a
		one-dimensional smoothing quotient
		\[
		N_{C/Y}|_x\longrightarrow T^1_{C,x}\cong\C.
		\]
		Here $T^1_{C,x}$ is the space of first-order deformations of the node.
		The map $N_{C/Y}|_x\to T^1_{C,x}$ is onto: in local coordinates in which the two branches are the
		coordinate axes, the deformation $uv=t$ smooths the node.  For each node $x$, compose evaluation with the smoothing quotient to obtain
		a surjection $\lambda_x:H^0(C,N_{C/Y})\to T^1_{C,x}\cong\C$.
		Global generation ensures surjectivity of evaluation.  Since the node set
		is finite, one can choose
		\[
		\sigma\in H^0(C,N_{C/Y})\setminus
		\bigcup_{x\in\operatorname{Sing}C}\ker\lambda_x.
		\]
		Then $\lambda_x(\sigma)\ne0$ at every node, so $\sigma$ smooths all nodes
		to first order.
		
		Since $C$ is a local complete intersection in the smooth variety $Y$,
		the embedded first-order deformations of $C$ have tangent space $H^0(C,N_{C/Y})$ and
		obstructions in $H^1(C,N_{C/Y})$ by
		\cite[Theorem~6.2(b) and Remark~6.2.1]{Hartshorne10}.
		The local lifting hypothesis in that theorem holds because $C$ is a local complete
		intersection in a smooth ambient variety.  Since $H^1(C,N_{C/Y})=0$, the
		Hilbert scheme is smooth at $[C]$, and the tangent vector $\sigma$
		integrates to a one-parameter embedded deformation
		\[
		\mathcal C\subset Y\times T
		\]
		over a smooth pointed curve $(T,0)$.  At each original node $x$, a local smoothing parameter $a_x(t)$
		satisfies $a_x(0)=0$ and $a'_x(0)=\lambda_x(\sigma)\ne0$.
		Thus $a_x(t)\ne0$ for small $t\ne0$.  Away from node neighborhoods,
		openness of smoothness and compactness of $C$ exclude new singularities
		after shrinking $T$.  Moreover,
		$h^0(C,\mathcal O_C)=1$.  Upper semicontinuity gives
		$h^0(C_t,\mathcal O_{C_t})\leq1$, while constants give the opposite
		inequality; hence every sufficiently small nonzero fiber $C_t$ is connected and smooth.
		
		After another shrinking, every fiber $C_t\subset Y$ is a local complete
		intersection.  Since $\mathcal C\to T$ is flat and locally of finite
		presentation and $Y\times T\to T$ is smooth, the fiberwise criterion for
		relative regular immersions shows that
		\[
		\mathcal C\hookrightarrow Y\times T
		\]
		is a regular immersion which remains regular after arbitrary base change on $T$
		\cite[Tags~063U and~063W]{Stacks}.  Consequently
		\[
		\mathcal N=N_{\mathcal C/(Y\times T)}
		\]
		is locally free.  Write $\mathcal I$ for the ideal of $\mathcal C$ in $Y\times T$
		and $I_t$ for the ideal of $C_t$ in $Y$.  Then
		\[
		(\mathcal I/\mathcal I^2)|_{C_t}\cong I_t/I_t^2.
		\]
		Dualizing these locally free conormal sheaves gives
		\[
		\mathcal N|_{C_t}\cong N_{C_t/Y}
		\]
		for every $t$.  Put
		\[
		\mathcal M=(p\circ\operatorname{pr}_Y|_{\mathcal C})^*M.
		\]
		The sheaf $\mathcal N\otimes\mathcal M$ is locally free on the flat family
		$\mathcal C\to T$, hence flat over $T$.  Upper semicontinuity and the
		assumed vanishing on $C=C_0$ give
		\[
		H^1\bigl(C_t,N_{C_t/Y}\otimes(p|_{C_t})^*M\bigr)=0
		\]
		for all sufficiently small $t$.
		
		Finally, let $L$ be an ample line bundle on $B$.  The pullback of $L$ and the
		structure sheaf on the proper flat family $\mathcal C\to T$ are flat over
		$T$.  After shrinking to a connected neighborhood of $0$, both Euler
		characteristics are constant \cite[Tag~0B9T]{Stacks}.  Riemann--Roch gives
		\[
		\deg((p|_{C_t})^*L)
		=\chi\bigl(C_t,(p|_{C_t})^*L\bigr)-\chi(C_t,\mathcal O_{C_t}).
		\]
		Thus $\deg((p|_{C_t})^*L)$ is constant and equals $\deg L>0$,
		because $q_*[C]=[B]$.  A constant map has pullback degree zero, so
		$p_t=p|_{C_t}$ is nonconstant.  For a nonconstant map of projective curves,
		\[
		\deg(p_t^*L)=\deg(p_t)\deg L.
		\]
		It follows that $\deg(p_t)=1$.  A degree-one morphism between smooth connected
		projective curves is an isomorphism.  Let $h_t:B\to Y$ be the inverse of $p_t$ followed by the inclusion
		$C_t\hookrightarrow Y$.  Then $dp\circ dh_t=\id_{T_B}$, so $p$ is a
		submersion along $h_t(B)$.  The map
		\[
		\Pi:u\longmapsto u-dh_t(dp(u))
		\]
		projects $h_t^*TY$ onto $h_t^*T_{Y/B}$ with kernel $dh_t(TB)$.
		The map $\Pi$ identifies the normal quotient $h_t^*TY/dh_t(TB)$ with
		$h_t^*T_{Y/B}$, proving the normal-bundle assertion.
	\end{proof}
	
	We can now combine localized elementary transforms with the
	normal-bundle calculations.  Attaching teeth over a chosen Euclidean
	open set and smoothing the comb produces the required section with
	vanishing twisted first cohomology.
	For a smooth proper variety over an algebraically closed field of
	characteristic zero, rational connectedness is understood in the
	equivalent senses of \cite[Definition--Theorem~2.1]{Kol00}.
	
	\begin{proposition}
		\label{prop:alg-localized-free-section}
		Let
		\[
		p:\mathcal X\longrightarrow B
		\]
		be a projective morphism from a smooth projective complex variety to a
		smooth projective curve.  Assume that the geometric generic fiber is smooth
		and rationally connected.  Let $R\subset B$ be an open set over which $p$
		is smooth with rationally connected fibers.  Given an effective divisor
		$D$ supported in $R$ and a nonempty Euclidean open set
		\[
		O\Subset R\setminus\operatorname{supp}D,
		\]
		there is an algebraic section $h:B\to\mathcal X$ such that
		\[
		H^1\bigl(B,h^*T_{\mathcal X/B}(-D)\bigr)=0.
		\]
		The comb construction can be localized over $O$: after the
		stabilization $\mathcal X\times\mathbb P^3\to B$, every tooth may be attached
		in a fiber over a point of $O$.
	\end{proposition}
	
	\begin{proof}
		The theorem of Graber--Harris--Starr~\cite[Theorem~1.1]{GHS03} gives a
		section $h_0:B\to\mathcal X$.  Since $dp\circ dh_0=\id_{T_B}$,
		$p$ is a submersion along $h_0(B)$, even over a singular fiber.  Thus
		$h_0^*T_{\mathcal X/B}$ is a vector bundle on all of $B$.  For embedded
		teeth in every relative dimension, put
		\[
		\widetilde{\mathcal X}=\mathcal X\times\mathbb P^3,
		\qquad
		\widetilde p=p\circ\operatorname{pr}_{\mathcal X},
		\]
		choose $y_0\in\mathbb P^3$, and let
		\[
		\widetilde h_0=(h_0,y_0):B\longrightarrow\widetilde{\mathcal X}.
		\]
		The projective factor provides nonzero projected attachment directions and
		explicit embedded teeth, even if their projections to $\mathcal X$ are
		not immersed.  Each tooth meets the handle, whose projective component is
		constant, only at its attaching point.  Finally, the direct-sum decomposition
		of the relative tangent bundle transfers the required $H^1$ vanishing to the projected
		section $h$.  The fibers of
		$\widetilde p$ over $R$ are smooth rationally connected varieties of
		dimension at least three.  Along $\widetilde h_0(B)$, the splitting induced by
		$d\widetilde p\circ d\widetilde h_0=\operatorname{id}_{TB}$ gives
		\[
		\widetilde E:=N_{\widetilde h_0(B)/\widetilde{\mathcal X}}
		\cong h_0^*T_{\mathcal X/B}
		\oplus\bigl(T_{\mathbb P^3,y_0}\otimes\mathcal O_B\bigr).
		\]
		
		Apply Lemma~\ref{lem:alg-localized-transform} to $\widetilde E$ with
		$N=\deg D+1$.  For the integer $m$ supplied by Lemma~\ref{lem:alg-localized-transform}, let
		$\mathscr G\subset\mathscr H_m(\widetilde E)$ be the locus on which every
		chosen line has nonzero projection to the summand
		$T_{\mathbb P^3,y_0}$.  The locus $\mathscr G$ is Zariski open and dense: its complement is
		the union, over the $m$ entries, of the loci where the chosen line is
		contained in $h_0^*T_{\mathcal X/B}$.  Apply the final assertion of Lemma~\ref{lem:alg-localized-transform}
		with this choice of $\mathscr G$.  We obtain points $b_i\in O$, tangent directions
		$\xi_i\subset\widetilde E_{b_i}$, and a positive elementary transform
		$\widetilde E'$ such that
		\[
		H^1\bigl(B,\widetilde E'(-D-b)\bigr)=0
		\qquad(b\in B).
		\]
		For each $b\in B$, twisting upward by $b$, by $D$, or by $D+b$
		gives
		\[
		H^1(B,\widetilde E'(-D))=0,
		\qquad
		H^1(B,\widetilde E'(-b))=0,
		\qquad
		H^1(B,\widetilde E')=0.
		\]
		For either $F=\widetilde E'(-D)$ or $F=\widetilde E'$, denote the
		skyscraper sheaf with fiber $F|_b$ by $F_b$.  The exact sequence
		\[
		0\longrightarrow F(-b)\longrightarrow F\longrightarrow F_b\longrightarrow0
		\]
		and $H^1(B,F(-b))=0$ imply that $H^0(B,F)\to F_b$ is onto for every
		$b\in B$.  Thus both bundles are globally generated.
		
		The projective-space factor gives explicit embedded teeth.  Identify
		\[
		\widetilde E_{b_i}
		=T_{\mathcal X_{b_i},h_0(b_i)}
		\oplus T_{\mathbb P^3,y_0}.
		\]
		Let $\ell_i\subset T_{\mathbb P^3,y_0}$ be the projection of $\xi_i$;
		$\ell_i$ is a line because $\xi_i$ has nonzero projective-space projection.  Let
		\[
		\nu_i:\Pone\hookrightarrow\mathbb P^3
		\]
		be the projective line through $y_0$ with tangent line $\ell_i$, and choose
		its parameter so that $\nu_i(0)=y_0$.  Put $w_i=\nu_i'(0)$.  The projection maps $\xi_i$ isomorphically onto
		$\ell_i$, so there is a unique $v_i\in T_{\mathcal X_{b_i},h_0(b_i)}$ such that $(v_i,w_i)\in\xi_i$.
		Theorem~\ref{thm:rc-facts}(i), applied in the smooth rationally connected
		fiber $\mathcal X_{b_i}$, gives a morphism
		\[
		u_i:\mathbb P^1\longrightarrow\mathcal X_{b_i}
		\]
		with $u_i(0)=h_0(b_i)$, $u_i'(0)=v_i$, and
		\[
		H^1\bigl(\mathbb P^1,u_i^*T_{\mathcal X_{b_i}}(-2[0])\bigr)=0.
		\]
		Write $u_i^*T_{\mathcal X_{b_i}}\cong\bigoplus_\nu\mathcal O(a_\nu)$.
		Since $h^1(\Pone,\mathcal O(a_\nu-2))=\max\{1-a_\nu,0\}$, the displayed
		vanishing implies $a_\nu\geq1$ for every summand.
		Put
		\[
		T_i=(u_i,\nu_i)(\mathbb P^1)
		\subset\mathcal X_{b_i}\times\mathbb P^3.
		\]
		The map $(u_i,\nu_i)$ is an embedding because $\nu_i$ is an embedding,
		and its tangent line at the attachment point is $\xi_i$.  Moreover,
		\[
		(u_i,\nu_i)^*T_{\widetilde{\mathcal X}_{b_i}}
		=u_i^*T_{\mathcal X_{b_i}}\oplus\nu_i^*T_{\mathbb P^3}
		\]
		is ample: the first summand is ample by $a_\nu\geq1$, and
		$\nu_i^*T_{\mathbb P^3}\cong\mathcal O(2)\oplus\mathcal O(1)^{\oplus2}$.
		Hence the quotient
		$N_{T_i/\widetilde{\mathcal X}_{b_i}}$ is ample.  Since the points
		$b_i$ are distinct, the teeth are pairwise disjoint.  The handle has fixed
		$\mathbb P^3$ coordinate $y_0$, and the embedding $\nu_i$ satisfies
		$\nu_i^{-1}(y_0)=\{0\}$.  Thus each tooth meets the handle at exactly its
		attachment point.  The two tangent lines there are distinct because the tooth's
		tangent is vertical and the handle maps isomorphically to $B$.  Hence
		\[
		C=\widetilde h_0(B)\cup T_1\cup\cdots\cup T_m
		\subset\widetilde{\mathcal X}
		\]
		is a connected nodal curve which is a local complete intersection in
		$\widetilde{\mathcal X}$ and whose dual graph is a tree.  Write
		$N_C=N_{C/\widetilde{\mathcal X}}$ and $q=\widetilde p|_C$.
		
		The normal-sheaf calculation of \cite[Lemma~2.6]{GHS03} identifies
		\[
		N_C\otimes\mathcal O_{\widetilde h_0(B)}\cong\widetilde E'.
		\]
		On a tooth $T_i$, the exact sequence
		\[
		0\longrightarrow N_{T_i/\widetilde{\mathcal X}_{b_i}}
		\longrightarrow N_{T_i/\widetilde{\mathcal X}}
		\longrightarrow\mathcal O_{T_i}\longrightarrow0
		\]
		splits because
		\[
		\operatorname{Ext}^1(\mathcal O_{T_i},
		N_{T_i/\widetilde{\mathcal X}_{b_i}})
		=H^1(T_i,N_{T_i/\widetilde{\mathcal X}_{b_i}})=0;
		\]
		the last vanishing follows from ampleness on $T_i\cong\Pone$.  By
		\cite[Proposition~23]{HT06}, applied with ambient variety
		$\widetilde{\mathcal X}$, smooth divisor $\widetilde{\mathcal X}_{b_i}$,
		curve $T_i$ in that divisor, and the handle as the transverse branch,
		restriction of the normal bundle of the nodal union to $T_i$ is the positive
		elementary transform of this split bundle in the direction of the handle.  The handle direction has nonzero
		projection to the last trivial summand, since the handle is transverse to
		the fiber.  Lemma~\ref{lem:alg-tooth-transform} therefore gives
		\[
		N_C\otimes\mathcal O_{T_i}
		\cong N_{T_i/\widetilde{\mathcal X}_{b_i}}
		\oplus\mathcal O_{\Pone}(1),
		\]
		which is globally generated and has vanishing first cohomology.
		
		The handle maps isomorphically to $B$ and every tooth is vertical, so
		$q_*[C]=[B]$.  Let
		\[
		\mathcal L=q^*\mathcal O_B(-D).
		\]
		The restriction of $\mathcal L$ to the handle is $\mathcal O_B(-D)$;
		the restriction of $\mathcal L$ to
		each vertical tooth is trivial.  Consequently both $N_C$ and
		$N_C\otimes\mathcal L$ restrict to globally generated bundles with
		vanishing first cohomology on every irreducible component.  By
		Lemma~\ref{lem:alg-tree-bundle},
		\[
		\begin{gathered}
			N_C\text{ is globally generated},\qquad H^1(C,N_C)=0,\\
			N_C\otimes\mathcal L\text{ is globally generated},\qquad
			H^1(C,N_C\otimes\mathcal L)=0.
		\end{gathered}
		\]
		
		Apply Lemma~\ref{lem:alg-normal-smoothing} with
		$Y=\widetilde{\mathcal X}$ and $M=\mathcal O_B(-D)$.  This gives a nearby
		smooth curve which is the image of a section
		\[
		\widetilde h=(h,\ell):B\longrightarrow
		\mathcal X\times\mathbb P^3
		\]
		and satisfies
		\[
		H^1\bigl(B,\widetilde h^*T_{\widetilde{\mathcal X}/B}(-D)\bigr)=0.
		\]
		Since
		\[
		\widetilde h^*T_{\widetilde{\mathcal X}/B}(-D)
		\cong h^*T_{\mathcal X/B}(-D)
		\oplus\ell^*T_{\mathbb P^3}(-D),
		\]
		the first cohomology of each direct summand vanishes.  In particular,
		\[
		H^1\bigl(B,h^*T_{\mathcal X/B}(-D)\bigr)=0,
		\]
		as required.  Every tooth used above lies over one of the chosen points
		$b_i\in O$.
	\end{proof}
	
	The localization conclusion concerns the attachment fibers; the smoothed
	section need not equal $h_0$ outside $O$.
	
	\section{Compact approximation and algebraic liftings}
	\label{sec:compact-algebraic-proof}
	
	\subsection{Approximation on a compact source}
	
	The compact approximation argument transports the equation by a smooth bundle
	isomorphism $J$, represented locally by $G$.  This introduces the linear
	coefficient $L=G^{-1}\dbar G$ in addition to the residual $e$.  The disk
	construction must make the residual small in $L^4$ and the multiplication
	operator defined by the linear coefficient small from $W^{1,4}$ to $L^4$ before
	Lemma~\ref{lem:fixed-bundle-solution} applies.  Near the prescribed data,
	holomorphy of the section makes $e=0$, and holomorphy of $J$ makes $L=0$;
	formula~\eqref{eq:alg-constant-linear-defects} records these two roles.
	
	\begin{theorem}
		\label{thm:alg-compact-global}
		Let $\Sigma$ be a compact Riemann surface, let $X$ be a complex manifold,
		and let $p:X\to\Sigma$ be a proper holomorphic map.  Let $\Omega\subset\Sigma$ be open such that
		$p^{-1}(\Omega)\to\Omega$ is a holomorphic submersion with the relative
		first-jet sphere-family property.  Let $s:\Sigma\to X$ be a smooth section whose image lies
		in the submersion locus of $p$.  Choose a vertical local addition near
		$s(\Sigma)$, and let $E_{D_s}$ be $s^*T_{X/\Sigma}$ with the
		holomorphic structure given by $D_s=d\mathscr S_s(0)$ in
		Lemma~\ref{lem:rel-normal-form}.  Let $K\subset\Sigma$ be compact, let
		$A\subset\Sigma$ be finite, and prescribe an integer $k_a\geq0$ for each $a\in A$.  Put
		\[
		D=\sum_{a\in A}(k_a+1)a.
		\]
		Assume that there is an open set $U_h\subset\Sigma$ such that
		\[
		K\cup A\cup(\Sigma\setminus\Omega)\subset U_h,
		\]
		$s$ is holomorphic on $U_h$, and there are a holomorphic vector bundle
		$V\to\Sigma$ with $H^1(\Sigma,V)=0$ and a smooth complex bundle isomorphism
		\[
		J:V\longrightarrow E_D:=E_{D_s}(-D)
		\]
		which is holomorphic on $U_h$.  Then, for every $\epsilon>0$, there is a
		holomorphic section $F:\Sigma\to X$ of $p$ such that
		\[
		\sup_Kd_X(F,s)<\epsilon,
		\qquad
		j_a^{k_a}F=j_a^{k_a}s\quad(a\in A).
		\]
	\end{theorem}
	
	\begin{proof}
		Put $n=\operatorname{rank}V=\operatorname{rank}E_D$.  Fix smooth metrics
		and an area form on $\Sigma$, and a bounded right
		inverse $T$ for $\dbar_V$, using $H^1(\Sigma,V)=0$ and
		Lemma~\ref{lem:fixed-bundle-solution}.  These choices precede the disk
		construction.  Let $C_S$ denote the Sobolev embedding constant for
		$W^{1,4}(\Sigma,V)\hookrightarrow C^0(\Sigma,V)$.
		We divide the proof into six steps.
		
		\medskip\noindent
		\emph{Step 1: the equation on the fixed bundle with vanishing first cohomology.}
		The local addition in the statement can be obtained by the finite-chart
		argument of Lemma~\ref{lem:rel-addition}, applied near the compact image
		$s(\Sigma)$ in the submersion locus.  Denote this local addition by $\Psi$.
		Its domain contains a fixed
		neighborhood of the zero section in $s^*T_{X/\Sigma}$.
		Applying Lemma~\ref{lem:rel-twist} with jet divisor $D$ gives the
		twisted section equation $\mathscr S_D$ on a fixed fiber neighborhood in $E_D$.
		
		Pull $\mathscr S_D$ back through $J$:
		\[
		\mathscr T(v)=J^{-1}\mathscr S_D(Jv).
		\]
		In a holomorphic frame of $V$, $\mathscr T$ has the form
		\begin{equation}
			\mathscr T(v)=\dbar_Vv+\mathcal A(z,v),
			\label{eq:alg-gauged-equation}
		\end{equation}
		where $\mathcal A$ is smooth in $z$, holomorphic in $v$, and contains no derivative
		of $v$.  Put
		\[
		e=\mathcal A(\,\cdot\,,0),
		\qquad
		L=d_v\mathcal A(\,\cdot\,,0).
		\]
		In holomorphic frames of $V$ and $E_D$, if $J$ is represented by
		$G(z)\in\operatorname{GL}_n(\C)$ and the coefficient of $\mathscr S_D$ is
		$A_D$, then
		\[
		\mathcal A(z,q)=G^{-1}(\dbar G)q+G^{-1}A_D(z,Gq).
		\]
		Since $(A_D)_q(z,0)=0$, evaluating at $q=0$ and differentiating in $q$ give
		\begin{equation}
			e=G^{-1}A_D(z,0),
			\qquad
			L=G^{-1}\dbar G.
			\label{eq:alg-constant-linear-defects}
		\end{equation}
		On $U_h$, holomorphy of $s$ gives $A_D(z,0)=0$ and hence
		$e=0$; holomorphy of $J$ gives $\dbar G=0$ and hence $L=0$.  Hence
		\[
		U=\{z\in\Sigma:|e(z)|+|L(z)|>0\}
		\]
		satisfies
		\[
		\overline U\subset\Sigma\setminus U_h
		\Subset\Omega\setminus\operatorname{supp}D.
		\]
		In particular, $\overline U$ is compact and avoids the jet divisor.  If $U=\varnothing$, then
		$e=L=0$ and $\mathscr T(0)=0$.  The original section $s$ is already
		holomorphic, so $F=s$ proves the theorem.  We henceforth assume that
		$U\ne\varnothing$.
		
		\medskip\noindent
		\emph{Step 2: local families for rational disks.}
		Cover $\overline U$ by finitely many source coordinates, relative target
		charts, and holomorphic frames
		$\tau_z:\C^n\to V_z$.  In such a frame, put
		\begin{equation}
			\Phi(z,q)=\Psi\bigl(s(z),\iota_zJ_z\tau_zq\bigr),
			\label{eq:alg-physical-family}
		\end{equation}
		where $\iota:E_D\to E_{D_s}$ is the divisor morphism.  For $z\in\overline U$, the maps $\iota_z$, $J_z$, and $\tau_z$ are
		isomorphisms, and
		\[
		d_q\Phi(z,0)=d_u\Psi(s(z),0)\circ\iota_zJ_z\tau_z
		=\iota_zJ_z\tau_z.
		\]
		Thus $d_q\Phi(z,0)$ is an isomorphism.  Lemma~\ref{lem:rel-uniform-matching}
		and its finite-local-family formulation therefore provide finitely many
		source neighborhoods, relative target charts, parameter balls, and prepared
		and matched families with common constants.  Each source neighborhood
		retains its chosen holomorphic frame of $V$, which will be used for every
		disk selected there.
		
		We verify the outer-collar coefficient.  Let $\phi$ be the vertical
		coordinate map of a selected relative target chart, and write
		$\widehat\Phi=\phi\circ\Phi$.  For a smooth map $q=q(z)$ with values in the chosen parameter ball,
		\[
		(\partial_q\widehat\Phi)^{-1}
		\dbar[\widehat\Phi(z,q(z))]
		=\dbar q+C^\Phi(z,q),
		\]
		where
		\[
		C^\Phi=(\partial_q\widehat\Phi)^{-1}
		\partial_{\bar z}\widehat\Phi\,d\bar z.
		\]
		By the definitions of the normalized equation, the divisor twist, and the
		gauge, the left side equals
		$\tau^{-1}J^{-1}\mathscr S_D(J\tau q)$, the coordinate expression in
		the frame $\tau$.  Comparing with
		\eqref{eq:alg-gauged-equation} gives
		\begin{equation}
			C^\Phi=\mathcal A
			\label{eq:alg-outer-coefficient}
		\end{equation}
		in the chosen frame.  Thus the coefficient of the matched family agrees with the
		gauged global coefficient on the outer collar of each selected disk and is exactly zero on the
		inner collar of that disk.
		
		\medskip\noindent
		\emph{Step 3: a finite Vitali selection.}
		Choose a finite smooth triangulation sufficiently fine that every closed triangle
		meeting $\overline U$ is contained in one of the smaller source
		neighborhoods chosen in Step~2.  Enumerate the closed triangles meeting $U$
		as $T_1,\ldots,T_M$.  Each $T_j$ inherits the source coordinate, frame, and
		local families from its assigned neighborhood.
		
		Before choosing $r$, fix a common fiber neighborhood for the equation.
		For the finitely many frames $\tau$, choose $C_F\geq1$ with
		$|\tau_z^{-1}v|\leq C_F|v|$ on their compact source sets. Let $d_*>0$
		be smaller than all parameter radii fixed in Step~2. Choose $\rho>0$
		so that $C_F\rho<d_*$ and the ball bundle
		$V(\rho)=\{v\in V:|v|<\rho\}$ lies in the domain of $\mathcal A$.
		These choices depend only on the fixed local data, before the scale $r$
		and the disks are selected.
		
		Given $0<r<\min\{1/16,\sigma_*\}$, where $\sigma_*$ is the common
		threshold from Lemma~\ref{lem:rel-uniform-matching}, use in each $T_j$
		the candidate disks $\Delta^+_{c,\sigma}$ from \eqref{eq:rel-disk-geometry} with
		\[
		c\in U\cap\operatorname{int}T_j,
		\qquad 0<\sigma\leq r,
		\qquad \overline{\Delta^+_{c,\sigma}}
		\subset U\cap\operatorname{int}T_j.
		\]
		The fixed smooth area form defines a Radon measure.  Apply the Vitali
		covering theorem for this measure as in Step~5 of
		Theorem~\ref{thm:rel-bordered}; the triangle edges have area zero.
		In each $T_j$, truncate the countable family so that the uncovered area in
		$U\cap\operatorname{int}T_j$ is at most $r^{1/2}/M$.
		The resulting finite family of pairwise disjoint closed disks satisfies
		\begin{equation}
			\operatorname{area}\left(
			U\setminus\bigcup_i\Delta^+_{c_i,\sigma_i}\right)
			\leq r^{1/2}.
			\label{eq:alg-uncovered}
		\end{equation}
		Put $h_i=\sigma_i^{3/2}$ and let $\mathcal T_r$ be the union of the
		transition annuli.  As in \eqref{eq:rel-area}, disjointness gives
		\begin{equation}
			\sum_i\sigma_i^2\leq C,
			\qquad
			\operatorname{area}(\mathcal T_r)\leq Cr^{1/2}.
			\label{eq:alg-annular-area}
		\end{equation}
		With $\zeta_i$ the centered source coordinate assigned to disk $i$, put
		\[
		\begin{aligned}
			\mathcal I_i&=\{z\in\Delta^+_{c_i,\sigma_i}:|\zeta_i|<\sigma_i-h_i/2\},\\
			\mathcal O_r&=\Sigma\setminus\bigcup_i
			\{z\in\Delta^+_{c_i,\sigma_i}:|\zeta_i|\leq\sigma_i+h_i/2\}.
		\end{aligned}
		\]
		The regions $\mathcal O_r,\mathcal I_i,\mathcal T_{c_i,\sigma_i}$ form
		an open cover of $\Sigma$.
		
		\medskip\noindent
		\emph{Step 4: the modified operator and the two small-area estimates.}
		On the fixed ball bundle $V(\rho)$ chosen in Step~3, let $\mathcal A_r$ be the coefficient obtained by using
		$\mathcal A$ over $\mathcal O_r$, the zero coefficient over every $\mathcal I_i$,
		and the coefficient of the matched family over the corresponding transition
		annulus.  Equation~\eqref{eq:alg-outer-coefficient} and the open-collar
		properties make $\mathcal A_r$ a smooth bundle map.  On overlaps of holomorphic frames,
		the zero-order coefficients transform tensorially, so the formulas
		represent one coefficient on $V(\rho)$.  The uniform estimates of
		Lemma~\ref{lem:rel-uniform-matching}
		give
		\begin{equation}
			\sup_{V(\rho)}
			\bigl(|\mathcal A_r|+|d_v\mathcal A_r|+|d_v^2\mathcal A_r|\bigr)\leq C,
			\label{eq:alg-fiber-bounds}
		\end{equation}
		with $C$ independent of the number and radii of the disks.  Put
		\[
		\mathbb X=W^{1,4}(\Sigma,V),
		\qquad
		\mathbb Y=L^4_{0,1}(\Sigma,V),
		\qquad
		\mathscr T_r(v)=\dbar_Vv+\mathcal A_r(z,v(z)),\qquad
		\|v\|_{C^0}<\rho.
		\]
		
		Put
		\[
		e_r=\mathcal A_r(\,\cdot\,,0),
		\qquad
		L_r=d_v\mathcal A_r(\,\cdot\,,0)
		\]
		and
		\[
		\mathcal E_r=\left(U\setminus\bigcup_i\Delta^+_{c_i,\sigma_i}\right)
		\cup\mathcal T_r.
		\]
		Both $e_r$ and $L_r$ vanish outside the measurable set $\mathcal E_r$.
		Indeed, the entire coefficient is zero on each inner region; on the
		unmodified outer region $\mathcal A_r=\mathcal A$ has constant and linear terms $e$ and $L$, which
		vanish off $U$; and on the annuli the two terms are uniformly bounded by
		\eqref{eq:alg-fiber-bounds}.  Equations
		\eqref{eq:alg-uncovered}--\eqref{eq:alg-annular-area} give
		\[
		\operatorname{area}(\mathcal E_r)\leq Cr^{1/2}.
		\]
		Consequently
		\begin{equation}
			\|\mathscr T_r(0)\|_{\mathbb Y}
			=\|e_r\|_{L^4}
			\leq Cr^{1/8}.
			\label{eq:alg-residual}
		\end{equation}
		Moreover, for every $\xi\in\mathbb X$,
		\[
		\begin{aligned}
			\|(d\mathscr T_r(0)-\dbar_V)\xi\|_{L^4}
			&=\|L_r\xi\|_{L^4(\mathcal E_r)}\\
			&\leq C\operatorname{area}(\mathcal E_r)^{1/4}\|\xi\|_{C^0}\\
			&\leq Cr^{1/8}\|\xi\|_{W^{1,4}},
		\end{aligned}
		\]
		where the last step uses $W^{1,4}\hookrightarrow C^0$.  Hence
		\begin{equation}
			\|d\mathscr T_r(0)-\dbar_V\|_{\mathbb X\to\mathbb Y}
			\leq Cr^{1/8}.
			\label{eq:alg-linear-defect}
		\end{equation}
		The linear gauge defect need not be pointwise small.  After the disk
		replacements, both $e_r$ and $L_r$ vanish outside $\mathcal E_r$,
		whose area is $O(r^{1/2})$.  The resulting $O(r^{1/8})$ estimates
		control the residual in $L^4$ and the linear perturbation from
		$W^{1,4}$ to $L^4$.  In the bordered construction, the residual can be
		nonzero only on the uncovered defect set and the transition annuli,
		whereas the linear perturbation vanishes outside the transition annuli of the rational disks.
		These estimates concern the measurable nonzero loci, not their closed supports.
		No small-area bound is asserted for the set on which the full nonlinear
		coefficient $\mathcal A_r(z,\cdot)$ differs from $\mathcal A(z,\cdot)$.
		
		\medskip\noindent
		\emph{Step 5: solution on the fixed bundle.}
		Apply Lemma~\ref{lem:fixed-bundle-solution} to $V$ and $\mathcal A_r$.
		The cohomology vanishing is a hypothesis, the uniform fiber bounds are
		\eqref{eq:alg-fiber-bounds}, and the two smallness assumptions follow from
		\eqref{eq:alg-residual} and \eqref{eq:alg-linear-defect}, with
		$\varepsilon_r=C_1r^{1/8}$ for a fixed sufficiently large $C_1$.
		Thus, for small $r$, there is a smooth solution $v_r$ satisfying
		\begin{equation}
			\mathscr T_r(v_r)=0,
			\qquad \|v_r\|_{W^{1,4}}\leq Cr^{1/8}.
			\label{eq:alg-solution-bound}
		\end{equation}
		For every assigned frame,
		\[
		|\tau_z^{-1}v_r(z)|\leq C_FC_S\|v_r\|_{W^{1,4}}
		\leq Cr^{1/8}<d_*
		\]
		for small $r$.  Thus each coordinate parameter belongs to its assigned
		ball.  Smoothness uses the open-collar
		construction for each fixed $r$, as in the solution lemma.
		
		\medskip\noindent
		\emph{Step 6: global target-valued reconstruction.}
		Let $q_i=\tau_i^{-1}v_r$, where $\tau_i$ is the holomorphic frame assigned
		to the $i$th selected disk.  On all of $\Sigma$, put
		\[
		F_r(z)=
		\begin{cases}
			H_{c_i,\sigma_i}(z,q_i(z)),&z\in\mathcal I_i,\\
			M_{c_i,\sigma_i}(z,q_i(z)),&z\in\mathcal T_{c_i,\sigma_i},\\
			\Psi\bigl(s(z),\iota_zJ_zv_r(z)\bigr),&z\in\mathcal O_r.
		\end{cases}
		\]
		On each inner collar, the cutoff $\eta_{c_i,\sigma_i}=1$ gives
		$M_{c_i,\sigma_i}=H_{c_i,\sigma_i}$; on each outer collar,
		$\eta_{c_i,\sigma_i}=0$ gives $M_{c_i,\sigma_i}=\Phi$.
		Thus the formulas glue smoothly.  Each branch projects to $z$ under $p$.
		On $\mathcal I_i$, the equation is $\dbar q_i=0$, so
		\[
		\dbar F_r=d_qH_{c_i,\sigma_i}(z,q_i(z))\,\dbar q_i=0.
		\]
		On $\mathcal T_{c_i,\sigma_i}$, equation~\eqref{eq:rel-factorization} gives
		\[
		\dbar(\phi\circ F_r)=\partial_q\Xi_{c_i,\sigma_i}
		\bigl(\dbar q_i+C^M_{c_i,\sigma_i}(z,q_i)\bigr)=0.
		\]
		On $\mathcal O_r$, use the identity
		\[
		\mathscr S_s(\iota Jv_r)=\iota\,\mathscr S_D(Jv_r)
		\]
		which extends smoothly across $\operatorname{supp}D$ by
		Lemma~\ref{lem:rel-twist}.  Since $\mathscr T_r=\mathscr T$ on
		$\mathcal O_r$, the outer branch satisfies
		\[
		\dbar F_r=d_u\Psi(s,\iota Jv_r)\,\iota J\,\mathscr T_r(v_r)=0.
		\]
		The equality holds also at points of $D$, by the smooth extension.
		These open regions cover $\Sigma$, so $F_r$ is a global
		holomorphic section of $p$.
		
		Every selected disk lies in $U$ and is disjoint from $U_h$, which contains
		$K\cup A$.  The approximation estimate uses only $U_h$.
		On a neighborhood of $K\cup A$,
		\[
		F_r=\Psi(s,\iota Jv_r).
		\]
		Let $C_\Psi$ bound the fiber derivative of $\Psi$ on the fixed neighborhood
		of the zero section used over $K$.  The bounds for $\iota$ and $J$ over
		a compact neighborhood of $K$, together with Sobolev embedding and
		\eqref{eq:alg-solution-bound}, give
		\[
		\sup_Kd_X(F_r,s)\leq C_\Psi\sup_K\|\iota_zJ_z\|\,\|v_r\|_{C^0}
		\leq Cr^{1/8}.
		\]
		Near $a\in A$, choose a source coordinate $z$ centered at $a$ and compatible
		holomorphic frames.  Then $\iota$ is multiplication by $z^{k_a+1}$ and $J$
		is holomorphic.  Boundedness of $Jv_r$ gives
		$\iota Jv_r=O(|z|^{k_a+1})$.  For a common target chart $\phi$,
		\[
		h(z)=\phi(F_r(z))-\phi(s(z))=O(|z|^{k_a+1}).
		\]
		The function $h$ is holomorphic, so $h(z)/z^{k_a+1}$ extends
		holomorphically across $0$.  Therefore $j_a^{k_a}F_r=j_a^{k_a}s$.
		Choose $r$ with $Cr^{1/8}<\epsilon$ to obtain the required approximation.
	\end{proof}
	
	\subsection{Completion of the algebraic proof}
	
	To apply the compact approximation theorem to an affine source, we first
	complete the family without changing the morphism over that source.
	
	\begin{lemma}
		\label{lem:alg-projective-model}
		Let $p_R:X_R\to R$ be a projective morphism from a smooth variety to a
		smooth affine curve.  There are a smooth projective completion
		$R\subset B$ and a projective morphism $p:\mathcal X\to B$ with
		$\mathcal X$ smooth projective such that $\mathcal X|_R\cong X_R$ over $R$.
	\end{lemma}
	
	\begin{proof}
		Choose a smooth projective completion $R\subset B$.  Embed $X_R$ as a
		closed subvariety of $R\times\mathbb P^N$ and let $\overline X_R$ be its
		Zariski closure in $B\times\mathbb P^N$.  Since $X_R$ is closed over $R$,
		the restriction of $\overline X_R$ to $R$ is exactly $X_R$.  Take a projective resolution $r:\mathcal X\to\overline X_R$ which is
		an isomorphism over
		the already smooth open subset $X_R$, using resolution of singularities in
		characteristic zero~\cite{Hironaka64}.  The morphism $p=\operatorname{pr}_B\circ r:\mathcal X\to B$
		has the required properties; no dominance assumption on $p_R$ is
		needed.
	\end{proof}
	
	\begin{proof}[Proof of Theorem~\ref{thm:relative-algebraic-oka1}]
		The analytification $\pi^{\mathrm{an}}$ is a proper smooth submersion and
		hence a Serre fibration by Ehresmann's theorem \cite{Ehr51}.  It remains to
		prove the stronger lifting assertion in the theorem.  Work with the
		algebraic pullback
		\[
		p_R:X_R=R\times_YZ\longrightarrow R.
		\]
		Let $s_f(z)=(z,f(z))$ be the continuous section of $p_R^{\mathrm{an}}$
		associated with the original lifting $f$.
		By Lemma~\ref{lem:alg-projective-model}, choose a smooth projective
		completion $R\subset B$ and a projective model $p:\mathcal X\to B$ with
		smooth total space which equals $X_R\to R$ over $R$.  The morphism $p$ may
		have bad fibers over $B\setminus R$.
		
		We verify the hypothesis on the geometric generic fiber in
		Proposition~\ref{prop:alg-localized-free-section} before applying it.
		Put $\bar\eta=\operatorname{Spec}\overline{\C(R)}$, where
		$\overline{\C(R)}$ is an algebraic closure of the function field of $R$,
		and write $X_{\bar\eta}=X_R\times_R\bar\eta$.
		This fiber is smooth and projective by base change.
		Since the complex fibers are connected,
		$h^0((X_R)_b,\mathcal O_{(X_R)_b})=1$ for every $b\in R(\C)$.
		Upper semicontinuity and compatibility of fiber cohomology with field
		extension give $h^0(X_{\bar\eta},\mathcal O_{X_{\bar\eta}})=1$
		\cite[Tag~0BDN]{Stacks}; nonemptiness and constant functions give the
		lower bound.  Thus $X_{\bar\eta}$ is connected and, being smooth over
		an algebraically closed field, integral.
		
		In relative dimension zero, $X_{\bar\eta}$ is a point and is rationally
		connected.  In positive relative dimension, fix $b\in R(\C)$.
		Holomorphic maps from $\Pone$ to the projective fiber $(X_R)_b$ are
		algebraic \cite[Proposition~15]{Serre56}, so its rational connectedness
		agrees with the algebraic definition.
		Definition--Theorem~2.1(3) of \cite{Kol00} gives a morphism
		$u:\Pone\to(X_R)_b$ with $u^*T_{(X_R)_b}$ ample.
		It is nonconstant, hence very free in the sense of
		\cite[Definition~15]{HT06}; the fiber is therefore separably rationally
		connected.  The restriction of $p_R$ to
		$\operatorname{Spec}\mathcal O_{R,b}$ is a smooth proper local model
		with this special fiber.  Thus $b$ is a place of very good reduction
		in the sense of \cite[Remark~28]{HT06}, which implies that the generic
		fiber is geometrically separably rationally connected over $\C(R)$.
		In particular, $X_{\bar\eta}$ is rationally connected.
		Since $\mathcal X|_R\cong X_R$, this is also the geometric generic
		fiber of $p:\mathcal X\to B$, as required.
		
		Put
		\[
		D=\sum_{a\in A}(k_a+1)a.
		\]
		Since $R$ is noncompact and $K\cup A$ is compact, choose a nonempty
		Euclidean open set $O\Subset R\setminus(K\cup A)$.  Proposition~\ref{prop:alg-localized-free-section}
		gives an algebraic section $h:B\to\mathcal X$ with
		\[
		V=h^*T_{\mathcal X/B}(-D),
		\qquad H^1(B,V)=0.
		\]
		
		Over $R$, the map $p_R$ is a smooth proper bundle with connected simply
		connected fibers.  Lemma~\ref{lem:rel-section-homotopy} gives a homotopy
		$H$ through sections from $h|_R$ to $s_f$.  Choose
		a smooth function $\chi:R\to[0,1]$ which is one near $K\cup A$ and has
		compact support, and put
		\[
		s_0(z)=H(z,\chi(z)).
		\]
		The map $s_0$ is a continuous section of $p_R$, equal to $s_f$ near $K\cup A$ and to
		$h$ outside a compact subset of $R$.  Apply
		the relative smoothing statement in Subsection~\ref{subsec:bundle-tools} to the smooth bundle $p_R:X_R\to R$, relative
		to smaller neighborhoods of $K\cup A$ and of the complement of a large
		compact set.  We obtain a smooth section $s_R:R\to X_R$ equal to $s_f$ near $K\cup A$ and to $h$ near the ends of $R$, homotopic to $s_0$ through sections by a homotopy fixed near $K\cup A$ and near the ends of $R$.  Since $s_R=h$ near the ends of $R$, extending $s_R$ by $h$ gives a
		smooth section $s:B\to\mathcal X$.  The formula
		\[
		H^\chi(z,t)=H(z,t\chi(z))
		\]
		gives a homotopy from $h|_R$ to $s_0$ which is fixed outside a compact
		subset of $R$.  Concatenate $H^\chi$ with the homotopy supplied by relative smoothing from $s_0$ to
		$s_R$, using the first and second halves of the time interval.  Extend
		that concatenation by the constant section $h$ near $B\setminus R$.
		Denote the resulting homotopy from $h$ to $s$ by $H^B$.  The homotopy
		$H^B$ is fixed near $B\setminus R$.  Thus the cutoff produces the
		homotopy on the compact curve; the section-homotopy lemma is applied only
		over the open curve $R$.
		
		Reparametrize $H^B$ to be constant near $t=0,1$.  Regard the restriction of $H^B$
		to $R\times[0,1]$ as a section of the pullback of $p_R$ by
		$R\times[0,1]\to R$. The relative smoothing statement in Subsection~\ref{subsec:bundle-tools}, relative to the
		endpoints and the fixed region near $B\setminus R$, gives a smooth
		homotopy, still denoted $H^B$.  The image of $H^B$ lies in the submersion locus: over $R$ the family is smooth, and near $B\setminus R$ the homotopy equals the
		constant section $h$, along which
		$dp\circ dh=\operatorname{id}_{T_B}$ makes $dp$ surjective even over a bad
		fiber.  Fix a vertical local addition near $s(B)$ and let
		$D_s=d\mathscr S_s(0)$ be its linearized operator on $s^*T_{\mathcal X/B}$.
		Lemma~\ref{lem:alg-homotopy-transport}, followed by tensoring with
		$\mathcal O_B(-D)$, gives a smooth complex bundle isomorphism
		\[
		J_0:h^*T_{\mathcal X/B}(-D)
		\longrightarrow E_{D_s}(-D).
		\]
		$J_0$ supplies the global smooth identification and is the natural
		identity near $B\setminus R$.  Lemma~\ref{lem:alg-natural-Ds} shows that
		this identity is holomorphic there.
		Choose an open subsurface $W\Subset R$ containing $K\cup A$;
		this verifies the open-surface hypothesis of Lemma~\ref{lem:alg-protected-gauge}.
		Apply
		Lemma~\ref{lem:alg-protected-gauge} with
		\[
		P=K\cup A,
		\qquad Q=B\setminus R.
		\]
		We obtain a smooth isomorphism
		\[
		J:V\longrightarrow E_{D_s}(-D)
		\]
		which is holomorphic near $K\cup A$ and agrees with the identity near
		$B\setminus R$.
		
		Choose an open neighborhood $U_h$ of
		$K\cup A\cup(B\setminus R)$ with closure contained in the region where
		both $s$ and $J$ are holomorphic, and put $\Omega=R$.
		The map $p$ is proper, and $p|_R$ has the relative first-jet sphere-family
		property by Proposition~\ref{prop:rel-jet-family}. The image of $s$ lies in the submersion locus, $H^1(B,V)=0$,
		and $s,J$ are holomorphic on $U_h$.
		Thus Theorem~\ref{thm:alg-compact-global} applies.
		To transfer the metric estimate from $\mathcal X$ to $Z$, write $\operatorname{pr}_Z:\mathcal X|_R\cong R\times_YZ\to Z$.
		Choose a compact neighborhood $C_K$ of $s(K)$ in $\mathcal X|_R$.
		First choose $\delta_0>0$ so that the $\delta_0$-neighborhood of $s(K)$
		is contained in $C_K$.  Uniform continuity of $\operatorname{pr}_Z|_{C_K}$
		then gives $0<\delta<\delta_0$ such that
		\[
		z\in K,\quad p(x)=z,\quad d_{\mathcal X}(x,s(z))<\delta
		\quad\Longrightarrow\quad
		d_Z(\operatorname{pr}_Z(x),f(z))<\epsilon.
		\]
		Apply Theorem~\ref{thm:alg-compact-global} with error $\delta$ to obtain a holomorphic section
		\[
		\widetilde F:B\longrightarrow\mathcal X
		\]
		approximating $s$ in $d_{\mathcal X}$ on $K$ and having the prescribed jets
		at $A$.  Since $B$ and $\mathcal X$ are projective, the graph form of
		Chow's theorem gives algebraicity of $\widetilde F$
		\cite[Proposition~15]{Serre56}.
		Restricting to $R$ and projecting by
		$\operatorname{pr}_Z$ gives an algebraic lifting $F:R\to Z$ with the required
		$d_Z$ error.  The relation $s|_R=s_f$ holds near $A$, and holomorphic composition
		preserves jet equality.  Hence
		\[
		j_a^{k_a}(\operatorname{pr}_Z\circ\widetilde F)
		=j_a^{k_a}(\operatorname{pr}_Z\circ s_f)=j_a^{k_a}f
		\qquad(a\in A).
		\]
		
		Finally, the maps
		\[
		z\longmapsto(z,F^{\mathrm{an}}(z)),\qquad z\longmapsto(z,f(z))
		\]
		are sections of $X_R^{\mathrm{an}}\to R^{\mathrm{an}}$.
		The fibers are connected and simply connected, so
		Lemma~\ref{lem:rel-section-homotopy} gives a homotopy of these graph
		sections.  Projection to $Z^{\mathrm{an}}$ gives a homotopy from $f$ to
		$F^{\mathrm{an}}$ through continuous liftings of $g^{\mathrm{an}}$.
	\end{proof}
	\section*{AI Use Disclosure}
	The central ideas, mathematical insights, and overall approach of this
	work originated with the authors. AI tools were used in a supporting role
	to assist in manuscript preparation and revision. The authors take full
	responsibility for the paper's content and the accuracy of its references.

\end{document}